\documentclass[11pt]{article}
\usepackage{amsfonts}
\usepackage{latexsym,amssymb,amsfonts,graphicx}
\usepackage{amsmath,dsfont}
\usepackage[linewidth=1pt]{mdframed}
\usepackage{mathrsfs}
\usepackage[normalem]{ulem}
\usepackage{epsfig}
\usepackage{tabularx}
\usepackage{geometry}
\usepackage{enumitem}
\usepackage{framed}
\usepackage{algorithm2e}

\usepackage{natbib}
\setcitestyle{numbers,square}

\usepackage{url}
\usepackage{bm}
\usepackage{color}
\usepackage{natbib}
\usepackage{verbatim}
\usepackage[unicode=true,colorlinks,citecolor=linkcolor,linkcolor=blue,urlcolor=blue]{hyperref}

\providecommand{\U}[1]{\protect\rule{.1in}{.1in}}

\newtheorem{thm}{Theorem}[section]

\newtheorem{lem}{Lemma}[section]

\newtheorem{prop}{Proposition}[section]

\usepackage[bf,SL,BF]{subfigure}

\newenvironment{proof}[1][Proof]{\noindent\textbf{#1.} }{\ \rule{0.5em}{0.5em}}
\normalsize
\numberwithin{equation}{section}
\allowdisplaybreaks

\definecolor{linkcolor}{rgb}{0,0,0.7}

\usepackage[utf8]{inputenc}

\title{On De Finetti's control under Poisson observations: optimality of a double barrier strategy in a Markov additive model}
\author{Lijun Bo\thanks{School of Mathematics and Statistics, Xidian University, Xian, Shanxi, China. Email: alijunbo@xidian.edu.cn}\,\,\,\,\,\,\,\,Wenyuan Wang\thanks{School of Mathematical Sciences, Xiamen University, Xiamen, Fujian, China. Email: wwywang@xmu.edu.cn}\,\,\,\,\,\,\,\,
Kaixin Yan\thanks{School of Mathematical Sciences, Xiamen University, Xiamen, Fujian, China. Email: kaixinyan@stu.xmu.edu.cn}
}
\date{}

\begin{document}

\maketitle

\begin{abstract}
In this paper we consider the De Finetti's optimal dividend and capital injection problem under a Markov additive model. We assume that the surplus process before dividends and capital injections follows a spectrally positive Markov additive process. Dividend payments are made only at the jump times of an independent Poisson process.  Capitals are required to be injected whenever needed to ensure a non-negative surplus process to avoid bankruptcy. Our purpose is to characterize the optimal periodic dividend and capital injection strategy that maximizes the expected total discounted
dividends subtracted by the total discounted costs of capital injection. To this end, we first consider an auxiliary optimal periodic dividend and capital injection problem with final payoff under a single spectrally positive L\'evy process and conjecture that the optimal strategy is a double barrier strategy. Using the fluctuation theory and excursion-theoretical approach of the spectrally positive L\'evy process and the Hamilton-Jacobi-Bellman
inequality approach of the control theory, we are able to verify the conjecture that some double barrier periodic dividend and capital injection strategy solves the auxiliary problem. With the results for the auxiliary control problem and a fixed point argument for recursive
iterations induced by the dynamic programming principle, the optimality of a regime-modulated double barrier periodic dividend and capital injection strategy is proved for our target control problem.

\ \\

\end{abstract}

\section{Introduction}

The De Finetti's dividend problem amounts to a kind of stochastic optimal control problem that aims, relying on the information of the controlled and uncontrolled processes available so far, to
identify some adapted control process that maximizes the expected total discounted net dividends.
It was first studied in \cite{De Finetti57} under a symmetric random walk model, where a single barrier dividend strategy was proved to be the optimal control strategy which yields the maximum expected accumulated discounted net dividends until ruin.
Actually, paying dividends to investors is a common policy in the
economics of corporate finance; interested readers are referred to \cite{Feldstein83} for detailed explanations on why companies pay dividends. However, substantial dividend payments will inevitably increase the risk exposure of bankruptcy for the company. Hence, to protect the company, it is a common business to raise new capital, which is termed as capital
injection in the literature, by issuing new debts that usually appears in the form of loan from the bank; see \cite{Easterbrook84} for more details.

Following the pioneer work of \cite{De Finetti57} on De Finetti's dividend problem, research works along this line have been growing fast in insurance and corporate finance. To name a few, the literature has been witnessing many progresses when the
underlying surplus processes are Cram\'er-Lundberg processes, diffusion processes and L\'evy processes; see \cite{Albrecher11}, \cite{Azcue15}, \cite{Avanzi11}, \cite{Avanzi13}, \cite{Avram07}, \cite{Bayraktar13}, \cite{Jiang12}, \cite{Noba21}, \cite{Noba18}, \cite{Perez18}, \cite{Wang22.b}, \cite{Wang18}, \cite{Wei16}, \cite{Zhao17}, \cite{Zhu21}, \cite{Zhu20}, etc. However, to the limit of the authors' knowledge, there are significantly limited existing research works concerning De Finetti's stochastic control problem under the Markov additive models; the only three of them can be found in \cite{Mata22}, \cite{Noba20} and \cite{Wang22.c}.
Under the spectrally negative Markov additive process, \cite{Noba20} studied De Finetti's dividend and capital injection problem subject to the constraint that the accumulated dividend process is absolutely continuous with bounded density, and verified the optimality of a regime-modulated refraction-reflection strategy.
Then, under the spectrally positive Markov additive process, De Finetti's dividend and capital injection problem was investigated in \cite{Wang22.c}; while this time the accumulated dividend process can be any non-decreasing, right-continuous and adapted process. A regime-modulated double barrier dividend and capital injection strategy is proved to dominate any other strategy.
Very recently, \cite{Mata22} appeared to be the first and latest work that considered the Poisson observation version of De Finetti's dividend and capital injection problem under the spectrally negative Markov additive process, where the decision maker makes dividend decisions only at independent Poisson arrival times, while capital is injected into the surplus process continuously in time so that the controlled surplus process is always non-negative. By considering an auxiliary problem first and then using approximations via recursive iterations, the authors showed the optimality of a regime-modulated double barrier dividend and capital injection strategy.

Motivated by the work of \cite{Mata22},  in this paper, we raise a natural conjecture that the form of optimal strategy solving the Poisson observation version of De Finetti's dividend and capital injection problem remains a regime-modulated double barrier strategy when the uncontrolled spectrally negative Markov additive process of \cite{Mata22} is replaced with a spectrally positive Markov additive process, and aim to verify the conjecture. We first address an auxiliary optimal dividend and capital injection problem with a final payoff under a single spectrally positive L\'evy process, of which the optimal strategy is guessed to be some double barrier strategy. To confirm our guess, we need to derive the expression of the value function of a double barrier strategy for the auxiliary problem. This boils down to the derivation of solutions to the expected discounted total dividend payment minus the expected discounted total capital injection as well as the potential measure associated with the spectrally positive L\'evy process controlled by a double barrier dividend and capital injection strategy. During the derivation process, the cases that the underlying single spectrally positive L\'evy process is of bounded or unbounded variation should be considered separately, and the fluctuation theory and excursion-theoretical approach play important roles.
With the expression of the value function of a double barrier strategy in hand, we can characterize the optimal strategy of the auxiliary problem among the set of double barrier strategies, to which end we manage to express the derivative of the value function in terms of the Laplace transform and potential measure of the spectrally positive L\'evy process controlled by a single barrier periodic dividend strategy until the first time it down-crosses $0$. This new compact expression facilitates us to identify the candidate optimal strategy among the set of double barrier dividend and capital injection strategies and the slope conditions of the value function of this candidate optimal double barrier strategy.
Using these slope conditions and the Hamilton-Jacobi-Bellman inequality approach of the control theory, the above candidate optimal double barrier strategy is verified as the optimal strategy for the auxiliary problem.
Finally, thanks to the results for the auxiliary control problem and a fixed point argument
for recursive iterations induced by the dynamic programming principle, the optimality of a regime-modulated double barrier periodic dividend and capital injection strategy is proved for our target control problem.

It needs mentioning that the Markov additive model is a regime-switching model. Actually, the Markov additive process can be viewed as a family of L\'evy processes switching according to an independent continuous time Markov chain.
The regime-switching model has been widely studied due to its capability to capture the transitions of market behaviors and its mathematical tractability and explicit structures. A list of empirical studies concerning regime-switching can be found in \cite{Ang12}, \cite{Hamilton89}, \cite{So98}, and the references therein. Apart from \cite{Mata22}, \cite{Noba20} and \cite{Wang22.c}, other works on dividend control problem with regime switching can be found in \cite{Azcue15}, \cite{Jiang12}, \cite{Wei16}, \cite{Yang16} and \cite{Zhu21}, etc. We also mention that the feature of Poisson observation has been considered in the literature. To name a few, exit problems were studied under various models in \cite{Albrecher13a}, \cite{Albrecher16}, \cite{Albrecher13b}, \cite{Bekker09}, and the references therein. In the context of dividend control problems under Poisson observation, we refer to a short list of \cite{Albrecher11}, \cite{Avanzi13}, \cite{Zhao17} and \cite{Mata22}, etc.

Our approach on verifying the optimality of a double barrier strategy relies on the standard two-step approach for optimal dividend and capital injection control problems under the Markov additive models, i.e., consider first the control problem with a single regime using purely probabilistic approaches, and then solve the target control problem with regimes via approximation arguments (see; \cite{Mata22}, \cite{Noba20}, \cite{Wang22.c}). However, our analysis within each step differs substantially from that of \cite{Noba20} and \cite{Wang22.c} due to the complexity coming from the Poisson observation. Comparing with \cite{Mata22} where Poisson observation is present, we stress that we work with spectrally positive L\'evy processes. Due to the different path properties of spectrally positive L\'evy processes from that of spectrally negative L\'evy processes, distinct computations and proofs are required to handle the auxiliary control problem. For example, in addition to the fluctuation theory of the spectrally positive L\'evy process, we need also resort to the excursion-theoretical approach, which is not necessary in the analysis of \cite{Mata22}, to derive the expression for the value function of a double barrier dividend and capital injection strategy.
In addition, we have to separately consider the two cases whether or not the spectrally positive L\'evy process has bounded path variation, while \cite{Mata22} does not need to treat these two cases separately. Actually, arguments involving fluctuation theory and excursion-theoretical approach are used for the case of bounded variation, approximation and limiting arguments are further needed to handle the case of unbounded variation.
We also note that, the dividend barriers associated with the optimal strategy of the auxiliary control problem and the target control problem are strictly positive, while the counterparts of \cite{Mata22} can be $0$ in the bounded path variation case.

The rest of the paper unfolds as follows. In Section \ref{sec:model}, the optimal dividend and capital injection problem under the spectrally positive Markov additive process is formulated, some preliminaries of spectrally positive L\'{e}vy processes are also introduced. Section \ref{sec:aux} is devoted to study an auxiliary optimal dividend and capital injection problem with a final payoff at an exponential terminal time. Using fluctuation theory and excursion-theoretical approach of the spectrally positive L\'evy processes, and the Hamilton-Jacobi-Bellman inequality approach of the control theory, a double-barrier strategy is verified as the optimal strategy. In Section \ref{sec:opt}, based on the results obtained for the auxiliary control problem and a fixed point argument for recursive iterations induced by the dynamic programming principle, the optimality of a regime-modulated double barrier periodic dividend and capital injection strategy is proved for our target control problem.

\section{Problem Formulation under Spectrally Positive Markov Additive Processes}
\label{sec:model}

\subsection{Some preliminaries of spectrally positive L\'{e}vy processes}

Let $X=(X_t)_{t\geq 0}$ be a L\'{e}vy process defined on a probability space $(\Omega, \mathcal{F},\mathbb{F},\mathrm{P})$, where $\mathbb{F}=\{\mathcal{F}_t\}$ satisfies the usual conditions. For $x\in\mathbb{R}$, we denote by $\mathrm{P}_x$ the law of $X$ starting from $x$ and write $\mathrm{E}_x$ the associated expectation. We also use $\mathrm{P}$ and $\mathrm{E}$ in place of $\mathrm{P}_0$ and $\mathrm{E}_0$. The L\'{e}vy process $X$ is said to be spectrally positive if it has no negative jumps and it is not a subordinator. The Laplace exponent $\psi: [0,\infty)\rightarrow \mathbb{R}$ satisfying
\begin{align*}
\mathrm{E}[e^{-\theta X_t}]=:e^{\psi(\theta)t}, \quad t,\theta\geq 0,
\end{align*}
is given by the L\'{e}vy-Khintchine formula that
\begin{align*}
\psi(\theta):=c \theta+\frac{\sigma^2}{2}\theta^2+\int_{(0,\infty)}(e^{-\theta z}-1+\theta z\mathbf{1}_{\{z<1\}})\upsilon(dz), \quad \theta\geq 0,
\end{align*}
where $\gamma\in\mathbb{R}$, $\sigma\geq 0$, and $\upsilon$ is the L\'{e}vy measure of $X$ on $(0,\infty)$ that satisfies
\begin{align*}
\int_{(0,\infty)}(1\wedge z^2)\upsilon(dz)<\infty.
\end{align*}
It is well-known that $X$ has paths of bounded variation if and only if $\sigma=0$ and $\int_{(0,1)}z\upsilon(dz)<\infty$; in this case, we have
\begin{align*}
X_t=-ct+S_t,\quad t\geq 0,
\end{align*}
where
\begin{align*}
\overline{c}:=c+\int_{(0,1)}z\upsilon(dz),
\end{align*}
and $(S_t)_{t\geq 0}$ is a driftless subordinator. As we have ruled out the case that $X$ has monotone paths, it holds that $c>0$. Its Laplace exponent is given by
\begin{align*}
\psi(\theta)=c\theta+\int_{(0,\infty)}(e^{-\theta z}-1)\upsilon(dz),\quad \theta \geq 0.
\end{align*}
To exclude the trivial case, it is assumed throughout the paper that
\begin{align*}
\mathrm{E}[X_1]=-\psi'(0+)<\infty.
\end{align*}
Let us also recall the $q$-scale function for the spectrally positive L\'{e}vy process $X$. For $q>0$, there exists a continuous and increasing function $W_{q}:\mathbb{R}\rightarrow [0,\infty)$, called the $q$-scale function such that $W_{q}(x)=0$ for all $x<0$ and its Laplace transform on $[0,\infty)$ is given by
\begin{align*}
\int_0^{\infty} e^{-sx}W_{q}(x)dx=\frac{1}{\psi(s)-q},\quad s>\Phi_{q},
\end{align*}
where $\Phi_{q}:=\sup\{s\geq 0: \psi(s)=q\}$. We also define $Z_{q}(x)$ by
\begin{align*}
Z_{q}(x):=1+q\int_0^x W_{q}(y)dy,\quad x\in\mathbb{R},
\end{align*}
and its anti-derivative
\begin{align*}
\overline{Z}_{q}(x):=\int_0^x Z_{q}(y)dy,\quad x\in\mathbb{R}.
\end{align*}
The so-called second scale function is defined by, for $s\geq0$,
\begin{eqnarray}
Z_q(x,s)=e^{sx}\Big(1-(\psi(s)-q)\int_0^xe^{-sy}W_q(y)\mathrm{d}y\Big),\quad x\geq0,\nonumber
\end{eqnarray}
and $Z_q(x,s)=e^{sx}$ for $x<0$. Note $Z_q(x,s)=Z_q(x)$ for $s=0$ and that we can rewrite $Z_q(x,s)$ for $s>\Phi_q$ in the form
\begin{eqnarray}
Z_q(x,s)=(\psi(s)-q)\int_0^{\infty}e^{-sy}W_q(x+y)\mathrm{d}y,\quad x\geq0,\,s>\Phi_q.\nonumber
\end{eqnarray}
We recall that if $X$ has paths of bounded variation, $W_q(x)\in C^1((0,\infty))$ if and only if the L\'{e}vy measure $\upsilon$ has no atoms. If $X$ has paths of unbounded variation, we have that $W_q(x)\in C^1((0,\infty))$. Moreover, if $\sigma>0$, we have $W_q(x)\in C^2((0,\infty))$. Hence, we have that $Z_q(x)\in C^1((0,\infty))$, $\overline{Z}_q(x)\in C^1(\mathbb{R})$ and $\overline{Z}_q(x) \in C^2((0,\infty))$ for bounded variation case; and we have $Z_q(x)\in C^1(\mathbb{R})$, $\overline{Z}_q(x) \in C^2(\mathbb{R})$ and $\overline{Z}_q(x)\in C^3((0,\infty))$ for the unbounded variation case. We also know that
\begin{align*}
\begin{split}
W_{q} (0+) &= \left\{ \begin{array}{ll} 0 & \textrm{if $X$ is of unbounded
variation,} \\ 1/\overline{c} & \textrm{if $X$ is of bounded variation.}
\end{array} \right.
\end{split}
\end{align*}
Let us define $\tau_{a}^{-}:=\inf\{t\geq 0; X_{t}<a\}$ and $\tau_{b}^{+}:=\inf\{t\geq 0; X_{t}>b\}$. Then, for $b\in(0,\infty)$ and $x\in[0,b]$, we have
\begin{eqnarray}
\label{fluc.1}
\mathrm{E}_{x}\left(e^{-q \tau_{0}^{-}}\mathbf{1}_{\{\tau_{0}^{-}<\tau_{b}^{+}\}}\right)
\hspace{-0.3cm}&=&\hspace{-0.3cm}
\frac{W_{q}(b-x)}{W_{q}(b)},
\\
\label{fluc.2}
\mathrm{E}_{x}\left(e^{-q \tau_{b}^{+}}\mathbf{1}_{\{\tau_{b}^{+}<\tau_{0}^{-}\}}\right)
\hspace{-0.3cm}&=&\hspace{-0.3cm}
Z_{q}(b-x)-\frac{Z_{q}(b)}{W_{q}(b)}W_{q}(b-x).
\end{eqnarray}

In addition, for further use, we recall briefly some concepts of the excursion theory for the spectrally positive L\'evy process $X$ reflected from the infimum, i.e., $\{X(t)-\underline{X}(t);t\geq0\}$ with $\underline{X}(t)=\inf_{0\leq s\leq t}X_{s}$, and we refer to \cite{Bertoin96} for more details.
For $x\in(-\infty,\infty)$, the process $\{L(t):= -\underline{X}(t)+x, t\geq0\}$ represents a local time at $0$ of
the Markov process $\{X(t)-\underline{X}(t);t\geq0\}$ under $\mathbb{P}_{x}$.
Define the inverse local time as
$$L^{-1}(t):=\inf\{s\geq0: L(s)>t\}.$$
Let $L^{-1}(t-):=\lim\limits_{s\rightarrow t-}L^{-1}(s)$ be the left limit of $L^{-1}(s)$ at $s=t$.
When $L^{-1}(t)-L^{-1}(t-)>0$, define a Poisson point process $\{(t, e_{t}); t\geq0\}$ as follows
$$e_{t}(s):=X(L^{-1}(t-)+s)-X(L^{-1}(t)), \,\,s\in(0,L^{-1}(t)-L^{-1}(t-)],$$
where $L^{-1}(t)-L^{-1}(t-)$ is referred to as the lifetime of $e_{t}$.
While, when $L^{-1}(t)-L^{-1}(t-)=0 $, define $e_{t}:=\Upsilon$, where $\Upsilon$ is some additional isolated point.
A conclusion drawn by It\^{o} states that $e$ is a Poisson point process with
characteristic measure $n$
when the reflected process $\{X(t)-\underline{X}(t);t\geq0\}$ is recurrent; otherwise $\{e_{t}; t\leq L(\infty)\}$ is a Poisson point process which stops until an excursion with infinite lifetime takes place. Here, $n$ is a measure on the space  $\mathcal{E}$ of excursions,
that is, the space $\mathcal{E}$  of c\`{a}dl\`{a}g functions $f$ such that
\begin{eqnarray}
&&f:\,(0,\zeta)\rightarrow (0,\infty)\,\quad \mbox{for some } \zeta=\zeta(f)\in(0,\infty],
\nonumber\\
&&f:\,\{\zeta\}\rightarrow (0,\infty)\,\,\,\,\,\,\quad \mbox{if } \zeta<\infty,
\nonumber
\end{eqnarray}
where $\zeta=\zeta(f)$ represents the lifetime of the excursion; and we refer to Definition 6.13 of \cite{Kyprianou14} for more details of the space $\mathcal{E}$ of canonical excursions.
Denote by $\varepsilon(\cdot)$, or $\varepsilon$ for short, a generic excursion
in $\mathcal{E}$.
Denote by $\overline{\varepsilon}=\sup\limits_{t\in[0,\zeta]}\varepsilon(t)$ the excursion height of $\varepsilon$. And, denote
by
$$
\rho_{b}^{+}\equiv\rho_{b}^{+}(\varepsilon) :=\inf\{t\in[0,\zeta]: \varepsilon(t)>b\},
$$
the first passage time of $\varepsilon$ with the convention of $\inf\emptyset:=\zeta$.
In addition, let $\varepsilon_{g}$ be
the excursion (away from $0$)  with left-end point $g$ for the reflected process $\{X(t)-\underline{X}(t);t\geq0\}$, and let $\zeta_{g}$ and $\overline{\varepsilon}_{g}$ be, respectively, the lifetime and excursion height of $\varepsilon_{g}$; see Section IV.4 of \cite{Bertoin96}.

\subsection{Problem Formulation}
Let us consider the risk process modelled by the spectrally positive Markov additive process $\{(X_t,Y_t);t\geq0\}$. Here, $\{Y_t;t\geq 0\}$ is a continuous time Markov chain with finite state space $\mathcal{E}$ and the generator matrix $(\lambda_{ij})_{i,j\in\mathcal{E}}$. Condition on that Markov chain $Y$ is in the state $i$, the process $X$ evolves as a spectrally
positive L\'evy process $X^{i}$ until the Markov chain $Y$ switches to another state $j\neq i$, at which instant there is a downward jump in $X$ with a random amount $J_{ij}$. We assume that $(X^{i})_{i\in\mathcal{E}}$, $Y$, and $(J_{ij})_{i,j\in\mathcal{E}}$ are independent from each other and are defined on a filtered probability space $(\Omega,\mathcal{F},\{\mathcal{F}_{t};t\geq 0\},\mathrm{P})$ satisfying the usual condition. Denote by $\mathrm{P}_{x,i}$ the law of the process $\{(X_{t},Y_t);t\geq 0\}$ conditioning on $\{X_{0}=x,Y_{0}=i\}$.

We consider a bail-out dividend control problem in this Markov additive framework, where the
beneficiaries of dividends are supposed to injects capitals into the surplus process so that the resulting
surplus process are always non-negative, i.e., bankruptcy never occurs. We consider two non-decreasing, right-continuous, adapted processes $\{D_{t};t\geq 0\}$ and $\{R_{t};t\geq 0\}$ defined on $(\Omega,\mathcal{F},\{\mathcal{F}_{t};t\geq 0\},\mathrm{P})$, which, respectively, represent the cumulative amount of dividends and injected capitals with $D_{0-}=R_{0-}=0$. In this paper we consider that the dividend payments can only be made at the arrival epochs $(T_n)_{n\geq1}$ of an independent Poisson process $(N_t)_{t\geq0}$ with the intensity $\gamma>0$. Contrary to the dividend payments, capital injection can be made continuously in time. Hence, the surplus process after taking into account the dividends and capital injection is formulated as $U_t=X_t-D_t+R_t,\,t\geq0$. The value function of the periodic dividend control problem with capital injection is defined by
\begin{eqnarray}
\label{goal}
V(x,i)\hspace{-0.2cm}&=&\hspace{-0.2cm}\sup_{D,R}\mathrm{E}_{x,i}\left(\int_{0}^{\infty}e^{-\int_{0}^{t}\delta_{Y_s}\mathrm{d}s}\mathrm{d}D_{t}-\phi\int_{0}^{\infty}e^{-\int_{0}^{t}\delta_{Y_s}\mathrm{d}s}\mathrm{d}R_{t}\right)\notag\\
\hspace{-0.2cm}&&\hspace{-0.2cm}
\text{subject to\ \  $U_{t}=X_{t}-D_{t}+R_{t}\geq 0$ for all $t\geq 0$,}
\nonumber\\
\hspace{-0.2cm}&&\hspace{-0.2cm}
\text{both $D_{t}$ and $R_{t}$ are non-decreasing, c\`{a}dl\`{a}g and adapted processes,}
\nonumber\\
\hspace{-0.2cm}&&\hspace{-0.2cm}
D_t=\int_0^t\Delta D_s\mathrm{d}N_s,\text{ $t\geq0$, for a Poisson process $(N_t)_{t\geq0}$ with intensity $\gamma>0$, }
\nonumber\\
\hspace{-0.2cm}&&\hspace{-0.2cm}
\text{$D_{0-}=R_{0-}=0$, and $\int_{0}^{\infty}e^{-\int_{0}^{t}\delta_{Y_s}\mathrm{d}s}\mathrm{d}R_{t}<\infty$, $\mathrm{P}_{x,i}$-almost surely},
\end{eqnarray}
where $(\delta_{i}) \in [0,\infty)^{\mathcal{E}}$ is a discounting rate function that switches according to the economic environment depicted by the Markov chain $Y$, and $\phi>1$ is the cost per unit capital injected. Our goal is to identify the value functions $(V(x,i))_{i\in\mathcal{E}}$ and find the optimal strategy $(D^*,R^*)$ that attains the value functions $(V(x,i))_{i\in\mathcal{E}}$.

Following the similar proof of Proposition 3.1 in \cite{Mata22}, we can readily obtain the following dynamic programming principle for value function of the control problem holds valid, and its proof is hence omitted.

\begin{prop}
For $x\in\mathbb{R}$ and $i\in\mathcal{E}$, we have that
\begin{align}\label{dpp}
V(x,i)=\sup_{D,R}\mathrm{E}_{x,i}\left [ \int_{0}^{e_{\lambda_{i}}}e^{-\int_{0}^{t}\delta_{Y_s}\mathrm{d}s}\mathrm{d}D_{t}-\phi\int_{0}^{e_{\lambda_{i}}}e^{-\int_{0}^{t}\delta_{Y_s}\mathrm{d}s}\mathrm{d}R_{t}+e^{-\int_0^{e_{\lambda_{i}}}\delta_{Y_s}ds}V(U_{e_{\lambda_{i}}}, Y_{e_{\lambda_{i}}})\right],\nonumber
\end{align}
where $e_{\lambda_{i}}$ is the first time $Y$ switches the regime state under $\mathrm{P}_{x,i}$.
\end{prop}

The next theorem is the main result of this paper, which confirms the optimality of a regime-modulated periodic-classical reflection strategy for the stochastic control problem \eqref{goal}, whose proof is deferred to Section 4.

\begin{thm}
\label{thm2.1}
There exists a function $\mathbf{b}^{*}=(b_{i}^{*})_{i\in\mathcal{E}}\in(0,\infty)^{\mathcal{E}}$ such that the periodic dividend and capital injection strategy with the dynamic upper periodic barrier $b^{*}_{Y_{t}}$ and fixed lower reflection barrier $0$ is optimal that attains the value function in \eqref{goal} that
$$V_{0,\mathbf{b}^{*}}(x,i)=V(x,i),\quad (x,i)\in\mathbb{R}_{+}\times \mathcal{E},$$
where $V_{0,\mathbf{b}^{*}}(x,i)$ represents the value function of the periodic-classical barrier dividend and capital injection strategy with upper barrier $b^{*}_{Y_{t}}$ and lower barrier $0$.
\end{thm}

\section{Auxiliary Optimal Dividend Problem with Final Payoff}
\label{sec:aux}

In this section we consider an auxiliary optimal dividend and capital injection problem with final payoff.
We suppose that, without a control, the surplus process evolves as a spectrally positive L\'evy process $(X_t)_{t\geq0}$.
It is then assumed that dividend decisions on weather or not a lump sum of dividend payment is deducted from the surplus process can only be made at the arrival epochs $(T_{n})_{n\geq 1}$ of an independent Poisson process $(N_{t})_{t\geq0}$ with the intensity $\gamma>0$. In this case, the non-decreasing, cadl\'ag, pure jump cumulative dividend payment process  $(D_t)_{t\geq0}$ can be written as
\begin{eqnarray}
D_t=\int_0^{t}\Delta D_s\mathrm{d}N_s=\sum_{n=1}^{\infty}\Delta D_{T_n}\mathbf{1}_{\{T_n\leq t\}},\quad t\geq0,\nonumber
\end{eqnarray}
where $T_i$ be the arrival times of an independent Poisson process of intensity $\gamma$ and $\Delta D_{T_n}=D_{T_n}-D_{T_{n}-}$ represents the amount of dividends paid at time $T_n$. Regarding the cumulative capital injection process $(R_t)_{t\geq0}$, we assume that it is a non-decreasing and cadl\'ag process starting from $0$, i.e, $R_{0-}=0$.
Given a dividend and capital injection control strategy $\pi=(D^{\pi},R^{\pi})$, we can write the resulting surplus process under $\pi$, denoted by $U^{\pi}=(U_t^{\pi})_{t\geq0}$, as
\begin{eqnarray}
U_t^{\pi}=X_t-\sum_{n=1}^{\infty}\Delta D^{\pi}_{T_n}\mathbf{1}_{\{T_n\leq t\}}+R_t^{\pi},\quad t\geq0.\nonumber
\end{eqnarray}

Throughout this section, we assume that the payoff function $\omega$ is continuous and concave over $[0, \infty)$ with
$\omega^{\prime}_{+}(0+)\leq \phi$ and $\omega^{\prime}_{+}(\infty)\in[0,1]$, where $\omega_+^{\prime}(x)$  denotes the right derivative of $\omega$ at $x$.
A periodic dividend and capital injection strategy $\pi=(D^{\pi},R^{\pi})$ is said to be admissible if $U_t^{\pi}\geq0$ for all $t\geq0$.
Let's denote by $\Pi$ the set of all admissible periodic dividend and capital injection strategies.
For $\delta,\lambda>0$ and $\pi\in\Pi$, the value function of the expected terminal payoff added with the expected difference between the present value of dividends and costs of capital injection is written as
\begin{eqnarray}
\label{valefunctionofpi}
V_{\pi}^{\omega}(x)
\hspace{-0.3cm}&=&\hspace{-0.3cm}
\mathrm{E}_x\bigg[\int_0^{e_{\lambda}}e^{-\delta t}\mathrm{d}D_{t}^{\pi}-\phi\int_0^{e_{\lambda}}e^{-\delta t}\mathrm{d}R_t^{\pi}+e^{-\delta e_{\lambda}}\omega(U_{e_{\lambda}}^{\pi})\bigg]
\nonumber\\
\hspace{-0.3cm}&=&\hspace{-0.3cm}
\mathrm{E}_x\bigg[\int_0^{\infty}\lambda e^{-\lambda t}\bigg[\sum_{n=1}^{\infty}e^{-\delta T_n}\Delta D^{\pi}_{T_n}\mathbf{1}_{\{T_n\leq t\}}-\phi\int_0^{t}e^{-\delta s}\mathrm{d}R_s^{\pi}+e^{-\delta t}\omega(U_{t}^{\pi})\bigg]\mathrm{d}t\bigg]
\nonumber\\
\hspace{-0.3cm}&=&\hspace{-0.3cm}
\mathrm{E}_x\bigg[\sum_{n=1}^{\infty}e^{-q T_n}\Delta D^{\pi}_{T_n}-\phi\int_0^{\infty}e^{-qt}\mathrm{d}R_t^{\pi}+\lambda\int_0^{\infty}e^{-qt}\omega(U_{t}^{\pi})\mathrm{d}t\bigg],
\end{eqnarray}
where $q=\delta+\lambda$.
The purpose of the auxiliary stochastic control problem is to find the optimal dividend and capital injection strategy $\pi^{*}\in\Pi$ in the sense that it dominates all other admissible strategies. The value function of the auxiliary stochastic control problem is then given by
\begin{eqnarray}\label{au.problem}
V_{\pi*}^{\omega}(x)=\sup_{\pi} V_{\pi}^{\omega}(x)
\hspace{-0.3cm}&&\hspace{-0.3cm}\,\,\,
\text{subject to\ \  $U^{\pi}_{t}=X_{t}-D^{\pi}_{t}+R^{\pi}_{t}\geq 0$ for all $t\geq 0$,}
\nonumber\\
\hspace{-0.3cm}&&\hspace{-0.3cm}\,\,\,
\text{both $D^{\pi}_{t}$ and $R^{\pi}_{t}$ are non-decreasing, c\`{a}dl\`{a}g and adapted processes,}
\nonumber\\
\hspace{-0.2cm}&&\hspace{-0.2cm}
D^{\pi}_t=\int_0^t\Delta D^{\pi}_s\mathrm{d}N_s,\text{ $t\geq0$, for a Poisson process $(N_t)_{t\geq0}$ with intensity $\gamma>0$, }
\nonumber\\
\hspace{-0.3cm}&&\hspace{-0.3cm}\,\,\,
\text{$D^{\pi}_{0-}=R^{\pi}_{0-}=0$, and $\int_0^{\infty}e^{-q t}\mathrm{d}R^{\pi}_t<\infty$, $\mathrm{P}_{x}$-almost surely}.
\end{eqnarray}

We first consider a smaller subset of admissible dividend and capital injection strategies. We note that the time value of money (i.e., $q>0$), it seems reasonable to inject capitals as late as possible. In addition, since there are transaction costs charged for per unit of capitals injected ($\phi>1$), whenever capitals injection is required, the injected capital should be the amount to keep the surplus process non-negative, that is, the surplus process will reflect from below at 0. Therefore, by the above intuitive arguments, we give the following Lemma \ref{R.continu.}. The proof is essentially similar to that of Lemma 4.2 in \cite{Wang22.b} and is hence omitted.
\begin{lem}
\label{R.continu.}
The optimal dividend and capital injection process $\{(D_{t}, R_{t});t\geq 0\}$ for the optimization problem is represented as \eqref{au.problem} is such that $0\leq \Delta D_{t}\leq X_{t-}$ and
\begin{eqnarray}\label{R.form}
R_{t}=-\inf_{s\leq t}(X_{s}-D_{s})\wedge 0.
\end{eqnarray}
In particular, $\{R_{t};t\geq 0\}$ is continuous.
\end{lem}

We conjecture that the optimality of the control problem \eqref{au.problem} can be attained by a double barrier strategy, associated to which the two dimensional cumulative dividend and capital injection process is denoted as $(D_t^{0,b},R_t^{0,b})_{t\geq0}$ with $b\in(0,\infty)$. Under the strategy $(D_t^{0,b},R_t^{0,b})_{t\geq0}$, the surplus process is observed at each arrival epoch of an independent Poisson process, whenever the surplus is observed to be above $b$, the overshoot is paid out as dividends; while capitals are injected into the surplus process to push it upward to $0$ whenever it is about to down-cross $0$. After adopting the strategy $(D_t^{0,b},R_t^{0,b})_{t\geq0}$, the resulting surplus process reads as
$$U_{t}^{0,b}:=X_t-D_{t}^{0,b}+R_{t}^{0,b},\quad t\geq0,$$
where
$$D_t^{0,b}=\sum_{n=1}^{\infty}\Big((U_{T_{n}-}^{0,b}+\Delta X_{T_{n}}-b)\vee0\Big)\mathbf{1}_{\{T_{n}\leq t\}}
\quad \text{
and }\quad
R_t^{0,b}=-\inf_{0\leq s< t}\Big((X_s-D_s^{0,b})\wedge0\Big),\quad t\geq0.$$
In addition, denote by down-crossing and up-crossing times of $U_t^{0,b}$ respectively by
$$\kappa_{a}^{-}:=\inf\{t\geq 0; U^{0,b}_{t}<a\}\quad \text{and }\,\,\,\kappa_{b}^{+}:=\inf\{t\geq 0; U^{0,b}_{t}>b\}.$$ Let us denote the performance function (see, \eqref{valefunctionofpi}) for the periodic-classical barrier dividend and capital injection strategy $\pi=(D_t^{0,b},R_t^{0,b})_{t\geq0}$ as
$$V_{0,b}^{\omega}(x)=\mathrm{E}_x\bigg[\int_0^{\infty}e^{-qt}\mathrm{d}D^{0,b}_t-\phi\int_0^{\infty}e^{-qt}\mathrm{d}R_t^{0,b}+\lambda\int_0^{\infty}e^{-qt}\omega(U^{0,b}_{t})\mathrm{d}t\bigg].$$
For further computations, 
define 
$\widetilde{Y}_t=X_{t}-\underline{X}_{t}\wedge0$ with $\underline{X}_{t}=\inf_{0\leq s\leq t}X_{s}$, and $\sigma_{a}^{+}:=\inf\{t\geq0: \widetilde{Y}_t>a\}$ for $a>0$. 
Actually, $\widetilde{Y}$ can be interpreted as a surplus process with capital injection that prevents the company from going bankruptcy, where $-\underline{X}_{t}\wedge0$ represents the total amount of capitals injected during the time period $[0,t]$.
Furthermore, let
$T_b^+:=\min\{T_i:\widetilde{Y}_{T_i}>b\}$
be the first time the process $\widetilde{Y}_{t}$ is observed to be above $b$ under Poisson observation, i.e., the first time when a lump sum of dividend is paid out in the surplus process $U_t^{0,b}$.


Since it has been conjectured that our auxiliary dividend control problem \eqref{au.problem} is solved by a double barrier peoriodoc dividend and capital injection strategy, we need to find an expression for the value function of a double barrier strategy. We decompose the corresponding computations into the upcoming Lemmas \ref{lem.D.R}-\ref{lem.w}.
In particular, the following Lemma \ref{lem.D.R} gives an expression for the expected discounted total dividend payment minus the expected discounted total capital injection under a double barrier dividend and capital injection strategy.

\begin{lem}\label{lem.D.R}
For $q>0$, $b>0$, and $\phi>1$, we have
\begin{eqnarray}
\label{exp.dif.div.cap.}
\mathrm{E}_x\bigg[\int_0^{\infty}e^{-qt}\mathrm{d}D^{0,b}_t-\phi\int_0^{\infty}e^{-qt}\mathrm{d}R_t^{0,b}\bigg]
\hspace{-0.3cm}&=&\hspace{-0.3cm}
-\frac{\gamma}{q+\gamma}\left[\overline{Z}_q(b-x)+\frac{\psi^{\prime}(0+)}{q}\right]
\nonumber\\
\hspace{-0.3cm}&&\hspace{-3.5cm}
+
\frac{\left(\gamma Z_q(b)-\phi(q+\gamma)\right)\left[Z_q(b-x,\Phi_{q+\gamma})+\frac{\gamma}{q}Z_q(b-x)\right]}{(q+\gamma)\Phi_{q+\gamma}Z_q(b,\Phi_{q+\gamma})}
,\quad x\in(0,\infty).
\end{eqnarray}
\end{lem}
\begin{proof}
Denote by $f(x)$ the left hand side of \eqref{exp.dif.div.cap.}. We recall from Lemma 3.4 of \cite{Zhao17} that
\begin{eqnarray}
\mathrm{E}_x\left[e^{-qT_1}(X_{T_1}-b)\mathbf{1}_{\{T_1<\tau_b^-\}}\right]
\hspace{-0.3cm}&=&\hspace{-0.3cm}
\frac{\gamma}{q+\gamma}\left[x-b-\frac{\psi^{\prime}(0+)}{q+\gamma}\left(1-e^{-\Phi_{q+\gamma}(x-b)}\right)\right],\quad x\in(b,\infty),
\nonumber\\
\mathrm{E}_x\left[e^{-q(T_1\wedge\tau_b^-)}\right]
\hspace{-0.3cm}&=&\hspace{-0.3cm}
\frac{\gamma}{q+\gamma}+\frac{q}{q+\gamma}e^{-\Phi_{q+\gamma}(x-b)},\quad x\in(b,\infty),\label{T1wed.tau.b-}\nonumber
\end{eqnarray}
which together with the strong Markov property leads to
\begin{eqnarray}\label{f.b.1}
\hspace{-0.8cm}
f(x)
\hspace{-0.3cm}&=&\hspace{-0.3cm}
\mathrm{E}_x\Big[e^{-qT_1}(X_{T_1}-b)\mathbf{1}_{\{T_1<\tau_b^-\}}\Big]+\mathrm{E}_x\Big[e^{-q(T_1\wedge\tau_b^-)}\Big]f(b)
\nonumber\\
\hspace{-0.3cm}&=&\hspace{-0.3cm}
\frac{\gamma}{q+\gamma}\left[x-b-\frac{\psi^{\prime}(0+)}{q+\gamma}\left(1-e^{-\Phi_{q+\gamma}(x-b)}\right)\right]+\left[\frac{\gamma}{q+\gamma}+\frac{q}{q+\gamma}e^{-\Phi_{q+\gamma}(x-b)}\right]f(b),\,\, x\in(b,\infty).
\end{eqnarray}
It is not hard to verify that
$U_{t}^{0,b}=\widetilde{Y}_{t}$ and $
R_{t}^{0,b}=-\inf_{s\leq t}X_{s}\wedge0=\sup_{s\leq t}((b-X_{s})-b)\vee 0$ for $t\leq T_{b}^{+}$, where the term $\sup_{s\leq t}((b-X_{s})-b)\vee 0$ can be interpreted as the amount of dividends paid over the time interval $[0,t]$ under a barrier dividend strategy with barrier $b$, given that the surplus process free of dividends evolves as $b-X_{t}$;
and $T_{b}^{+}$ is identical to the first instant the surplus process $(b-X_{t})-\sup_{s\leq t}((b-X_{s})-b)\vee 0=b-\widetilde{Y}_{t}$ with dividends paid (out of $b-X_{t}$) according to the barrier dividend strategy with barrier $b$ is observed to be below $0$ under the Poisson observation, i.e., $T_{b}^{+}
=\inf\{T_{i}; (b-X_{T_{i}})-\sup_{s\leq T_{i}}((b-X_{s})-b)\vee 0<0\}$. Hence, by (23) and (27) of \cite{Albrecher16}, we have
\begin{eqnarray}
\label{V.b.3}\mathrm{E}_x\left[e^{-qT_b^+}\right]\hspace{-0.3cm}&=&\hspace{-0.3cm}
\frac{\gamma}{\gamma+q}\left[Z_q(b-x)-qW_q(b)\frac{Z_q(b-x,\Phi_{q+\gamma})}{Z^{\prime}_q(b,\Phi_{q+\gamma})}\right],\quad x\in[0,b],
\\
\mathrm{E}_x\left[\int_0^{T_b^+}e^{-qt}\mathrm{d}R_t^{0,b}\right]
\hspace{-0.3cm}&=&\hspace{-0.3cm}
\frac{Z_q(b-x,\Phi_{q+\gamma})}{Z^{\prime}_q(b,\Phi_{q+\gamma})},\quad x\in[0,b],
\\
\mathrm{E}_{x}\Big[e^{-qT_{b}^{+} +\theta (b-\widetilde{Y}(T_{b}^{+}))}\Big]
\hspace{-0.3cm}&=&\hspace{-0.3cm}
\frac{\gamma\left[Z_q(b-x,\theta)+Z_q(b-x,\Phi_{q+\gamma})\frac{W_q(b)(\psi(\theta)-q)-\theta Z_q(b,\theta)}{Z^{\prime}_q(b,\Phi_{q+\gamma})}\right]}{q+\gamma-\psi(\theta)}\quad x\in[0,b].\label{joint.L.}
\end{eqnarray}
Differentiating the both sides of \eqref{joint.L.} with respect to $\theta$ and then letting $\theta\rightarrow 0$ yields
\begin{eqnarray}\label{V.b.4}
\hspace{-0.3cm}
\mathrm{E}_{x}\left[e^{-qT_{b}^{+}}\left(b-\widetilde{Y}(T_{b}^{+})\right)\right]
\hspace{-0.3cm}&=&\hspace{-0.3cm}
\frac{\gamma\psi^{\prime}(0+)}{(q+\gamma)^2}\left[Z_q(b-x)-Z_q(b-x,\Phi_{q+\gamma})\frac{qW_q(b)}{Z_q^{\prime}(b,\Phi_{q+\gamma})}\right]
+\frac{\gamma}{q+\gamma}\nonumber\\
\hspace{-0.3cm}&&\hspace{-4cm}
\times\left[\overline{Z}_q(b-x)-\psi^{\prime}(0+)\overline{W}_q(b-x)
+Z_q(b-x,\Phi_{q+\gamma})\frac{W_q(b)\psi^{\prime}(0+)-Z_q(b)}{Z_q^{\prime}(b,\Phi_{q+\gamma})}
\right],\quad x\in[0,b].
\end{eqnarray}
Combing (\ref{V.b.3})-(\ref{V.b.4}) and the strong Markov property, we have
\begin{eqnarray}\label{f.b.2}
\hspace{-3cm}
f(x)
\hspace{-0.3cm}&=&\hspace{-0.3cm}
-\phi\mathrm{E}_x\left[\int_0^{T_b^+}e^{-qt}\mathrm{d}R^{0,b}_t\right]+\mathrm{E}_x\Big[e^{-qT^+_b}\left(\widetilde{Y}(T_{b}^{+})-b+f(b)\right)\Big]
\nonumber\\
\hspace{-0.3cm}&=&\hspace{-0.3cm}
-\phi\frac{Z_q(b-x,\Phi_{q+\gamma})}{Z^{\prime}_q(b,\Phi_{q+\gamma})}
+\frac{\gamma}{q+\gamma}\left[\left[Z_q(b-x)-qW_q(b)\frac{Z_q(b-x,\Phi_{q+\gamma})}{Z^{\prime}_q(b,\Phi_{q+\gamma})}\right]\left(f(b)-\frac{\psi^{\prime}(0+)}{q+\gamma}\right)\right.
\nonumber\\
\hspace{-0.3cm}&&\hspace{-0.3cm}
-\left.\left[\overline{Z}_q(b-x)-\psi^{\prime}(0+)\overline{W}_q(b-x)
+Z_q(b-x,\Phi_{q+\gamma})\frac{W_q(b)\psi^{\prime}(0+)-Z_q(b)}{Z_q^{\prime}(b,\Phi_{q+\gamma})}
\right]\right],
\quad x\in[0,b].
\end{eqnarray}
Letting $x=b$ in (\ref{f.b.2}) gives rise to
\begin{eqnarray}\label{Vdr.b}
f(b)
\hspace{-0.3cm}&=&\hspace{-0.3cm}
\frac{\gamma Z_q(b)-\phi(q+\gamma)}{q\Phi_{q+\gamma}Z_q(b,\Phi_{q+\gamma})}-\frac{\gamma\psi^{\prime}(0+)}{q(q+\gamma)},\nonumber
\end{eqnarray}
substituting which into (\ref{f.b.1}) and (\ref{f.b.2}) yields the desired result. The proof is complete.
\end{proof}

To obtain an expression for the value function associated with a double barrier periodic dividend and capital injection strategy, we still need to find an expression for the potential measure of the spectrally positive L\'evy process controlled by a double barrier periodic dividend and capital injection strategy. To this end, we need to make some preparations in the following Lemmas \ref{lem2.1}-\ref{2.2}. In Lemma \ref{lem2.1} below, we compute an integral with respect to the excursion measure $n$, where the integrand involves the scale function, the excursion height as well as the first passage time of excursion.

\begin{lem}
\label{lem2.1}
For $x\in(0,\infty)$, we have
\begin{eqnarray}
\hspace{-0.3cm}&&\hspace{-0.3cm}
n\left(e^{-p \rho_{x}^{+}(\varepsilon)}
W_{q}(
x-\epsilon(\rho_{x}^{+})+y
)
\mathbf{1}_{\{\overline{\epsilon}
\geq x\}}\right)
\nonumber\\
\hspace{-0.3cm}&=&\hspace{-0.3cm}
-W_{p}(x)\frac{\mathrm{d}}{\mathrm{d}x}\left[\frac{W_{p}(x+y)+(q-p)\int_{0}^{y}W_{p}(x+y-z)W_{q}(z)\mathrm{d}z}
{W_{p}(x)}\right].
\nonumber
\end{eqnarray}
\end{lem}

\begin{proof}
For $0\leq a\leq x\leq c$ and $y\geq0$, by \eqref{fluc.1} it can be checked that
\begin{eqnarray}
\frac{{W}_{q}(x+y)}{{W}_{q}(c+y)}
\hspace{-0.3cm}&=&\hspace{-0.3cm}
\mathrm{E}_{-x-y}\left[e^{-q\tau_{-c-y}^{-}}\mathbf{1}_{\{\tau_{-c-y}^{-}<\tau_{0}^{+}\}}\right]
\nonumber\\
\hspace{-0.3cm}&=&\hspace{-0.3cm}
\mathrm{E}_{-x-y}\left[e^{-q\tau_{-c-y}^{-}}\mathbf{1}_{\{\tau_{-c-y}^{-}<\tau_{-a-y}^{+}\}}\right]
+\mathrm{E}_{-x-y}\left[\mathrm{E}_{-x-y}\left[\left.e^{-q\tau_{-c-y}^{-}}\mathbf{1}_{\{\tau_{-a-y}^{+}<\tau_{-c-y}^{-}<\tau_{0}^{+}\}}\right|\mathcal{F}_{\tau_{-a-y}^{+}}\right]\right]
\nonumber\\
\hspace{-0.3cm}&=&\hspace{-0.3cm}
\frac{{W}_{q}(x-a)}{{W}_{q}(c-a)}
+\mathrm{E}_{-x-y}\left[e^{-q\tau_{-a-y}^{+}}\mathbf{1}_{\{\tau_{-a-y}^{+}<\tau_{-c-y}^{-}\}}
\frac{{W}_{q}(-X_{\tau_{-a-y}^{+}})}{{W}_{q}(c+y)}\right],\nonumber
\end{eqnarray}
which is equivalent to
\begin{eqnarray}
\mathrm{E}_{-x}\left[e^{-q\tau_{-a}^{+}}
{W}_{q}(-X_{\tau_{-a}^{+}}+y)\mathbf{1}_{\{\tau_{-a}^{+}<\tau_{-c}^{-}\}}\right]
\hspace{-0.3cm}&=&\hspace{-0.3cm}{W}_{q}(x+y)-\frac{{W}_{q}(x-a)}{{W}_{q}(c-a)}{W}_{q}(c+y).\nonumber
\end{eqnarray}
Then, by Lemma 2.1 of \cite{Loeffen14}, we have that
\begin{eqnarray}
\mathrm{E}_{-x}\left[e^{-p\tau_{-a}^{+}}
{W}_{q}(-X_{\tau_{-a}^{+}}+y)\mathbf{1}_{\{\tau_{-a}^{+}<\tau_{-c}^{-}\}}\right]
\hspace{-0.3cm}&=&\hspace{-0.3cm}
{W}_{q}(x+y)-(q-p)\int_{a}^{x}{W}_{p}(x-z){W}_{q}(z+y)\mathrm{d}z
\nonumber\\
\hspace{-0.3cm}&&\hspace{-3cm}
-\frac{{W}_{p}(x-a)}{{W}_{p}(c-a)}\left({W}_{q}(c+y)-(q-p)\int_{a}^{c}{W}_{p}(c-z)W_{q}(z+y)\mathrm{d}z\right).\nonumber
\end{eqnarray}
In particular, we have
\begin{eqnarray}
\label{2.10.0}
\mathrm{E}_{-x}\left[e^{-p\tau_{0}^{+}}
{W}_{q}(-X_{\tau_{0}^{+}}+y)\mathbf{1}_{\{\tau_{0}^{+}<\infty\}}\right]
\hspace{-0.3cm}&=&\hspace{-0.3cm}
{W}_{q}(x+y)-(q-p)\int_{0}^{x}{W}_{p}(x-z){W}_{q}(z+y)\mathrm{d}z
\nonumber\\
\hspace{-0.3cm}&&\hspace{-3cm}
-\lim\limits_{c\uparrow\infty}\frac{{W}_{p}(x)}{{W}_{p}(c)}
\left({W}_{p}(c+y)+(q-p)\int_{0}^{y}{W}_{p}(c+y-z){W}_{q}(z)\mathrm{d}z\right)
\nonumber
\\
\hspace{-0.3cm}&=&\hspace{-0.3cm}
{W}_{p}(x+y)+(q-p)\int_{0}^{y}{W}_{p}(x+y-z){W}_{q}(z)\mathrm{d}z
\nonumber\\
\hspace{-0.3cm}&&\hspace{-3cm}
-{W}_{p}(x)
\left(e^{\Phi_{p}y}+(q-p)\int_{0}^{y}e^{\Phi_{p}(y-z)}{W}_{q}(z)\mathrm{d}z\right),
\end{eqnarray}
where we have used the fact that $\lim\limits_{z\rightarrow\infty}\frac{{W}_{q}(x+z)}{{W}_{q}(z)}=\mathrm{e}^{\Phi_{q}x}$, and that, for non-negative $p$ and $q$
\begin{eqnarray}
\label{3.17}
\hspace{-0.3cm} & & \hspace{-0.3cm}
{W}_{q}(x+y)-(q-p)\int_{0}^{x}{W}_{q}(z+y){W}_{p}(x-z)\mathrm{d}z\nonumber \\
\hspace{-0.3cm} & = & \hspace{-0.3cm}
{W}_{p}(x+y)+(q-p)\int_{0}^{y}{W}_{p}(x+y-z){W}_{q}(z)\mathrm{d}z,
\quad
x>0, \,\,x+y\geq0,
\end{eqnarray}
which can be verified by taking Laplace transform on both sides of the above equation.
By the compensation formula, we have
\begin{eqnarray}
\hspace{-0.3cm}&&\hspace{-0.3cm}
\mathrm{E}_{-x}\left(e^{-p \tau_{0}^{+}}W_{q}(-X_{\tau_{0}^{+}}+y)\mathbf{1}_{\{\tau_{0}^{+}<\infty\}}\right)
\nonumber\\
\hspace{-0.3cm}&=&\hspace{-0.3cm}
\mathrm{E}_{-x}\left(\sum_{g}e^{-p g}\prod\limits_{h<g}\mathbf{1}_{\{\overline{\epsilon}_{h}
<x+L(h)\}}
e^{-p \rho_{x+L(g)}^{+}(\epsilon_{g})}\right.
\nonumber\\
\hspace{-0.3cm}&&\hspace{0.5cm}
\left.
\times
W_{q}(
x+L(g)-\epsilon_{g}(\rho_{x+L(g)}^{+})
+y)
\mathbf{1}_{\{\overline{\epsilon}_{g}\geq x+L(g)\}}\right)
\nonumber\\
\hspace{-0.3cm}&=&\hspace{-0.3cm}
\mathrm{E}_{-x}\left(\int_{0}^{\infty}e^{-p w}\prod\limits_{h<w}\mathbf{1}_{\{\overline{\epsilon}_{h}
<x+L(h)\}}
\int_{\mathcal{E}}e^{-p \rho_{x+L(w)}^{+}(\varepsilon)}\right.
\nonumber\\
\hspace{-0.3cm}&&\hspace{0.5cm}
\left.
\times
W_{q}(
x+L(w)-\epsilon(\rho_{x+L(w)}^{+})+y
)
\mathbf{1}_{\{\overline{\epsilon}
\geq x+L(w)\}}\,n(\,\mathrm{d}\varepsilon)\,\mathrm{d}L(w)\right)
\nonumber\\
\hspace{-0.3cm}&=&\hspace{-0.3cm}
\mathrm{E}_{-x}\left(\int_{0}^{\infty}e^{-p L^{-1}(w-)}
\prod\limits_{h<L^{-1}(w-)}\mathbf{1}_{\{\overline{\epsilon}_{h}
<x+L(h)\}}
\right.
\nonumber\\
\hspace{-0.3cm}&&\hspace{0.5cm}
\left.
\times\int_{\mathcal{E}}
e^{-p \rho_{x+w}^{+}(\varepsilon)}
W_{q}(
x+w-\epsilon(\rho_{x+w}^{+})+y
)
\mathbf{1}_{\{\overline{\epsilon}
\geq x+w\}}\,n(\,\mathrm{d}\varepsilon)\,\mathrm{d}w\right)
\nonumber\\
\hspace{-0.3cm}&=&\hspace{-0.3cm}
\int_{x}^{\infty}\mathrm{E}_{-x}\left(e^{-p \tau_{-w}^{-}}\mathbf{1}_{\{\tau_{-w}^{-}<\tau_{0}^{+}\}}\right)
\,n\left(e^{-p \rho_{w}^{+}(\varepsilon)}
W_{q}(
w-\epsilon(\rho_{w}^{+})+y
)
\mathbf{1}_{\{\overline{\epsilon}
\geq w\}}\right)\,\mathrm{d}w
\nonumber\\
\hspace{-0.3cm}&=&\hspace{-0.3cm}
\int_{x}^{\infty}
\frac{W_{p}(x)}{W_{p}(w)}
\,n\left(e^{-p \rho_{w}^{+}(\varepsilon)}
W_{q}(
w-\epsilon(\rho_{w}^{+})+y
)
\mathbf{1}_{\{\overline{\epsilon}
\geq w\}}\right)\,\mathrm{d}w
,\nonumber
\end{eqnarray}
which together with \eqref{2.10.0} yields the desired result.
\end{proof}

The following Lemma \ref{2.2} below computes an expectation with the integrand involving the scale function, the surplus process with capital injection $\widetilde{Y}$ and the first passage time of $\widetilde{Y}$.

\begin{lem}\label{2.2}
For $x\in(0,b)$ and $y\in(0,\infty)$, we have
\begin{eqnarray}
\hspace{-0.6cm}\hspace{-0.3cm}
\mathrm{E}_{x}\left[e^{-q\sigma_{b}^{+}}W_{q+\gamma}(b-{\widetilde{Y}}_{{\sigma}_{b}^{+}}+y)\mathbf{1}_{\{\sigma_{b}^{+}<\infty\}}\right]
\hspace{-0.3cm}&=&\hspace{-0.3cm}
W_q(b-x+y)+\gamma\int_0^yW_q(b-x+y-z)W_{q+\gamma}(z)\mathrm{d}z
\nonumber\\
\hspace{-0.3cm}&&\hspace{-0.3cm}\hspace{-0.3cm}\hspace{-0.3cm}\hspace{-1cm}
-\frac{W_q(b-x)}{W_q^{\prime+}(b)}\Big(W^{\prime+}_q(b+y)+\gamma\int_0^yW^{\prime+}_q(b+y-z)W_{q+\gamma}(z)\mathrm{d}z\Big).\nonumber
\end{eqnarray}
\end{lem}
\begin{proof}
Put $\xi_{b}(z)=z\wedge 0+b$ and $\overline{\xi_{b}}(z)=\xi_{b}(z)-z$. Recall that
$\widetilde{Y}_t=X_{t}-\underline{X}_{t}\wedge0$ and
${\sigma}_{b}^{+}=\inf\{t\geq0; \widetilde{Y}_t>b\}=\inf\{t\geq0; X_{t}>\xi_{b}(\underline{X}_{t})\}$.
Adapting Lemma 3.2 of \cite{Wang18}, one can check that
\begin{eqnarray}\label{tsei.}
\mathbb{E}_{x}\left(e^{-q\tau_{x-w}^{-}}
\mathbf{1}_{\{\tau_{x-w}^{-}<\sigma_{b}^{+}\}}\right)
=\exp\left(-\int_{x-w}^{x}\frac{W_{q}^{\prime}(\overline{\xi_{b}}\left(z\right))}
{W_{q}(\overline{\xi_{b}}\left(z\right))}\mathrm{d}z\right), \quad x\in(0,b),\, w\in(0,\infty).
\end{eqnarray}
By \eqref{tsei.}, Lemma \ref{lem2.1} as well as the compensation formula, we have
\begin{eqnarray}
\hspace{-0.3cm}&&\hspace{-0.3cm}
\mathrm{E}_{x}\left[e^{-q\sigma_{b}^{+}}W_{q+\gamma}(b-{\widetilde{Y}}_{{\sigma}_{b}^{+}}+y)\mathbf{1}_{\{\sigma_{b}^{+}<\infty\}}\right]
\nonumber\\
\hspace{-0.3cm}&=&\hspace{-0.3cm}
\mathbb{E}_{x}\left(\sum_{g}e^{-q g}\prod\limits_{h<g}\mathbf{1}_{\{\overline{\epsilon}_{h}
<\overline{\xi_{b}}(x-L(h))\}}
e^{-q \rho_{\overline{\xi_{b}}(x-L(g))}^{+}(\epsilon_{g})}\right.
\nonumber\\
\hspace{-0.3cm}&&\hspace{0.5cm}
\left.
\times
W_{q+\gamma}(b+((x-L(g))\wedge 0)-
(x-L(g))-\epsilon_{g}(\rho_{\overline{\xi_{b}}(x-L(g))}^{+})
+y)
\mathbf{1}_{\{\overline{\epsilon}_{g}\geq \overline{\xi_{b}}(x-L(g))\}}\right)
\nonumber\\
\hspace{-0.3cm}&=&\hspace{-0.3cm}
\mathbb{E}_{x}\left(\int_{0}^{\infty}e^{-q w}\prod\limits_{h<w}\mathbf{1}_{\{\overline{\epsilon}_{h}
<\overline{\xi_{b}}(x-L(h))\}}
\int_{\mathcal{E}}e^{-q \rho_{\overline{\xi_{b}}(x-L(w))}^{+}(\varepsilon)}\right.
\nonumber\\
\hspace{-0.3cm}&&\hspace{0.5cm}
\left.
\times
W_{q+\gamma}(
\overline{\xi_{b}}(x-L(w))-\epsilon_{g}(\rho_{\overline{\xi_{b}}(x-L(w))}^{+})
+y
)
\mathbf{1}_{\{\overline{\epsilon}
\geq \overline{\xi_{b}}(x-L(w))\}}\,n(\,\mathrm{d}\varepsilon)\,\mathrm{d}L(w)\right)
\nonumber\\
\hspace{-0.3cm}&=&\hspace{-0.3cm}
\mathbb{E}_{x}\left(\int_{0}^{\infty}e^{-q L^{-1}(w-)}
\prod\limits_{h<L^{-1}(w-)}\mathbf{1}_{\{\overline{\epsilon}_{h}
<\overline{\xi_{b}}(x-L(h))\}}
\right.
\nonumber\\
\hspace{-0.3cm}&&\hspace{0.5cm}
\left.
\times\int_{\mathcal{E}}
e^{-q \rho_{\overline{\xi_{b}}(x-w)}^{+}(\varepsilon)}
W_{q+\gamma}(
\overline{\xi_{b}}(x-w)-\epsilon(\rho_{\overline{\xi_{b}}(x-w)}^{+})+y
)
\mathbf{1}_{\{\overline{\epsilon}
\geq \overline{\xi_{b}}(x-w)\}}\,n(\,\mathrm{d}\varepsilon)\,\mathrm{d}w\right)
\nonumber\\
\hspace{-0.3cm}&=&\hspace{-0.3cm}
\int_{0}^{\infty}\mathbb{E}_{x}\left(e^{-q\tau_{x-w}^{-}}
\mathbf{1}_{\{\tau_{x-w}^{-}<\sigma_{b}^{+}\}}\right)
\,n\left(e^{-q \rho_{\overline{\xi_{b}}(x-w)}^{+}(\varepsilon)}
W_{q+\gamma}(
\overline{\xi_{b}}(x-w)-\epsilon(\rho_{\overline{\xi_{b}}(x-w)}^{+})+y
)
\mathbf{1}_{\{\overline{\epsilon}
\geq \overline{\xi_{b}}(x-w)\}}\right)\,\mathrm{d}w
\nonumber\\
\hspace{-0.3cm}&=&\hspace{-0.3cm}
\int_{-\infty}^{x}
\exp\left(-\int_{w}^{x}\frac{W_{q}^{\prime}(\overline{\xi_{b}}\left(z\right))}
{W_{q}(\overline{\xi_{b}}\left(z\right))}\mathrm{d}z\right)
\,n\left(e^{-q\rho_{\overline{\xi_{b}}(w)}^{+}(\varepsilon)}
W_{q+\gamma}(
\overline{\xi_{b}}(w)-\epsilon(\rho_{\overline{\xi_{b}}(w)}^{+})+y
)
\mathbf{1}_{\{\overline{\epsilon}
\geq \overline{\xi_{b}}(w)\}}\right)\,\mathrm{d}w
\nonumber\\
\hspace{-0.3cm}&=&\hspace{-0.3cm}
-\int_{-\infty}^{x}
\exp\left(-\int_{w}^{x}\frac{W_{q}^{\prime+}(\overline{\xi_{b}}\left(z\right))}
{W_{q}(\overline{\xi_{b}}\left(z\right))}\mathrm{d}z\right)
\nonumber\\
\hspace{-0.3cm}&&\hspace{-0.3cm}
\times
W_{q}(\overline{\xi_{b}}(w))\left.\frac{\mathrm{d}}{\mathrm{d}v}\left[\frac{W_{q}(v+y)+\gamma\int_{0}^{y}W_{q}(v+y-z)W_{q+\gamma}(z)\mathrm{d}z}
{W_{q}(v)}\right]\right|_{v=\overline{\xi_{b}}(w)}
\mathrm{d}w
\nonumber\\
\hspace{-0.3cm}&=&\hspace{-0.3cm}
-\int_{0}^{x}
\exp\left(-\int_{w}^{x}\frac{W_{q}^{\prime+}(b-z)}
{W_{q}(b-z)}\mathrm{d}z\right)
\nonumber\\
\hspace{-0.3cm}&&\hspace{-0.3cm}
\times
W_{q}(b-w)\left.\frac{\mathrm{d}}{\mathrm{d}v}\left[\frac{W_{q}(v+y)+\gamma\int_{0}^{y}W_{q}(v+y-z)W_{q+\gamma}(z)\mathrm{d}z}
{W_{q}(v)}\right]\right|_{v=b-w}
\mathrm{d}w
\nonumber\\
\hspace{-0.3cm}&&\hspace{-0.3cm}
-\int_{-\infty}^{0}
\exp\left(-\int_{0}^{x}\frac{W_{q}^{\prime+}(b-z)}
{W_{q}(b-z)}\mathrm{d}z\right)
\exp\left(-\int_{w}^{0}\frac{W_{q}^{\prime+}(b)}
{W_{q}(b)}\mathrm{d}z\right)
\nonumber\\
\hspace{-0.3cm}&&\hspace{-0.3cm}
\times
W_{q}(b)\left.\frac{\mathrm{d}}{\mathrm{d}v}\left[\frac{W_{q}(v+y)+\gamma\int_{0}^{y}W_{q}(v+y-z)W_{q+\gamma}(z)\mathrm{d}z}
{W_{q}(v)}\right]\right|_{v=b}
\mathrm{d}w
\nonumber\\
\hspace{-0.3cm}&=&\hspace{-0.3cm}
-W_q(b-x)\int_{-x}^{0}
\left.\frac{\mathrm{d}}{\mathrm{d}v}\left[\frac{W_{q}(v+y)+\gamma\int_{0}^{y}W_{q}(v+y-z)W_{q+\gamma}(z)\mathrm{d}z}
{W_{q}(v)}\right]\right|_{v=w+b}
\mathrm{d}w
\nonumber\\
\hspace{-0.3cm}&&\hspace{-0.3cm}
-W_q(b-x)\int_{0}^{\infty}e^{-\frac{W_q^{\prime}(b)}{W_q(b)}w}\mathrm{d}w
\left.\frac{\mathrm{d}}{\mathrm{d}v}\left[\frac{W_{q}(v+y)+\gamma\int_{0}^{y}W_{q}(v+y-z)W_{q+\gamma}(z)\mathrm{d}z}
{W_{q}(v)}\right]\right|_{v=b}
\nonumber\\
\hspace{-0.3cm}&=&\hspace{-0.3cm}
W_q(b-x+y)+\gamma\int_0^yW_q(b-x+y-z)W_{q+\gamma}(z)\mathrm{d}z
\nonumber\\
\hspace{-0.3cm}&&\hspace{-0.3cm}
-\frac{W_q(b-x)}{W_q^{\prime+}(b)}\Big(W^{\prime+}_q(b+y)+\gamma\int_0^yW^{\prime+}_q(b+y-z)W_{q+\gamma}(z)\mathrm{d}z\Big)
,\nonumber
\end{eqnarray}
which is the desired result.
\end{proof}

Thanks to the above Lemma \ref{2.2}, we are now able to give the expression for the potential measure of the spectrally positive L\'evy process controlled by a double barrier periodic dividend and capital injection strategy in the following Lemma \ref{lem.w}.

\begin{lem}\label{lem.w}
For $q>0$, $x>0$, and $\lambda>0$, we have
\begin{eqnarray}\label{pot.mea}
\mathrm{P}_x\Big(U_{e_q}^{0,b}\in\mathrm{d}y\Big)\hspace{-0.3cm}&=&\hspace{-0.3cm}
\frac{qZ_q(b-x,\Phi_{q+\gamma})+\gamma Z_q(b-x)}{\Phi_{q+\gamma}Z_q(b,\Phi_{q+\gamma})}
\bigg[W_q(0+)\delta_0(\mathrm{d}y)+W_q^{\prime+}(y)\mathbf{1}_{(0,b)}(y)\mathrm{d}y
\nonumber\\
\hspace{-0.3cm}&&\hspace{-0.3cm}
+\Big(\gamma W_q(b)W_{q+\gamma}(y-b)+W_q^{\prime+}(y)+\gamma\int_0^{y-b}W_q^{\prime+}(y-z)W_{q+\gamma}(z)\mathrm{d}z\Big)\mathbf{1}_{(b,\infty)}(y)\mathrm{d}y
\nonumber\\
\hspace{-0.3cm}&&\hspace{-0.3cm}
-\frac{q\Phi_{q+\gamma}Z_q(b,\Phi_{q+\gamma})}{qZ_q(b-x,\Phi_{q+\gamma})+\gamma Z_q(b-x)}\Big(\frac{\gamma}{q}W_{q+\gamma}(y-b)Z_q(b-x)+W_q(y-x)
\nonumber\\
\hspace{-0.3cm}&&\hspace{-0.3cm}
+\gamma\int_0^yW_q(y-x-z)W_{q+\gamma}(z)\mathrm{d}z\Big)\mathbf{1}_{(b,\infty)}(y)\mathrm{d}y\bigg]-qW_q(y-x)\mathbf{1}_{(0,b)}(y)\mathrm{d}y.
\end{eqnarray}


\end{lem}
\begin{proof}
Denote by $g(x)$ the left hand side of \eqref{pot.mea}.
One can verify that
\begin{eqnarray}\label{2.11}
\hspace{-0.3cm}g(x)\hspace{-0.3cm}&=&\hspace{-0.3cm}
\mathrm{P}_x\Big(U_{e_q}^{0,b}\in\mathrm{d}y;e_q<T_{1}\wedge \kappa_{b}^{-}\Big)+\mathrm{P}_x\Big(U_{e_q}^{0,b}\in\mathrm{d}y;e_q>T_{1}\wedge \kappa_{b}^{-}\Big)
\nonumber\\
\hspace{-0.3cm}&=&\hspace{-0.3cm}
\mathrm{P}_x\Big(U_{e_q}^{0,b}\in\mathrm{d}y;e_q>T_{1},\kappa_{b}^{-}>T_1\Big)+\mathrm{P}_x\Big(U_{e_q}^{0,b}\in\mathrm{d}y;e_q>\kappa_{b}^{-},T_1>\kappa_{b}^{-}\Big)
\nonumber\\
\hspace{-0.3cm}&&\hspace{-0.3cm}
+\frac{q}{q+\gamma}\mathrm{P}_x\Big(X_{e_{q+\gamma}}\in\mathrm{d}y;e_{q+\gamma}<\tau_b^-\Big)
\nonumber\\
\hspace{-0.3cm}&=&\hspace{-0.3cm}
\frac{q}{q+\gamma}\mathrm{P}_x\Big(X_{e_{q+\gamma}}\in\mathrm{d}y;e_{q+\gamma}<\tau_b^-\Big)+\bigg[\mathrm{E}_{x}\left[e^{-(q+\gamma) \tau_{b}^{-}}\right]
+\frac{\gamma}{q+\gamma}\mathrm{E}_{x}\left[1-e^{-(q+\gamma) \tau_{b}^{-}}\right]\bigg]g(b)
\nonumber\\
\hspace{-0.3cm}&=&\hspace{-0.3cm}
\frac{q}{q+\gamma}\mathrm{P}_x\Big(X_{e_{q+\gamma}}\in\mathrm{d}y;e_{q+\gamma}<\tau_b^-\Big)+\bigg[\frac{\gamma}{q+\gamma}+\frac{q}{q+\gamma}\mathrm{E}_{x}\left[e^{-(q+\gamma) \tau_{b}^{-}}\right]
\bigg]g(b)
\nonumber\\
\hspace{-0.3cm}&=&
\hspace{-0.3cm}
q\left(e^{\Phi_{q+\gamma}(b-x)}W_{q+\gamma}(y-b)-W_{q+\gamma}(y-x)\right)\mathbf{1}_{(b,\infty)}(y)\mathrm{d}y+\frac{\gamma}{q+\gamma}g(b)
\nonumber\\
\hspace{-0.3cm}&&
\hspace{-0.3cm}
+\frac{q}{q+\gamma}e^{\Phi_{q+\gamma}(b-x)}g(b),\quad x\in(b,\infty).
\end{eqnarray}
To obtain an expression of $g(x)$ for $x\in[0,b]$,
recall that $\widetilde{Y}_t=X_{t}-\underline{X}_{t}\wedge0$
and
${\sigma}_{b}^{+}=\inf\{t\geq0; \widetilde{Y}_t>b\}$. 
When $x\in(0,b)$, by Theorem 1 of \cite{Pistorius04} and \eqref{2.11}, it is verified that
\begin{eqnarray}\label{g.1}
g(x)\hspace{-0.3cm}&=&
\hspace{-0.3cm}\mathrm{P}_x\Big(U_{e_q}^{0,b}\in\mathrm{d}y;e_q<\kappa_b^+\Big)+\mathrm{P}_x\Big(U_{e_q}^{0,b}\in\mathrm{d}y;e_q>\kappa_b^+\Big)\nonumber\\
\hspace{-0.3cm}&=&\hspace{-0.3cm}
\mathrm{P}_x\Big(\widetilde{Y}_{e_q}\in\mathrm{d}y;e_q<\sigma_b^+\Big)+\mathrm{E}_x\big[e^{-q\sigma_b^+}g(\widetilde{Y}_{\sigma_b^+})\mathbf{1}_{\{\sigma_b^+<\infty\}}\big]
\nonumber\\
\hspace{-0.3cm}&=&\hspace{-0.3cm}
q\frac{W_{q}(b-x)W_{q}(0+)}{W_{q}^{\prime+}(b)}\delta_0(\mathrm{d}y)+q
\left(W_{q}(b-x)\frac{W_{q}^{\prime+}(y)}{W_{q}^{\prime+}(b)}-W_{q}(y-x)\right)\mathbf{1}_{(0,b)}(y)\mathrm{d}y
\nonumber\\
\hspace{-0.3cm}&&\hspace{-0.3cm}
+q\left(
\mathrm{E}_{x}\left[e^{-q{\sigma}_{b}^{+}}e^{\Phi_{q+\gamma}(b-{\widetilde{Y}}_{{\sigma}_{b}^{+}})}\right]W_{q+\gamma}(y-b)-\mathrm{E}_{x}\left[e^{-q{\sigma}_{b}^{+}}W_{q+\gamma}(y-{\widetilde{Y}}_{{\sigma}_{b}^{+}})\right]\right)\mathbf{1}_{(b,\infty)}(y)\mathrm{d}y
\nonumber\\
\hspace{-0.3cm}&&\hspace{-0.3cm}
+\frac{\gamma}{q+\gamma}\mathrm{E}_{x}\big[e^{-q{\sigma}_b^+}\big]g(b)
+\frac{q}{q+\gamma}\mathrm{E}_{x}\left[e^{-q{\sigma}_{b}^{+}}e^{\Phi_{q+\gamma}(b-{\widetilde{Y}}_{{\sigma}_{b}^{+}})}\right]g(b)
,\quad x\in[0,b].
\end{eqnarray}
By adapting Lemma 3.2 of \cite{Wang21}, one can derive
\begin{eqnarray}\label{g.2}
\hspace{-0.3cm}&&\hspace{-0.3cm}
\mathrm{E}_{x}\left[e^{-q{\sigma}_{b}^{+}}e^{\Phi_{q+\gamma}(b-{\widetilde{Y}}_{{\sigma}_{b}^{+}})}\right]
=
\int_{-x}^{0}\exp{\bigg(-\int_{-x}^s\frac{W_q^{\prime+}(z+b)}{W_q(z+b)}\mathrm{d}z\bigg)}
\nonumber\\
\hspace{-0.3cm}&&\hspace{-0.3cm}
\times\bigg(\frac{W_q^{\prime+}(s+b)}{W_q(s+b)}Z_q(s+b,\Phi_{q+\gamma})
-\Phi_{q+\gamma}Z_q(s+b,\Phi_{q+\gamma})+\gamma W_q(s+b)\bigg)\mathrm{d}s
\nonumber\\
\hspace{-0.3cm}&&\hspace{-0.3cm}
+\int_{0}^{\infty}\exp{\bigg(-\int_{-x}^0\frac{W_q^{\prime+}(z+b)}{W_q(z+b)}\mathrm{d}z-\int_{0}^s\frac{W_q^{\prime+}(b)}{W_q(b)}\mathrm{d}z\bigg)}
\nonumber\\
\hspace{-0.3cm}&&\hspace{-0.3cm}
\times\bigg(\frac{W_q^{\prime+}(b)}{W_q(b)}Z_q(b,\Phi_{q+\gamma})
-\Phi_{q+\gamma}Z_q(b,\Phi_{q+\gamma})+\gamma W_q(b)\bigg)\mathrm{d}s
\nonumber\\
\hspace{-0.3cm}&=&\hspace{-0.3cm}
Z_q(b-x,\Phi_{q+\gamma})-\frac{W_q(b-x)}{W_q^{\prime+}(b)}\Big(\Phi_{q+\gamma}Z_q(b,\Phi_{q+\gamma})-\gamma W_q(b)\Big),\quad x\in[0,b].
\end{eqnarray}
Combining (\ref{g.1}), (\ref{g.2}) and Lemma \ref{2.2}, it is verified that
\begin{eqnarray}\label{3.23}
g(x)\hspace{-0.3cm}&=&\hspace{-0.3cm}
q\frac{W_q(b-x)W_q(0+)}{W_q^{\prime+}(b)}\delta_0(\mathrm{d}y)
+q\Bigg[W_{q+\gamma}(y-b)\bigg(Z_q(b-x,\Phi_{q+\gamma})
-\frac{W_q(b-x)}{W_q^{\prime+}(b)}
\nonumber\\
\hspace{-0.3cm}&&\hspace{-0.3cm}
\times\Big(\Phi_{q+\gamma}Z_q(b,\Phi_{q+\gamma})-\gamma W_q(b)\Big)\bigg)
-W_q(y-x)-\gamma\int_0^{y-b}W_q(y-x-z)W_{q+\gamma}(z)\mathrm{d}z
\nonumber\\
\hspace{-0.3cm}&&\hspace{-0.3cm}
+\frac{W_q(b-x)}{W^{\prime+}_q(b)}\Big(W_q^{\prime+}(y)+\gamma\int_0^{y-b}W_q^{\prime+}(y-z)W_{q+\gamma}(z)\mathrm{d}z\Big)\Bigg]\mathbf{1}_{\{(b,\infty)\}}(y)\mathrm{d}y
\nonumber\\
\hspace{-0.3cm}&&\hspace{-0.3cm}
+\bigg[\frac{\gamma}{q+\gamma}Z_q(b-x)+\frac{q}{q+\gamma}\Big(Z_q(b-x,\Phi_{q+\gamma})-\Phi_{q+\gamma}Z_q(b,\Phi_{q+\gamma})\frac{W_q(b-x)}{W_q^{\prime+}(b)}\Big)\bigg]g(b)
\nonumber\\
\hspace{-0.3cm}&&\hspace{-0.3cm}
+q\Big(W_q(b-x)\frac{W_q^{\prime+}(y)}{W_q^{\prime+}(b)}-W_q(y-x)\Big)\mathbf{1}_{(0,b)}(y)\mathrm{d}y,\quad x\in[0,b].
\end{eqnarray}
When $X$ has paths of bounded variation, letting $x=b$ in (\ref{3.23}) and then using \eqref{3.17} with $x=0$, we get
\begin{eqnarray}\label{g.b.1}
g(b)
\hspace{-0.3cm}&=&\hspace{-0.3cm}
\Big(\frac{q}{q+\gamma}\Phi_{q+\gamma}Z_q(b,\Phi_{q+\gamma})\Big)^{-1}\Bigg[qW_q(0+)\delta_0(\mathrm{d}y)+qW_q^{\prime+}(y)\mathbf{1}_{(0,b)}(y)\mathrm{d}y+q\Big[W_{q+\gamma}(y-b)
\nonumber\\
\hspace{-0.3cm}&&\hspace{-1cm}
\times\Big(\gamma  W_q(b)-\Phi_{q+\gamma}Z_q(b,\Phi_{q+\gamma})\Big)
+W_q^{\prime}(y)+\gamma\int_0^{y-b}W_q^{\prime}(y-z)W_{q+\gamma}(z)\mathrm{d}z\Big]\mathbf{1}_{(b,\infty)}(y)\mathrm{d}y\Bigg].
\end{eqnarray}
We claim that \eqref{g.b.1} remains valid when $X$ has paths of unbounded variation. Actually, by (9) of \cite{Albrecher16} and the strong Markov property, we have
\begin{eqnarray}\label{g.b.unb}
g(b)\hspace{-0.3cm}&=&\hspace{-0.3cm}\mathrm{P}_b\big(U_{e_q}^{0,b}\in\mathrm{d}y;\kappa_x^-<e_q\wedge T_1\big)+\mathrm{P}_b\big(U_{e_q}^{0,b}\in\mathrm{d}y;e_q<\kappa_x^-\wedge T_1\big)+\mathrm{P}_b\big(U_{e_q}^{0,b}\in\mathrm{d}y;T_1<\kappa_x^-\wedge e_{q}\big)
\nonumber\\
\hspace{-0.3cm}&=&\hspace{-0.3cm}
\mathrm{E}_b\Big[e^{-(q+\gamma)\tau_x^-}\Big]g(x)+\frac{q}{q+\gamma}\mathrm{P}_b\Big[X_{e_{q+\gamma}}\in\mathrm{d}y;e_{q+\gamma}<\tau_x^-\Big]
\nonumber\\
\hspace{-0.3cm}&&\hspace{-0.3cm}
+\frac{\gamma}{q+\gamma}g(b)\Big[\mathrm{P}_b\big(e_{q+\gamma}<\tau_x^-\big)-\mathrm{P}_b\big(X_{e_{q+\gamma}}\in(x,b);\,e_{q+\gamma}<\tau_x^-\big)\Big]
\nonumber\\
\hspace{-0.3cm}&&\hspace{-0.3cm}
+\frac{\gamma}{q+\gamma}\int_x^bg(z)\mathrm{P}_b\left(X_{e_{q+\gamma}}\in\mathrm{d}z;\,e_{q+\gamma}<\tau_x^-\right)
\nonumber\\
\hspace{-0.3cm}&=&\hspace{-0.3cm}
e^{-\Phi_{q+\gamma}(b-x)}g(x)+q\left(e^{-\Phi_{q+\gamma}(b-x)}W_{q+\gamma}(y-x)-W_{q+\gamma}(y-b)\right)\mathbf{1}_{(x,\infty)}(y)\mathrm{d}y
\nonumber\\
\hspace{-0.3cm}&&\hspace{-0.3cm}
+\frac{\gamma}{q+\gamma}g(b)\left[1-e^{-\Phi_{q+\gamma}(b-x)}-(q+\gamma)\int_0^{b-x}e^{-\Phi_{q+\gamma}(b-x)}W_{q+\gamma}(b-x-z)\mathrm{d}z\right]
\nonumber\\
\hspace{-0.3cm}&&\hspace{-0.3cm}
+{\gamma}\int_0^{b-x}e^{-\Phi_{q+\gamma}(b-x)}W_{q+\gamma}(b-x-z)g(b-z)\mathrm{d}z,\quad x\in(0,b).
\end{eqnarray}
Plugging (\ref{g.b.unb}) into (\ref{3.23}), using \eqref{3.17}, and then rearranging the yielding equation gives
\begin{eqnarray}\label{g.x.unb}
\hspace{-0.3cm}&&\hspace{-0.3cm}
g(x)\left[1-\frac{\gamma Z_q(b-x)+qZ_q(b-x,\Phi_{q+\gamma})-q\Phi_{q+\gamma}Z_q(b,\Phi_{q+\gamma})\frac{W_q(b-x)}{W_q^{\prime+}(b)}}{qe^{\Phi_{q+\gamma}(b-x)}+\gamma Z_{q+\gamma}(b-x)}\right]
\nonumber\\
\hspace{-0.3cm}&=&\hspace{-0.3cm}
q\frac{W_q(b-x)W_q(0+)}{W_q^{\prime+}(b)}\delta_0(\mathrm{d}y)+q\Big(W_q(b-x)\frac{W_q^{\prime+}(y)}{W_q^{\prime+}(b)}-W_q(y-x)\Big)\mathbf{1}_{(0,b)}(y)\mathrm{d}y
\nonumber\\
\hspace{-0.3cm}&&\hspace{-0.3cm}
+q\frac{W_q(b-x)}{W_q^{\prime+}(b)}\Big[W_{q+\gamma}(y-b)\Big(\gamma W_q(b)-\Phi_{q+\gamma}Z_q(b,\Phi_{q+\gamma})\Big)+W_q^{\prime+}(y)
\nonumber\\
\hspace{-0.3cm}&&\hspace{-0.3cm}
+\gamma\int_0^{y-b}W_q^{\prime+}(y-z)W_{q+\gamma}(z)\mathrm{d}z\Big]\mathbf{1}_{(b,\infty)}(y)\mathrm{d}y
+\left(q+\gamma e^{-\Phi_{q+\gamma}(b-x)}Z_q(b-x)\right)^{-1}
\nonumber\\
\hspace{-0.3cm}&&\hspace{-0.3cm}
\times\bigg[q\Big(q\gamma\int_0^{b-x}W_{q+\gamma}(z+y-b)W_q(b-x-z)\mathrm{d}z
\nonumber\\
\hspace{-0.3cm}&&\hspace{-0.3cm}
+\gamma^2Z_q(b-x)e^{-\Phi_{q+\gamma}(b-x)}\int_0^{b-x}W_{q+\gamma}(z+y-b)W_q(b-x-z)\mathrm{d}z
\nonumber\\
\hspace{-0.3cm}&&\hspace{-0.3cm}
-\gamma^2Z_q(b-x)W_{q+\gamma}(y-b)\int_0^{b-x}e^{-\Phi_{q+\gamma}z}W_q(z)\mathrm{d}z
\nonumber\\
\hspace{-0.3cm}&&\hspace{-0.3cm}
-q\gamma W_{q+\gamma}(y-x)\int_0^{b-x}e^{-\Phi_{q+\gamma}z}W_q(z)\mathrm{d}z
\Big)\mathbf{1}_{(b,\infty)}(y)\mathrm{d}y\bigg]
\nonumber\\
\hspace{-0.3cm}&&\hspace{-0.3cm}
+\frac{\gamma Z_q(b-x)+qZ_q(b-x,\Phi_{q+\gamma})-q\Phi_{q+\gamma}Z_q(b,\Phi_{q+\gamma})\frac{W_q(b-x)}{W_q^{\prime+}(b)}}{q+\gamma e^{-\Phi_{q+\gamma}(b-x)}Z_q(b-x)}
\nonumber\\
\hspace{-0.3cm}&&\hspace{-0.3cm}
\times e^{-\Phi_{q+\gamma}(b-x)}\int_0^{b-x}\left(q\omega(b-z)+\gamma g(b-z)\right)W_{q+\gamma}(b-x-z)\mathrm{d}z.
\end{eqnarray}
Recall that $W_{q}(0+)=0$ in case $X$ has paths of unbounded variation. Hence, for any $\theta\geq0$ and any bounded function $h(y)$, one can verify that
\begin{eqnarray}\label{lim.1}
\lim\limits_{x\uparrow b}\frac{\int_0^{b-x}e^{-\theta z}W_q(z)\mathrm{d}z}{W_q(b-x)}
=\lim\limits_{x\uparrow b}\frac{\int_0^{b-x}W_{q+\gamma}(z)\mathrm{d}z}{W_q(b-x)}=\lim\limits_{x\uparrow b}\frac{\int_0^{b-x}h(z)W_{q+\gamma}(z)\mathrm{d}z}{W_q(b-x)}
=0,
\end{eqnarray}
where, in the last equality of \eqref{lim.1} we have used the fact that
\begin{eqnarray}
\label{lim.2}
\lim\limits_{x\uparrow \,b^{-}}\frac{W_{q+\gamma}(b-x)}{W_q(b-x)}=1,
\end{eqnarray}
which can be achieved by setting $y=0$ and $p=q+\gamma$ in \eqref{3.17}.
Dividing the both sides of (\ref{g.x.unb}) by $W_q(b-x)$, sending $x$ upward to $b$, and then using (\ref{lim.1}), we finally find that \eqref{g.b.1} also holds true for the case when $X$ has paths of unbounded variation.
Plugging (\ref{g.b.1}) into \eqref{2.11} and \eqref{3.23} yields \eqref{pot.mea}.
The proof is complete.
\end{proof}

With the preparations made in Lemma \ref{lem.D.R} and Lemma \ref{lem.w}, we are now ready to give an expression of the value function, denoted as $V_{0,b}^{\omega}$, of a double barrier periodic dividend and capital injection strategy in the following Lemma \ref{V.x}. Since Lemma \ref{V.x} is a direct consequence of Lemma \ref{lem.D.R}, Lemma \ref{lem.w} as well as the integrating by parts formula, we omit its proof.

\begin{lem}\label{V.x}
For $q>0$, and $\lambda>0$, we have
\begin{eqnarray}
V_{0,b}^{\omega}(x)
=
\begin{cases}
&\hspace{-0.3cm}-\frac{\gamma}{q+\gamma}\left[\overline{Z}_q(b-x)+\frac{\psi^{\prime}(0+)}{q}\right]
+
\frac{\left(\gamma Z_q(b)-\phi(q+\gamma)\right)\left[Z_q(b-x,\Phi_{q+\gamma})+\frac{\gamma}{q}Z_q(b-x)\right]}{(q+\gamma)\Phi_{q+\gamma}Z_q(b,\Phi_{q+\gamma})}
\\&\hspace{-0.3cm}
+\frac{\lambda\omega(b)}{q}Z_q(b-x)-\lambda\int_0^b\omega(y)W_q(y-x)\mathrm{d}y-\frac{qZ_q(b-x,\Phi_{q+\gamma})+\gamma Z_q(b-x)}{q\Phi_{q+\gamma}Z_q(b,\Phi_{q+\gamma})}
\\
&\hspace{-0.3cm}
\times\bigg[\lambda\int_0^b\omega^{\prime}_+(y)W_q(y)\mathrm{d}y+\lambda\int_0^{\infty}\omega^{\prime}_+(b+y)\Big[\Big(W_q(b+y)+\gamma\int_0^yW_q(b+y-z)W_{q+\gamma}(z)\mathrm{d}z\Big)
\\
&\hspace{-0.3cm}
-\frac{q\Phi_{q+\gamma}Z_q(b,\Phi_{q+\gamma})}{qZ_q(b-x,\Phi_{q+\gamma})+\gamma Z_q(b-x)}\Big(\frac{1}{q}Z_q(b-x+y)
\\
&\hspace{-0.3cm}
+\frac{\gamma}{q}\int_0^yZ_q(b-x+y-z)W_{q+\gamma}(z)\mathrm{d}z\Big)\Big]\mathrm{d}y\bigg],\quad x\in[0,\infty),
\\
&\hspace{-0.3cm}
\phi x+V_{0,b}^{\omega}(0),\quad x\in(-\infty,0).
\end{cases}\nonumber
\end{eqnarray}
\end{lem}

The following Lemma \ref{lem.par.b} characterizes the smooth conditions of the value function $V_{0,b}^{\omega}(x)$.

\begin{lem}\label{lem.par.b}
For each $b\in(0,\infty)$, the function $V_{0,b}^{\omega}(x)$ is continuously differentiable over $(-\infty,\infty)$. Furthermore, if $X$ has paths of unbounded variation, $V_{0,b}^{\omega}(x)$ is twice continuously differentiable over $(0,\infty)$.
\end{lem}
\begin{proof}
By Lemma \ref{V.x}, one can readily derive that
\begin{eqnarray}\label{Vx.p.1}
V_{0,b}^{\omega\prime}(x)\hspace{-0.3cm}&=&\hspace{-0.3cm}
\frac{\gamma}{q+\gamma}Z_q(b-x)-\frac{\gamma Z_q(b)-\phi(q+\gamma)}{(q+\gamma)Z_q(b,\Phi_{q+\gamma})}Z_q(b-x,\Phi_{q+\gamma})-\lambda\int_{0+}^bW_q(y-x)\omega^{\prime}_+(y)\mathrm{d}y
\nonumber\\
\hspace{-0.3cm}&&\hspace{-0.3cm}
+\frac{Z_q(b-x,\Phi_{q+\gamma})}{Z_q(b,\Phi_{q+\gamma})}\bigg[\lambda\int_{0+}^bW_q(y)\omega^{\prime}_+(y)\mathrm{d}y+\lambda\int_0^{\infty}\omega^{\prime}_+(b+y)\Big[\Big(W_{q}(b+y)
\nonumber\\
\hspace{-0.3cm}&&\hspace{-0.3cm}
+\gamma\int_0^yW_q(b+y-z)W_{q+\gamma}(z)\mathrm{d}z\Big)-\frac{Z_q(b,\Phi_{q+\gamma})}{Z_q(b-x,\Phi_{q+\gamma})}\Big(W_{q}(b-x+y)
\nonumber\\
\hspace{-0.3cm}&&\hspace{-0.3cm}
+\gamma\int_0^yW_q(b-x+y-z)W_{q+\gamma}(z)\mathrm{d}z\Big)\Big]\mathrm{d}y
\bigg],\quad x\in(0,b),
\end{eqnarray}
and
\begin{eqnarray}\label{Vx.p.2}
V_{0,b}^{\omega\prime}(x)
\hspace{-0.3cm}&=&\hspace{-0.3cm}
\frac{\gamma}{q+\gamma}-\frac{\gamma Z_q(b)-\phi(q+\gamma)}{(q+\gamma)Z_q(b,\Phi_{q+\gamma})}e^{-\Phi_{q+\gamma}(x-b)}
+\frac{e^{-\Phi_{q+\gamma}(x-b)}}{Z_q(b,\Phi_{q+\gamma})}\bigg[\lambda\int_{0+}^bW_q(y)\omega^{\prime}_+(y)\mathrm{d}y
\nonumber\\
\hspace{-0.3cm}&&\hspace{-0.3cm}
+\lambda\int_0^{\infty}\omega^{\prime}_+(b+y)\Big[W_q(b+y)
+\gamma\int_0^yW_q(b+y-z)W_{q+\gamma}(z)\mathrm{d}z-\frac{Z_q(b,\Phi_{q+\gamma})}{e^{\Phi_{q+\gamma}(b-x)}}
\nonumber\\
\hspace{-0.3cm}&&\hspace{-0.3cm}
\times\Big(W_{q+\gamma}(b-x+y)
-\gamma\int_0^{b-x}W_{q+\gamma}(y+z)W_q(b-x-z)\mathrm{d}z\Big)\Big]\mathrm{d}y\bigg]
,\quad x\in(b,\infty).
\end{eqnarray}
Combining (\ref{3.17}), (\ref{lim.2}), (\ref{Vx.p.1}) and (\ref{Vx.p.2}), we can find
\begin{eqnarray}
V_{0,b}^{\omega\prime}(b+)-V_{0,b}^{\omega\prime}(b-)\hspace{-0.3cm}&=&\hspace{-0.3cm}
{\lambda}\int_0^{\infty}\omega^{\prime}_+(b+y)\Big(W_{q}(y)+\gamma\int_0^yW_q(y-z)W_{q+\gamma}(z)-W_{q+\gamma}(y) \mathrm{d}z\Big)\mathrm{d}y
=0,\nonumber
\end{eqnarray}
where the fact that $\lim_{x\uparrow b}\int_{0+}^b\omega^{\prime}_+(y)W_q(y-x)\mathrm{d}y=\lim_{x\uparrow b}\int_{x}^b\omega^{\prime}_+(y)W_q(y-x)\mathrm{d}y=0$ was used.
In addition, it is easy to get that
\begin{eqnarray}
V^{\omega\prime}_{0,b}(0+)=\phi=V^{\omega\prime}_{0,b}(0-).\nonumber
\end{eqnarray}
The above arguments imply the continuous differentiability of $V_{0,b}^{\omega}(x)$ over $(-\infty,\infty)$.

Furthermore, when $X$ has paths of unbounded
variation, the scale function $W_q$ is continuously differentiable, and hence, we have
\begin{eqnarray}\label{Vx.pp.1}
V_{0,b}^{\omega\prime\prime}(x)
\hspace{-0.3cm}&=&\hspace{-0.3cm}
-\frac{q\gamma}{q+\gamma}W_q(b-x)+\frac{\gamma Z_q(b)-\phi(q+\gamma)}{(q+\gamma)Z_q(b,\Phi_{q+\gamma})}\left[\Phi_{q+\gamma}Z_q(b-x,\Phi_{q+\gamma})-\gamma W_q(b-x)\right]
\nonumber\\
\hspace{-0.3cm}&&\hspace{-0.3cm}
+\int_0^b\omega^{\prime}_+(y)W_q^{\prime+}(y-x)\mathrm{d}y-\frac{\Phi_{q+\gamma}Z_q(b-x,\Phi_{q+\gamma})-\gamma W_q(b-x)}{Z_q(b,\Phi_{q+\gamma})}
\bigg[\lambda\int_{0+}^bW_q(y)\omega^{\prime}_+(y)\mathrm{d}y
\nonumber\\
\hspace{-0.3cm}&&\hspace{-0.3cm}
+\lambda\int_0^{\infty}\omega^{\prime}_+(b+y)\Big[\Big(W_{q}(b+y)+\gamma\int_0^yW_q(b+y-z)W_{q+\gamma}(z)\mathrm{d}z\Big)
\nonumber\\
\hspace{-0.3cm}&&\hspace{-0.3cm}
-\frac{Z_q(b,\Phi_{q+\gamma})}{\Phi_{q+\gamma}Z_q(b-x,\Phi_{q+\gamma})-\gamma W_q(b-x)}\Big(W_{q}^{\prime}(b-x+y)
\nonumber\\
\hspace{-0.3cm}&&\hspace{-0.3cm}
+\gamma\int_0^yW_q^{\prime+}(b-x+y-z)W_{q+\gamma}(z)\mathrm{d}z\Big)\Big]\bigg],\quad x\in(0,b),
\end{eqnarray}
and
\begin{eqnarray}\label{Vx.pp.2}
V_{0,b}^{\omega\prime\prime}(x)
\hspace{-0.3cm}&=&\hspace{-0.3cm}
\Phi_{q+\gamma}\frac{\gamma Z_q(b)-\phi(q+\gamma)}{(q+\gamma)Z_q(b,\Phi_{q+\gamma})}e^{-\Phi_{q+\gamma}(x-b)}
-\frac{\Phi_{q+\gamma}e^{-\Phi_{q+\gamma}(x-b)}}{Z_q(b,\Phi_{q+\gamma})}\bigg[\lambda\int_{0+}^bW_q(y)\omega^{\prime}_+(y)\mathrm{d}y
\nonumber\\
\hspace{-0.3cm}&&\hspace{-0.3cm}
+\lambda\int_0^{\infty}\omega^{\prime}_+(b+y)\Big[W_q(b+y)
+\gamma\int_0^yW_q(b+y-z)W_{q+\gamma}(z)\mathrm{d}z-\frac{Z_q(b,\Phi_{q+\gamma})}{\Phi_{q+\gamma}e^{\Phi_{q+\gamma}(x-b)}}
\nonumber\\
\hspace{-0.3cm}&&\hspace{-0.3cm}
\times\Big(W_{q}^{\prime+}(b-x+y)
+\gamma\int_0^yW^{\prime+}_q(b-x+y-z)W_{q+\gamma}(z)\mathrm{d}z\Big)\Big]\mathrm{d}y\bigg],\quad x\in(b,\infty).
\end{eqnarray}
Combining (\ref{Vx.pp.1}), (\ref{Vx.pp.2}), we have
\begin{eqnarray}
V_{0,b}^{\omega\prime\prime}(b+)-V_{0,b}^{\omega\prime\prime}(b-)\hspace{-0.3cm}&=&\hspace{-0.3cm}
\frac{q\gamma}{q+\gamma}W_q(0+)+\frac{\gamma Z_q(b)-\phi(q+\gamma)}{(q+\gamma)Z_q(b,\Phi_{q+\gamma})}\gamma W_q(0+)-\frac{\gamma W_q(0+)}{Z_q(b,\Phi_{q+\gamma})}
\nonumber\\
\hspace{-0.3cm}&&\hspace{-3.3cm}
\times\left[\lambda\int_{0+}^bW_q(y)\omega^{\prime}_+(y)\mathrm{d}y+\lambda\int_0^{\infty}\omega^{\prime}_+(b+y)\Big(W_{q}(b+y)+\gamma\int_0^yW_q(b+y-z)W_{q+\gamma}(z)\mathrm{d}z\Big)\right]
\nonumber\\
\hspace{-0.3cm}&&\hspace{-3.3cm}=0.\nonumber
\end{eqnarray}
As a result, $V_{0,b}^{\omega}$ is twice continuously differentiable over $(0,\infty)$ when X has paths of unbounded variation. This completes the proof.
\end{proof}

Bearing in mind that the optimal strategy solving the auxiliary problem \eqref{valefunctionofpi} is conjectured to be some double barrier periodic dividend and capital injection strategy, we hence need to characterize the barrier level corresponding to the optimal strategy among the set of double barrier strategies. To achieve this goal, we manage to express the derivative of the value function in terms of the Laplace transform and potential measure of the spectrally positive L\'evy process controlled by a single barrier periodic dividend strategy until the first time it down-crosses $0$.
Actually, this new compact expression will facilitate us to identify the candidate optimal strategy among the set of double barrier dividend and capital injection strategies and the slope conditions of the value function of this candidate optimal double barrier strategy.
To implement these ideas, we should first provide some preliminary results concerning the spectrally positive L\'evy process controlled by a single barrier periodic dividend strategy. Put
$$U_t^b:=X_t-D_t^b,\quad t\geq 0,$$
where $X_{t}$ is the single spectrally positive L\'evy process, $D_t^b$ is the cumulative dividend process defined as
$$D_t^b=\sum_{n=1}^{\infty}\left((U_{T_n-}^b+\Delta X_{T_n}-b)\right)\mathbf{1}_{\{T_n\leq t\}},\quad  t\geq0,$$
with $(T_{n})_{n\geq1}$ being the event epochs of an independent Poisson process.
In addition, denote the down-crossing and up-crossing times of $U_t^b$, respectively, by $$\overline{\kappa}_a^-:=\inf\{t\geq0;U_t^b<a\}\quad\text{and}\quad\overline{\kappa}_c^+:=\inf\{t\geq0;U_t^b>c\}.$$
The following Lemma \ref{lem.pm} gives the potential measure of the process $U_t^b$ until the first time it down-crosses $0$. It seems a little bit unexpected that this problem has not been considered in the literature. We hence provide this result with a detailed proof.

\begin{lem}\label{lem.pm}
For $q>0$, $\gamma>0$ and $b>0$, we have
\begin{eqnarray}\label{U.b}
\hspace{-1cm}\mathrm{P}_x\left(U_{e_q}^{b}\in\mathrm{d}y;e_q<\overline{\kappa}_0^-\right)
\hspace{-0.3cm}&=&\hspace{-0.3cm}
q\bigg[Z_q(b-x,\Phi_{q+\gamma})W_{q+\gamma}(y-b)-W_{q+\gamma}(y-x)
\nonumber\\
\hspace{-0.3cm}&&\hspace{-4cm}
+\gamma\int_0^{b-x}W_q(b-x-z)W_{q+\gamma}(y-b+z)\mathrm{d}z\bigg]\mathbf{1}_{(b,\infty)}(y)\mathrm{d}y-qW_q(y-x)\mathbf{1}_{(0,b)}(y)\mathrm{d}y
\nonumber\\
\hspace{-0.3cm}&&\hspace{-4cm}
+\frac{\gamma Z_q(b-x)+qZ_q(b-x,\Phi_{q+\gamma})}{\gamma Z_q(b)+qZ_q(b,\Phi_{q+\gamma})}\bigg[qW_q(y)\mathbf{1}_{(0,b)}(y)\mathrm{d}y+q\Big(W_{q+\gamma}(y)
\nonumber\\
\hspace{-0.3cm}&&\hspace{-4cm}
-Z_q(b,\Phi_{q+\gamma})W_{q+\gamma}(y-b)-\gamma\int_0^bW_q(b-z)W_{q+\gamma}(y-b+z)\mathrm{d}z\Big)\mathbf{1}_{(b,\infty)}(y)\mathrm{d}y\bigg], \quad x\in(0,\infty).
\end{eqnarray}
\end{lem}
\begin{proof}
Denote by $\overline{g}(x)$ the left hand side of (\ref{U.b}). One can verify that
\begin{eqnarray}
\hspace{-0.3cm}\overline{g}(x)\hspace{-0.3cm}&=&
\hspace{-0.3cm}
\mathrm{P}_{x}\left(U^{b}_{e_q}\in\mathrm{d}y;{e_q<T_{1}\wedge \overline{\kappa}_{b}^{-}}\right)+\mathrm{P}_{x}\left(U^{b}_{e_q}\in\mathrm{d}y;e_q>T_{1}\wedge\overline{\kappa}_{b}^{-},e_q<\overline{\kappa}_0^-\right)
\nonumber\\
\hspace{-0.3cm}&=&
\hspace{-0.3cm}
\frac{q}{q+\gamma}\mathrm{P}_{x}\left(X_{e_{q+\gamma}}\in\mathrm{d}y;{e_{q+\gamma}< \tau_{b}^{-}}\right)+\mathrm{P}_x\Big(U^b_{e_q}\in\mathrm{d}y;e_q>T_1,\overline{\kappa}_b^->T_1,e_q<\overline{\kappa}_0^-\Big)
\nonumber\\
\hspace{-0.3cm}&&\hspace{-0.3cm}
+\mathrm{P}_x\Big(U^b_{e_q}\in\mathrm{d}y;e_q>\overline{\kappa}_b^-,T_1>\overline{\kappa}_b^-,e_q<\overline{\kappa}_0^-\Big)
\nonumber\\
\hspace{-0.3cm}&=&
\hspace{-0.3cm}
\frac{q}{q+\gamma}\mathrm{P}_{x}\left(X_{e_{q+\gamma}}\in\mathrm{d}y;{e_{q+\gamma}< \tau_{b}^{-}}\right)
+\left[\mathrm{E}_{x}\left[e^{-(q+\gamma)\tau_{b}^{-}}\right]
+\frac{\gamma}{q+\gamma}\mathrm{E}_{x}\left[1-e^{-(q+\gamma) \tau_{b}^{-}}\right]\right]\overline{g}(b)
\nonumber\\
\hspace{-0.3cm}&=&\hspace{-0.3cm}
\frac{q}{q+\gamma}\mathrm{P}_{x}\left(X_{e_{q+\gamma}}\in\mathrm{d}y;{e_{q+\gamma}< \tau_{b}^{-}}\right)+\left[\frac{\gamma}{q+\gamma}+\frac{q}{q+\gamma}\mathrm{E}_{x}\left[e^{-(q+\gamma) \tau_{b}^{-}}\right]\right]\overline{g}(b)
\nonumber\\
\hspace{-0.3cm}&=&\hspace{-0.3cm}
q\left[e^{\Phi_{q+\gamma}(b-x)}W_{q+\gamma}(y-b)-W_{q+\gamma}(y-x)\right]\mathbf{1}_{(b,\infty)}(y)\mathrm{d}y
\nonumber\\
\hspace{-0.3cm}&&\hspace{-0.3cm}
+\frac{\gamma}{q+\gamma}\overline{g}(b)+\frac{q}{q+\gamma}e^{\Phi_{q+\gamma}(b-x)}\overline{g}(b),\quad x\in(b,\infty).\nonumber
\end{eqnarray}
Using a similar manner and \eqref{fluc.2}, one can verify that
\begin{eqnarray}\label{U.b.2}
\overline{g}(x)
\hspace{-0.3cm}&=&\hspace{-0.3cm}
\mathrm{P}_{x}\Big(U^{b}_{e_q}\in\mathrm{d}y;e_q< \overline{\kappa}_{b}^{+}\wedge\overline{\kappa}_0^-\Big)+\mathrm{P}_{x}\Big(U^{b}_{e_q}\in\mathrm{d}y;\overline{\kappa}_b^+<e_q<\overline{\kappa}_{0}^{-}\Big)
\nonumber\\
\hspace{-0.3cm}&=&
\hspace{-0.3cm}
\mathrm{P}_{x}\Big(X_{e_q}\in\mathrm{d}y;e_q< \tau_{b}^{+}\wedge\tau_0^-\Big)+\mathrm{E}_x\big[e^{-q\tau_b^+}\mathbf{1}_{\{\tau_b^+<\tau_0^-\}}\overline{g}(X_{\tau_b^+})\big]
\nonumber\\
\hspace{-0.3cm}&=&
\hspace{-0.3cm}
q\left[\frac{W_q(b-x)}{W_q(b)}W_q(y)-W_q(y-x)\right]\mathbf{1}_{(0,b)}(y)\mathrm{d}y+\frac{\gamma}{q+\gamma}\mathrm{E}_x\Big[e^{-q\tau_b^+}\mathbf{1}_{\{\tau_b^+<\tau_0^-\}}\Big]\overline{g}(b)
\nonumber\\
\hspace{-0.3cm}&&\hspace{-0.3cm}
+q\bigg[\mathrm{E}_x\left[e^{-q\tau_b^+}e^{\Phi_{q+\gamma}(b-X_{\tau_b^+})}\mathbf{1}_{\{\tau_b^+<\tau_0^-\}}\right]W_{q+\gamma}(y-b)
\nonumber\\
\hspace{-0.3cm}&&\hspace{-0.3cm}
-\mathrm{E}_x\left[e^{-q\tau_b^+}W_{q+\gamma}(y-X_{\tau_b^+})\mathbf{1}_{\{\tau_b^+<\tau_0^-\}}\right]\bigg]\mathbf{1}_{(b,\infty)}(y)\mathrm{d}y
\nonumber\\
\hspace{-0.3cm}&&\hspace{-0.3cm}
+\frac{q}{q+\gamma}\mathrm{E}_x\left[e^{-\tau_b^+}e^{\Phi_{q+\gamma}(b-X_{\tau_b^+})}\mathbf{1}_{\{\tau_b^+<\tau_0^-\}}\right]\overline{g}(b)
\nonumber\\
\hspace{-0.3cm}&=&\hspace{-0.3cm}
q\bigg[\left[Z_q(b-x,\Phi_{q+\gamma})-W_q(b-x)\frac{Z_q(b,\Phi_{q+\gamma}) }{W_q(b)}\right]W_{q+\gamma}(y-b)-W_{q+\gamma}(y-x)
\nonumber\\
\hspace{-0.3cm}&&\hspace{-0.3cm}
+\gamma\int_0^{b-x}W_q(b-x-z)W_{q+\gamma}(y-b+z)\mathrm{d}z+\frac{W_q(b-x)}{W_q(b)}
\nonumber\\
\hspace{-0.3cm}&&\hspace{-0.3cm}
\times\left[W_{q+\gamma}(y)-\gamma\int_0^{b}W_q(b-z)W_{q+\gamma}(y-b+z)\mathrm{d}z\right]\bigg]\mathbf{1}_{(b,\infty)}(y)\mathrm{d}y
\nonumber\\
\hspace{-0.3cm}&&\hspace{-0.3cm}
+\bigg[\frac{\gamma Z_q(b-x)+qZ_q(b-x,\Phi_{q+\gamma})}{q+\gamma}-\frac{\gamma Z_q(b)+qZ_q(b,\Phi_{q+\gamma})}{q+\gamma}\frac{W_q(b-x)}{W_q(b)}\bigg]\overline{g}(b)\nonumber\\
\hspace{-0.3cm}&&\hspace{-0.3cm}
+q\left[\frac{W_q(b-x)}{W_q(b)}W_q(y)-W_q(y-x)\right]\mathbf{1}_{(0,b)}(y)\mathrm{d}y,\quad x\in(0,b).
\end{eqnarray}
When $X$ has paths of bounded variation, letting $x\uparrow b$ in (\ref{U.b.2}) yields
\begin{eqnarray}
\label{gb}
\overline{g}(b)
\hspace{-0.3cm}&=&\hspace{-0.3cm}
\left[\frac{\gamma Z_q(b)+qZ_q(b,\Phi_{q+\gamma})}{q+\gamma}\right]^{-1}\bigg[qW_q(y)\mathbf{1}_{(0,b)}(y)\mathrm{d}y+q\Big(W_{q+\gamma}(y)
\nonumber\\
\hspace{-0.3cm}&&\hspace{-0.3cm}
-Z_q(b,\Phi_{q+\gamma})W_{q+\gamma}(y-b)-\gamma\int_0^bW_q(b-z)W_{q+\gamma}(y-b+z)\mathrm{d}z\Big)\mathbf{1}_{(b,\infty)}(y)\mathrm{d}y\bigg].
\end{eqnarray}
We now claim that \eqref{gb} remains valid even when $X$ has paths of unbounded variation. Using a method similar to that of the proof of Lemma \ref{lem.w}, we get
\begin{eqnarray}\label{g.bar.b}
\overline{g}(b)
\hspace{-0.3cm}&=&\hspace{-0.3cm}
\mathrm{P}_b\Big(U_{e_q}^b\in\mathrm{d}y;e_q<T_1\wedge\overline{\kappa}_x^-\Big)+\mathrm{P}_b\Big(U_{e_q}^b\in\mathrm{d}y;\overline{\kappa}_x^-<T_1\wedge e_q,e_q<\overline{\kappa}_0^-\Big)
\nonumber\\
\hspace{-0.3cm}&&\hspace{-0.3cm}
+\mathrm{P}_b\Big(U_{e_q}^b\in\mathrm{d}y;T_1<\overline{\kappa}_x^-\wedge e_q,e_q<\overline{\kappa}_0^-\Big)
\nonumber\\
\hspace{-0.3cm}&=&\hspace{-0.3cm}
\frac{q}{q+\gamma}\mathrm{P}_b\Big(X_{e_{q+\gamma}}\in\mathrm{d}y;e_{q+\gamma}<\tau_x^-\Big)+\mathrm{E}_b\Big[e^{-(q+\gamma)\tau_x^-}\Big]\overline{g}(x)
\nonumber\\
\hspace{-0.3cm}&&\hspace{-0.3cm}
+\frac{\gamma}{q+\gamma}\Big[\mathrm{P}_b\big(e_{q+\gamma}<\tau_x^-\big)-\mathrm{P}_b\big(X_{e_{q+\gamma}}\in(x,b);\,e_{q+\gamma}<\tau_x^-\big)\Big]\overline{g}(b)
\nonumber\\
\hspace{-0.3cm}&&\hspace{-0.3cm}
+\frac{\gamma}{q+\gamma}\int_x^b\overline{g}(y)\mathrm{P}_b\left(X_{e_{q+\gamma}}\in\mathrm{d}y;\,e_{q+\gamma}<\tau_x^-\right)
\nonumber\\
\hspace{-0.3cm}&=&\hspace{-0.3cm}
q\Big(e^{-\Phi_{q+\gamma}(b-x)}W_{q+\gamma}(y-x)-W_{q+\gamma}(y-b)\Big)\mathbf{1}_{(x,\infty)}(y)\mathrm{d}y+e^{-\Phi_{q+\gamma}(b-x)}\overline{g}(x)
\nonumber\\
\hspace{-0.3cm}&&\hspace{-0.3cm}
+\frac{\gamma}{q+\gamma}\overline{g}(b)\left[1-e^{-\Phi_{q+\gamma}(b-x)}-(q+\gamma)\int_0^{b-x}e^{-\Phi_{q+\gamma}(b-x)}W_{q+\gamma}(b-x-y)\mathrm{d}y\right]
\nonumber\\
\hspace{-0.3cm}&&\hspace{-0.3cm}
+{\gamma}\int_0^{b-x}e^{-\Phi_{q+\gamma}(b-x)}W_{q+\gamma}(b-x-y)\overline{g}(b-y)\mathrm{d}y,\quad x\in(0,b).
\end{eqnarray}
Plugging (\ref{g.bar.b}) into (\ref{U.b.2}), and then rearranging the yielding equation gives
\begin{eqnarray}\label{g.bar.x}
\hspace{-0.3cm}&&\hspace{-0.3cm}
\overline{g}(x)\left[1-\frac{\gamma Z_q(b-x)+qZ_q(b-x,\Phi_{q+\gamma})-\Big(\gamma Z_q(b)+qZ_q(b,\Phi_{q+\gamma})\Big)\frac{W_q(b-x)}{W_q(b)}}{qe^{\Phi_{q+\gamma}(b-x)}+\gamma Z_{q+\gamma}(b-x)}\right]
\nonumber\\
\hspace{-0.3cm}&=&\hspace{-0.3cm}
q\left[\frac{W_q(b-x)}{W_q(b)}W_q(y)-W_q(y-x)\right]\mathbf{1}_{(0,b)}(y)\mathrm{d}y
+\frac{qW_q(b-x)}{W_q(b)}
\Big[W_{q+\gamma}(y)
\nonumber\\
\hspace{-0.3cm}&&\hspace{-0.3cm}
-Z_q(b,\Phi_{q+\gamma})W_{q+\gamma}(y-b)
-\gamma\int_0^{b}W_q(b-z)W_{q+\gamma}(y-b+z)\mathrm{d}z\Big]\mathbf{1}_{(b,\infty)}(y)\mathrm{d}y
\nonumber\\
\hspace{-0.3cm}&&\hspace{-0.3cm}
+q\bigg[Z_q(b-x,\Phi_{q+\gamma})W_{q+\gamma}(y-b)-W_{q+\gamma}(y-x)
\nonumber\\
\hspace{-0.3cm}&&\hspace{-0.3cm}
+\gamma\int_0^{b-x}W_q(b-x-z)W_{q+\gamma}(y-b+z)\mathrm{d}z\bigg]\mathbf{1}_{(b,\infty)}(y)\mathrm{d}y
\nonumber\\
\hspace{-0.3cm}&&\hspace{-0.3cm}
+\bigg[\frac{\gamma Z_q(b-x)+qZ_q(b-x,\Phi_{q+\gamma})}{q+\gamma}-\frac{\gamma Z_q(b)+qZ_q(b,\Phi_{q+\gamma})}{q+\gamma}\frac{W_q(b-x)}{W_q(b)}\bigg]
\nonumber\\
\hspace{-0.3cm}&&\hspace{-0.3cm}
\times\bigg[\frac{q+\gamma e^{-\Phi_{q+\gamma}(b-x)}Z_{q+\gamma}(b-x)}{q+\gamma}\bigg]^{-1}\times\Big[{\gamma}\int_0^{b-x}e^{-\Phi_{q+\gamma}(b-x)}W_{q+\gamma}(b-x-y)\overline{g}(b-y)\mathrm{d}y
\nonumber\\
\hspace{-0.3cm}&&\hspace{-0.3cm}
+q\Big(e^{-\Phi_{q+\gamma}(b-x)}W_{q+\gamma}(y-x)-W_{q+\gamma}(y-b)\Big)\mathbf{1}_{(x,\infty)}(y)\mathrm{d}y
\Big]
\nonumber\\
\hspace{-0.3cm}&=&\hspace{-0.3cm}
q\frac{W_q(b-x)}{W_q(b)}W_q(y)\mathbf{1}_{(0,b)}(y)\mathrm{d}y-qW_q(y-x)\mathbf{1}_{(x,b)}(y)\mathrm{d}y
+\frac{qW_q(b-x)}{W_q(b)}
\Big[W_{q+\gamma}(y)
\nonumber\\
\hspace{-0.3cm}&&\hspace{-0.3cm}
-Z_q(b,\Phi_{q+\gamma})W_{q+\gamma}(y-b)
-\gamma\int_0^{b}W_q(b-z)W_{q+\gamma}(y-b+z)\mathrm{d}z\Big]\mathbf{1}_{(b,\infty)}(y)\mathrm{d}y
\nonumber\\
\hspace{-0.3cm}&&\hspace{-0.3cm}
+\gamma\int_0^{b-x}W_q(b-x-z)W_{q+\gamma}(y-b+z)\mathrm{d}z\mathbf{1}_{(b,\infty)}(y)\mathrm{d}y
-\frac{W_q(b-x)}{W_q(b)}
\nonumber\\
\hspace{-0.3cm}&&\hspace{-0.3cm}
\times\frac{\gamma Z_q(b)+qZ_q(b,\Phi_{q+\gamma})}{q+\gamma e^{-\Phi_{q+\gamma}(b-x)}Z_{q+\gamma}(b-x)}
\Big[{\gamma}\int_0^{b-x}e^{-\Phi_{q+\gamma}(b-x)}W_{q+\gamma}(b-x-y)\overline{g}(b-y)\mathrm{d}y
\nonumber\\
\hspace{-0.3cm}&&\hspace{-0.3cm}
+q\Big(e^{-\Phi_{q+\gamma}(b-x)}W_{q+\gamma}(y-x)-W_{q+\gamma}(y-b)\Big)\mathbf{1}_{(x,\infty)}(y)\mathrm{d}y
\Big]
\nonumber\\
\hspace{-0.3cm}&&\hspace{-0.3cm}
+\frac{\gamma Z_q(b-x)+qZ_q(b-x,\Phi_{q+\gamma})}{q+\gamma e^{-\Phi_{q+\gamma}(b-x)}Z_{q+\gamma}(b-x)}\times\Big[{\gamma}\int_0^{b-x}e^{-\Phi_{q+\gamma}(b-x)}W_{q+\gamma}(b-x-y)\overline{g}(b-y)\mathrm{d}y
\nonumber\\
\hspace{-0.3cm}&&\hspace{-0.3cm}
+q\Big(e^{-\Phi_{q+\gamma}(b-x)}W_{q+\gamma}(y-x)-W_{q+\gamma}(y-b)\Big)\mathbf{1}_{(x,b)}(y)\mathrm{d}y
\Big]
\nonumber\\
\hspace{-0.3cm}&&\hspace{-0.3cm}
+q\bigg[\frac{\gamma e^{-\Phi_{q+\gamma}(b-x)}\left(\int_0^{b-x}qW_q(z)\mathrm{d}z-\int_0^{b-x}(q+\gamma)W_{q+\gamma}(z)\mathrm{d}z\right)}{q+\gamma e^{-\Phi_{q+\gamma}(b-x)}Z_{q+\gamma}(b-x)}W_{q+\gamma}(y-x)
\nonumber\\
\hspace{-0.3cm}&&\hspace{-0.3cm}
-\frac{q\int_0^{b-x}\gamma e^{-\Phi_{q+\gamma}z}W_q(z)\mathrm{d}z+q\int_0^{b-x}(q+\gamma)W_{q+\gamma}(z)\mathrm{d}z}{q+\gamma e^{-\Phi_{q+\gamma}(b-x)}Z_{q+\gamma}(b-x)}W_{q+\gamma}(y-b)\bigg]\mathbf{1}_{(b,\infty)}(y)\mathrm{d}y
.
\end{eqnarray}
Dividing the both sides of (\ref{g.bar.x}) with $W_q(b-x)$, sending $x$ upward to $b$, and then using (\ref{lim.1}), (\ref{lim.2}), we finally find that \eqref{gb} also holds true for the case when $X$ has paths of unbounded variation. Plugging \eqref{gb} into \eqref{U.b.2} yields \eqref{U.b}. The proof is complete.
\end{proof}

The following Lemma \ref{lem.stop} gives the Laplace transform of the first passage time of $0$ for the spectrally positive L\'evy process with dividends subtracted according to the single barrier periodic dividend strategy with barrier level $b$. As far as the authors know, this result is also new to the literature.

\begin{lem}\label{lem.stop}
For $q>0$, $\gamma>0$ and $b>0$, we have
\begin{eqnarray}\label{over.kappa}
\mathrm{E}_x\Big[e^{-q\overline{\kappa}_0^-}\Big]=\frac{\gamma Z_q(b-x)+qZ_q(b-x,\Phi_{q+\gamma})}{\gamma Z_q(b)+qZ_q(b,\Phi_{q+\gamma})}, \quad x\in(0,\infty).
\end{eqnarray}
\end{lem}
\begin{proof}
By \eqref{T1wed.tau.b-} and the strong Markov property, we have
\begin{eqnarray}\label{over.kappa.1}
\mathrm{E}_x\Big[e^{-q\overline{\kappa}_0^-}\Big]
\hspace{-0.3cm}&=&\hspace{-0.3cm}\mathrm{E}_x\Big[e^{-q\overline{\kappa}_0^-}\mathbf{1}_{\{T_1\wedge\overline{\kappa}_b^-<\overline{\kappa}_0^-\}}\Big]
=\mathrm{E}_x\Big[e^{-q(T_1\wedge \tau_b^-)}\Big]\mathrm{E}_b\Big[e^{-q\overline{\kappa}_0^-}\Big]
\nonumber\\
\hspace{-0.3cm}&=&\hspace{-0.3cm}
\left(\frac{\gamma}{q+\gamma}+\frac{q}{q+\gamma}e^{-\Phi_{q+\gamma}(x-b)}\right)\mathrm{E}_b\Big[e^{-q\overline{\kappa}_0^-}\Big],\quad x\in(b,\infty),
\end{eqnarray}
and
\begin{eqnarray}\label{stop.time.1}
\mathrm{E}_x\Big[e^{-q\overline{\kappa}_0^-}\Big]
\hspace{-0.3cm}&=&\hspace{-0.3cm}
\mathrm{E}_x\Big[e^{-q\overline{\kappa}_0^-}\mathbf{1}_{\{\overline{\kappa}_0^-<\overline{\kappa}_b^+\}}\Big]+\mathrm{E}_x\Big[e^{-q\overline{\kappa}_0^-}\mathbf{1}_{\{\overline{\kappa}_b^+<\overline{\kappa}_0^-\}}\Big]
\nonumber\\
\hspace{-0.3cm}&=&\hspace{-0.3cm}
\mathrm{E}_x\Big[e^{-q\tau_0^-}\mathbf{1}_{\{\tau_0^-<\tau_b^+\}}\Big]+\mathrm{E}_x\left[e^{-q\tau_b^+}\mathbf{1}_{\{\tau_b^+<\tau_0^-\}}\mathrm{E}_{X_{\tau_b^+}}\left[e^{-q\overline{\kappa}_0^-}\right]\right]
\nonumber\\
\hspace{-0.3cm}&=&\hspace{-0.3cm}
\frac{W_q(b-x)}{W_q(b)}+\mathrm{E}_b\Big[e^{-q\overline{\kappa}_0^-}\Big]\bigg[\frac{\gamma}{q+\gamma}\mathrm{E}_x\left[e^{-q\tau_b^+}\mathbf{1}_{\{\tau_b^+<\tau_0^-\}}\right]
\nonumber\\
\hspace{-0.3cm}&&\hspace{-0.3cm}
+\frac{q}{q+\gamma}\mathrm{E}_x\left[e^{-q\tau_b^+}e^{-\Phi_{q+\gamma}(X_{\tau_b^+}-b)}\mathbf{1}_{\{\tau_b^+<\tau_0^-\}}\right]\bigg]
\nonumber\\
\hspace{-0.3cm}&=&\hspace{-0.3cm}
\frac{W_q(b-x)}{W_q(b)}+\mathrm{E}_b\Big[e^{-q\overline{\kappa}_0^-}\Big]\bigg[\frac{\gamma}{q+\gamma}\left(Z_q(b-x)-\frac{Z_q(b)}{W_q(b)}W_q(b-x)\right)
\nonumber\\
\hspace{-0.3cm}&&\hspace{-0.3cm}
+\frac{q}{q+\gamma}\left(Z_q(b-x,\Phi_{q+\gamma})-\frac{W_q(b-x)}{W_q(b)}Z_q(b,\Phi_{q+\gamma})\right)\bigg],\quad x\in(0,b).
\end{eqnarray}
When $X$ has paths of bounded variation, letting $x\uparrow b$ in (\ref{stop.time.1}) yields
\begin{eqnarray}\label{stop.time.b}
\mathrm{E}_b\left[e^{-q\overline{\kappa}_0^-}\right]=\frac{q+\gamma}{qZ_q(b,\Phi_{q+\gamma})+\gamma Z_q(b)}.
\end{eqnarray}
We claim that (\ref{stop.time.b}) remains valid when $X$ has paths of unbounded variation. Actually, by the strong Markov property, we have
\begin{eqnarray}\label{stop.time.2}
\mathrm{E}_b\left[e^{-q\overline{\kappa}_0^-}\right]
\hspace{-0.3cm}&=&\hspace{-0.3cm}
\mathrm{E}_b\Big[e^{-q\overline{\kappa}_0^-}\mathbf{1}_{\{\overline{\kappa}_x^-<\overline{\kappa}_0^-\wedge T_1\}}\Big]+\mathrm{E}_b\Big[e^{-q\overline{\kappa}_0^-}\mathbf{1}_{\{T_1<\overline{\kappa}_x^-<\overline{\kappa}_0^-\}}\Big]
\nonumber\\
\hspace{-0.3cm}&=&\hspace{-0.3cm}
\mathrm{E}_b\Big[e^{-(q+\gamma)\tau_x^-}\Big]\mathrm{E}_x\left[e^{-q\overline{\kappa}_0^-}\right]+\frac{\gamma}{q+\gamma}\int_x^b\mathrm{E}_y\Big[e^{-q\overline{\kappa}_0^-}\Big]\mathrm{P}_b\left(X_{e_{q+\gamma}}\in\mathrm{d}y;\,e_{q+\gamma}<\tau_x^-\right)
\nonumber\\
\hspace{-0.3cm}&&\hspace{-0.3cm}
+\frac{\gamma}{q+\gamma}\mathrm{E}_b\left[e^{-q\overline{\kappa}_0^-}\right]\Big[\mathrm{P}_b\big(e_{q+\gamma}<\tau_x^-\big)
-\mathrm{P}_b\big(X_{e_{q+\gamma}}\in(x,b);\,e_{q+\gamma}<\tau_x^-\big)\Big]
\nonumber\\
\hspace{-0.3cm}&=&\hspace{-0.3cm}
\frac{\gamma}{q+\gamma}\mathrm{E}_b\left[e^{-q\overline{\kappa}_0^-}\right]\left[1-e^{-\Phi_{q+\gamma}(b-x)}-(q+\gamma)\int_0^{b-x}e^{-\Phi_{q+\gamma}(b-x)}W_{q+\gamma}(b-x-y)\mathrm{d}y\right]
\nonumber\\
\hspace{-0.3cm}&&\hspace{-0.3cm}
+{\gamma}\int_0^{b-x}e^{-\Phi_{q+\gamma}(b-x)}W_{q+\gamma}(b-x-y)\mathrm{E}_{b-y}\left[e^{-q\overline{\kappa}_0^-}\right]\mathrm{d}y
\nonumber\\
\hspace{-0.3cm}&&\hspace{-0.3cm}
+e^{-\Phi_{q+\gamma}(b-x)}\mathrm{E}_x\left[e^{-q\overline{\kappa}_0^-}\right],\quad x\in(0,b).
\end{eqnarray}
Plugging (\ref{stop.time.2}) into (\ref{stop.time.1}), and then rearranging the yielding equation gives
\begin{eqnarray}\label{stop.time.3}
\hspace{-0.3cm}&&\hspace{-0.3cm}
\mathrm{E}_x\left[e^{-q\overline{\kappa}_0^-}\right]\left[1-\frac{\gamma Z_q(b-x)+qZ_q(b-x,\Phi_{q+\gamma})-\Big(\gamma Z_q(b)+qZ_q(b,\Phi_{q+\gamma})\Big)\frac{W_q(b-x)}{W_q(b)}}{qe^{\Phi_{q+\gamma}(b-x)}+\gamma Z_{q+\gamma}(b-x)}\right]
\nonumber\\
\hspace{-0.3cm}&=&\hspace{-0.3cm}
\frac{W_q(b-x)}{W_q(b)}+\bigg[\frac{q+\gamma e^{-\Phi_{q+\gamma}(b-x)}Z_{q+\gamma}(b-x)}{q+\gamma}\bigg]^{-1}
\nonumber\\
\hspace{-0.3cm}&&\hspace{-0.3cm}
\times\bigg[\frac{\gamma Z_q(b-x)+qZ_q(b-x,\Phi_{q+\gamma})}{q+\gamma}
-\frac{\gamma Z_q(b)+qZ_q(b,\Phi_{q+\gamma})}{q+\gamma}\frac{W_q(b-x)}{W_q(b)}\bigg]
\nonumber\\
\hspace{-0.3cm}&&\hspace{-0.3cm}
\times\Big[{\gamma}\int_0^{b-x}e^{-\Phi_{q+\gamma}(b-x)}W_{q+\gamma}(b-x-y)\mathrm{E}_{b-y}\left[e^{-q\overline{\kappa}_0^-}\right]\mathrm{d}y\Big].
\end{eqnarray}
Dividing the both sides of (\ref{stop.time.3}) with $W_q(b-x)$, sending $x$ upward to $b$, and then using (\ref{lim.1}), (\ref{lim.2}), we finally find that \eqref{stop.time.b} also holds true for the case when $X$ has paths of unbounded variation.
Plugging (\ref{stop.time.b}) into \eqref{over.kappa.1} and  (\ref{stop.time.1}) yields \eqref{over.kappa}.
\end{proof}

With the help of Lemmas \ref{lem.pm}-\ref{lem.stop}, we can give, in the following Lemma \ref{LTkappa0-}, an alternative expression for the derivative of the value function of a double barrier strategy for the auxiliary problem, in terms of, the Laplace transform of $\overline{\kappa}_0^-$ as well as the potential measure of $U_t^b$ until $\overline{\kappa}_0^-$.

\begin{lem}\label{LTkappa0-}
We have
\begin{eqnarray}
V_{0,b}^{\omega\prime}(b)=\frac{\phi-1-\mathrm{E}_b\Big[\int_0^{\overline{\kappa}_0^-}e^{-qt}(q\phi-\lambda\omega_+^{\prime}(U_t^b))\mathrm{d}t\Big]}{\mathrm{E}_b\Big[e^{-q\overline{\kappa}_0^-}\Big]Z_q(b,\Phi_{q+\gamma})}+1, \quad b\in(0,\infty).\nonumber
\end{eqnarray}
\end{lem}
\begin{proof}
By Lemmas \ref{lem.pm}-\ref{lem.stop}, one can verify that
\begin{eqnarray}\label{par.b.2}
V_{0,b}^{\omega\prime}(b)
\hspace{-0.3cm}&=&\hspace{-0.3cm}
\frac{\gamma}{q+\gamma}-\frac{\gamma Z_q(b)-\phi(q+\gamma)}{(q+\gamma)Z_q(b,\Phi_{q+\gamma})}
\nonumber\\
\hspace{-0.3cm}&&\hspace{-0.3cm}
+\frac{1}{Z_q(b,\Phi_{q+\gamma})}\bigg[\lambda\int_{0+}^bW_q(y)\omega_+^{\prime}(y)\mathrm{d}y+\lambda\int_0^{\infty}\omega_+^{\prime}(b+y)\bigg[W_q(b+y)
\nonumber\\
\hspace{-0.3cm}&&\hspace{-0.3cm}
+\gamma\int_0^yW_q(b+y-z)W_{q+\gamma}(z)\mathrm{d}z-{Z_q(b,\Phi_{q+\gamma})}W_{q+\gamma}(y)\mathrm{d}y\bigg]\bigg]
\nonumber\\
\hspace{-0.3cm}&=&\hspace{-0.3cm}
\frac{\lambda\mathrm{E}_b\left[\int_0^{\overline{\kappa}_0^-}e^{-qt}\omega_{+}^{\prime}(U_t^b)\mathrm{d}t\right]+\phi\mathrm{E}_b\Big[e^{-q\overline{\kappa}_0^-}\Big]-1}{\mathrm{E}_b\Big[e^{-q\overline{\kappa}_0^-}\Big]Z_q(b,\Phi_{q+\gamma})}+1
\nonumber\\
\hspace{-0.3cm}&=&\hspace{-0.3cm}
\frac{\phi-1-\mathrm{E}_b\Big[\int_0^{\overline{\kappa}_0^-}e^{-qt}(q\phi-\lambda\omega_+^{\prime}(U_t^b))\mathrm{d}t\Big]}{\mathrm{E}_b\Big[e^{-q\overline{\kappa}_0^-}\Big]Z_q(b,\Phi_{q+\gamma})}+1,\nonumber
\end{eqnarray}
which is the desired result.
\end{proof}

Thanks to Lemma \ref{LTkappa0-}, we are able to define, in the following Lemma \ref{lem.b.w}, the dividend barrier level $b^{\omega}$ of the candidate optimal periodic dividend and capital injection strategy among the set of double barrier strategies.

\begin{lem}\label{lem.b.w}
Let us denote $b^{\omega}$ by
\begin{eqnarray}
b^{\omega}:=\inf\left\{b\geq0:\phi-1-\mathrm{E}_b\Big[\int_0^{\overline{\kappa}_0^-}e^{-qt}(q\phi-\lambda\omega_+^{\prime}(U_t^b))\mathrm{d}t\Big]\leq0\right\}.\nonumber
\end{eqnarray}
Then $b^{\omega}>0$ exists and is unique.
\end{lem}
\begin{proof}
We recall the assumption that the payoff function $\omega$ is continuous and concave over $[0, \infty)$ with
$\omega^{\prime}_{+}(0+)\leq \phi$ and $\omega^{\prime}_{+}(\infty)\in[0,1]$. Under this assumption, we have that $b\mapsto q\phi-\lambda\omega_+^{\prime}(b)$ is non-decreasing. By definition, we know that $\overline{\kappa}_0^-$ is non-decreasing with respect to the starting value $b$ of the process $U_t^b$, which combined with the concavity of $\omega$ results in the fact that the function
\begin{eqnarray}
\label{3.49}
b\mapsto\phi-1-\mathrm{E}_b\left[\int_0^{\overline{\kappa}_0^-}e^{-qt}(q\phi-\lambda\omega_+^{\prime}(U_t^b))\mathrm{d}t\right]
\end{eqnarray}
is non-increasing. In addition, due to spatial homogeneity of L\'evy processes and dominated convergence theorem, it can be verified that
\begin{eqnarray}
\label{3.50}
\hspace{-0.3cm}&&\hspace{-0.3cm}
\lim_{b\rightarrow\infty}\left[\phi-1-\mathrm{E}_b\left[\int_0^{\overline{\kappa}_0^-}e^{-qt}(q\phi-\lambda\omega_+^{\prime}(U_t^b))\mathrm{d}t\right]\right]
\nonumber\\
\hspace{-0.3cm}&=&\hspace{-0.3cm}
\phi-1-\lim_{b\rightarrow\infty}\mathrm{E}_0\left[\int_0^{\infty}e^{-qt}\mathbf{1}_{\{t\leq\overline{\kappa}_{-b}^-\}}(q\phi-\lambda\omega_+^{\prime}(b+U_t^0))\mathrm{d}t\right]
\nonumber\\
\hspace{-0.3cm}&=&\hspace{-0.3cm}
\frac{\lambda}{q}\omega_+^{\prime}(\infty)-1<0.
\end{eqnarray}
Furthermore, by \eqref{stop.time.b}, we know that $\lim\limits_{b\downarrow 0}\overline{\kappa}_0^-=0$ almost surely, which together with the dominated convergence theorem gives
\begin{eqnarray}
\label{3.51}
\hspace{-0.3cm}&&\hspace{-0.3cm}
\lim_{b\rightarrow0}\left[\phi-1-\mathrm{E}_b\left[\int_0^{\overline{\kappa}_0^-}e^{-qt}(q\phi-\lambda\omega_+^{\prime}(U_t^b))\mathrm{d}t\right]\right]
=\phi-1>0.
\end{eqnarray}
Piecing together \eqref{3.50}, \eqref{3.51} as well as the non-increasing property of the function given by \eqref{3.49} yields the desired result. The proof is complete.
\end{proof}

With the candidate optimal dividend barrier defined in Lemma \ref{lem.b.w}, we can now investigate, in the upcoming Lemma \ref{V.bw.parl}, the analytical properties, especially the slope conditions, of the value function $V_{0,b^{\omega}}^{\omega}(x)$ of the double barrier periodic dividend and capital injection strategy with dividend barrier $b^{\omega}$ and capital injection barrier $0$.

\begin{lem}\label{V.bw.parl}
The value function $V_{0,b^{\omega}}^{\omega}(x)$ is increasing and concave over $(-\infty,\infty)$. In addition, we have $1\leq V_{0,b^{\omega}}^{\omega\prime}(x)\leq \phi$ for $x\in(0,b^{\omega})$, and $0\leq V_{0,b^{\omega}}^{\omega\prime}(x)\leq1$ for $x\in[b^{\omega},\infty)$.
\end{lem}
\begin{proof}
By Lemma \ref{lem.b.w} and the proof of Lemma \ref{lem.par.b}, one finds that
\begin{eqnarray}\label{v.bw.par.x}
\hspace{-0.8cm}V_{0,b^{\omega}}^{\omega\prime}(x)
\hspace{-0.3cm}&=&\hspace{-0.3cm}
\frac{\gamma}{q+\gamma}Z_q(b^{\omega}-x)-\frac{\gamma Z_q(b^{\omega})-\phi(q+\gamma)}{(q+\gamma)Z_q(b^{\omega},\Phi_{q+\gamma})}Z_q(b^{\omega}-x,\Phi_{q+\gamma})
\nonumber\\
\hspace{-0.3cm}&&\hspace{-0.3cm}
-\lambda\int_{0+}^bW_q(y-x)\omega^{\prime}_+(y)\mathrm{d}y+\frac{Z_q(b^{\omega}-x,\Phi_{q+\gamma})}{Z_q(b^{\omega},\Phi_{q+\gamma})}\bigg[\lambda\int_{0+}^{b^{\omega}}W_q(y)\omega^{\prime}_+(y)\mathrm{d}y
\nonumber\\
\hspace{-0.3cm}&&\hspace{-0.3cm}
+\lambda\int_0^{\infty}\omega^{\prime}_+(b^{\omega}+y)\Big[\Big(W_{q}(b^{\omega}+y)+\gamma\int_0^yW_q(b^{\omega}+y-z)W_{q+\gamma}(z)\mathrm{d}z\Big)-\frac{Z_q(b^{\omega},\Phi_{q+\gamma})}{Z_q(b^{\omega}-x,\Phi_{q+\gamma})}
\nonumber\\
\hspace{-0.3cm}&&\hspace{-0.3cm}
\times\Big(W_{q}(b^{\omega}-x+y)+\gamma\int_0^yW_q(b^{\omega}-x+y-z)W_{q+\gamma}(z)\mathrm{d}z\Big)\Big]\mathrm{d}y
\bigg], \quad x\in(0,\infty),
\end{eqnarray}
and
\begin{eqnarray}\label{V.bw.par.b}
1=V_{0,b^{\omega}}^{\omega\prime}(b^{\omega})
\hspace{-0.3cm}&=&\hspace{-0.3cm}
\frac{\gamma}{q+\gamma}-\frac{\gamma Z_q(b^{\omega})-\phi(q+\gamma)}{(q+\gamma)Z_q(b^{\omega},\Phi_{q+\gamma})}
\nonumber\\
\hspace{-0.3cm}&&\hspace{-0.3cm}
+\frac{1}{Z_q(b^{\omega},\Phi_{q+\gamma})}\bigg[\lambda\int_{0+}^{b^{\omega}}W_q(y)\omega_+^{\prime}(y)\mathrm{d}y+\lambda\int_0^{\infty}\omega_+^{\prime}(b^{\omega}+y)\bigg[W_q(b^{\omega}+y)
\nonumber\\
\hspace{-0.3cm}&&\hspace{-0.3cm}
+\gamma\int_0^yW_q(b^{\omega}+y-z)W_{q+\gamma}(z)\mathrm{d}z-{Z_q(b^{\omega},\Phi_{q+\gamma})}W_{q+\gamma}(y)\mathrm{d}y\bigg]\bigg].
\end{eqnarray}
Plugging (\ref{V.bw.par.b}) into (\ref{v.bw.par.x}) and then rearranging the yielding equation gives
\begin{eqnarray}\label{v.bw.pr.x2}
V_{0,b^{\omega}}^{\omega\prime}(x)
\hspace{-0.3cm}&=&\hspace{-0.3cm}
\frac{\gamma Z_q(b^{\omega}-x)+qZ_q(b^{\omega}-x,\Phi_{q+\gamma})}{q+\gamma}-\lambda\int_{0+}^{b^{\omega}}W_q(y-x)\omega_+^{\prime}(y)\mathrm{d}y
\nonumber\\
\hspace{-0.3cm}&&\hspace{-0.3cm}
+\lambda\int_0^{\infty}\omega^{\prime}_+(b^{\omega}+y)\Big(W_{q+\gamma}(y)Z_q(b^{\omega}-x,\Phi_{q+\gamma})-W_{q+\gamma}(y-x+b^{\omega})
\nonumber\\
\hspace{-0.3cm}&&\hspace{-0.3cm}
-\gamma\int_0^yW_q(y-z-x+b^{\omega})W_{q+\gamma}(z)\mathrm{d}z\Big)\mathrm{d}y
\nonumber\\
\hspace{-0.3cm}&=&\hspace{-0.3cm}
\phi\mathrm{E}_x\Big[e^{-q\overline{\kappa}_0^-}\Big]+\lambda\mathrm{E}_x\Big[\int_0^{\overline{\kappa}_0^-}e^{-qt}\omega^{\prime}_+(U_t^{b^{\omega}})\mathrm{d}t\Big]
\nonumber\\
\hspace{-0.3cm}&=&\hspace{-0.3cm}
\phi-\mathrm{E}_x\Big[\int_0^{\overline{\kappa}_0^-}e^{-qt}\big(q\phi-\lambda\omega^{\prime}_+(U_t^{b^{\omega}})\big)\mathrm{d}t\Big].\nonumber
\end{eqnarray}
Recall that the payoff function $\omega$ is continuous and concave over $[0, \infty)$ with
$\omega^{\prime}_{+}(0+)\leq \phi$ and $\omega^{\prime}_{+}(\infty)\in[0,1]$. It follows that $V_{0,b^{\omega}}^{\omega\prime}(0)\leq\phi$ and $V_{0,b^{\omega}}^{\omega\prime}$ is no-increasing on $(-\infty,\infty)$. Furthermore, due to the fact that $V_{0,b^{\omega}}^{\omega\prime}(b^{\omega})=1$, we derive that $1\leq V_{0,b^{\omega}}^{\omega\prime}(x)\leq \phi$ for $x\in(0,b^{\omega})$, and $0\leq V_{0,b^{\omega}}^{\omega\prime}(x)\leq1$ for $x\in[b^{\omega},\infty)$. The proof is complete.
\end{proof}

Thanks to the slope conditions provided in Lemma \ref{V.bw.parl},
we are now to confirm our conjecture that the double barrier periodic dividend and capital injection strategy with dividend barrier $b^{\omega}$ and capital injection barrier $0$ is the optimal strategy of the auxiliary control problem \eqref{valefunctionofpi}. To this purpose, we need the following Lemma \ref{lem.V.L} for use of verification.
To begin with, for any function $v$ that is sufficiently differentiable (i.e., $v$ is once (resp., twice) continuously differentiable when $X$ has paths of bounded (resp., unbounded) variation), let us define an operator $\mathcal{A}$ on $v$ that
$$\mathcal{A}v(x):=\frac{1}{2}\sigma^2v^{\prime\prime}(x)-cv^{\prime}(x)+\int_{(0,\infty)}\left(v(x+y)-v(x)-v^{\prime}(x)y\mathbf{1}_{(0,1)}(y)\right)\nu(\mathrm{d}y),$$where $x\in(-\infty,\infty)$.

\begin{lem}[Verification Lemma]\label{lem.V.L}
Suppose that
the function $v(x)$ is non-decreasing and continuously differentiable over $(-\infty,\infty)$. Suppose further that $v(x)$ is twice continuously differentiable over $(0,\infty)$ if $X$ has paths of unbounded variation. Additionally, suppose
\begin{eqnarray}
\label{HJB}
\max\{\left(\mathcal{A}-q\right)v(x){\color{black}+\lambda\omega(x)}+\gamma\max_{0\leq z\leq x}\left(z+v(x-z)-v(x)\right), v^{\prime}(x)-\phi\}\leq 0.
\end{eqnarray}
Then
$v(x)\geq  V_{\pi}^{\omega}(x)$ for all $x\in(-\infty,\infty)$ and all admissible periodic dividend and capital injection strategy $(D^{\pi},R^{\pi})\in\Pi$.
\end{lem}
\begin{proof}
Let $\mathcal{D}$ be the set of admissible dividend and capital injection strategy $(D^{\pi}_{t},R^{\pi}_{t})_{t\geq0}$ with $R^{\pi}_{t}$ being continuous and of form \eqref{R.form}.
By Lemma \ref{R.continu.}, we only need to prove that $v(x)$ dominates the value function of any admissible dividend and capital injection strategies among $\mathcal{D}$.
For a given strategy $(D^{\pi},R^{\pi})\in\mathcal{D}$, recall that $U^{\pi}_{t}=X_{t}-\sum_{n=1}^{\infty}\Delta D^{\pi}_{T_n}\mathbf{1}_{\{T_n\leq t\}}+R^{\pi}_{t}$ for $t\geq0$.
We follow Theorem 2.1 in \cite{Kyprianou14} to denote $X_{t}$ as the sum of the independent processes $-c t+\sigma B_{t}$, $\sum_{s\leq t}\Delta X_{s}\mathbf{1}_{\{\Delta X_{s}\geq 1\}}$, and
$X_{t}+c t-\sigma B_{t}-\sum_{s\leq t}\Delta X_{s}\mathbf{1}_{\{\Delta X_{s}\geq 1\}}$,
with the latter one being a square integrable martingale.
Denote by $\{ {U}^{\pi,c}_{t};t\geq0\}$
as the continuous part of  $\{ {U}^{\pi}_{t};t\geq0\}$. By Theorem 4.57 (It\^{o}'s formula) in \cite{Jacod03}, we have, for $x\in(0,\infty)$,
{\color{black}\begin{eqnarray}\label{HJB.2}
\hspace{-0.3cm}&&\hspace{-0.3cm}
e^{-q t}v({U}^{\pi}_{t})-v(x)\nonumber\\
\hspace{-0.3cm}&=&\hspace{-0.3cm}
-\int_{0-}^{t}q e^{-q s}v( {U}^{\pi}_{s-})\mathrm{d}s+\int_{0-}^{t} e^{-q s}v^{\prime}( {U}^{\pi}_{s-})\mathrm{d} {U}^{\pi}_{s}\nonumber\\
\hspace{-0.3cm}&&\hspace{-0.3cm}
+\frac{1}{2} \int_{0-}^{t} e^{-q s}v^{\prime\prime}( {U}^{\pi}_{s-})\mathrm{d}\langle  {U}^{\pi,c}(\cdot), {U}^{\pi,c}(\cdot)\rangle_{s}
\nonumber\\
\hspace{-0.3cm}&&\hspace{-0.3cm}
+\sum_{s\leq t}e^{-q s}\big(v( {U}^{\pi}_{s-}+\Delta  {U}^{\pi}_{s})-v( {U}^{\pi}_{s-})-v^{\prime}( {U}^{\pi}_{s-})\Delta  {U}^{\pi}_{s}\big)
\nonumber\\
\hspace{-0.3cm}&=&\hspace{-0.3cm}
-\int_{0-}^{t}q e^{-q s}v({U}^{\pi}_{s-})\mathrm{d}s+\int_{0-}^{t} e^{-q s}v^{\prime}({U}^{\pi}_{s-})
\mathrm{d}(-c s+\sigma B_{s})\nonumber\\
\hspace{-0.3cm}&&\hspace{-0.3cm}
+\int_{0-}^{t} e^{-q s}v^{\prime}({U}^{\pi}_{s-})
\mathrm{d}\big(X_{s}+c s-\sigma B_{s}-\sum_{r\leq s}\Delta X_r\mathbf{1}_{\{\Delta X_r\geq 1\}}\big)
\nonumber\\
\hspace{-0.3cm}&&\hspace{-0.3cm}
+\int_{0-}^{t} e^{-q s}v^{\prime}( {U}^{\pi}_{s-})
\mathrm{d}R^{\pi}_{s}
+\frac{\sigma^{2}}{2} \int_{0-}^{t} e^{-q s}v^{\prime\prime}( {U}^{\pi}_{s-})\mathrm{d}s
\nonumber\\
\hspace{-0.3cm}&&\hspace{-0.3cm}
+\int_{0-}^{t} e^{-q s}v^{\prime}( {U}_{s-})
\mathrm{d}\big(\sum_{r\leq s}\Delta X_r\mathbf{1}_{\{\Delta X_r\geq 1\}}\big)
\nonumber\\
\hspace{-0.3cm}&&\hspace{-0.3cm}
+\sum_{s\leq t,\Delta D^{\pi}_s=0,\Delta X_s\neq0}e^{-q s}\big[v( {U}^{\pi}_{s-}+\Delta X_{s})-v({U}^{\pi}_{s-})-v^{\prime}( {U}^{\pi}_{s-})\Delta X_{s}\big]
\nonumber\\
\hspace{-0.3cm}&&\hspace{-0.3cm}
+\sum_{s\leq t,\Delta D_s^{\pi}\neq0}e^{-q s}\big[v( {U}^{\pi}_{s-}+\Delta X_s+\Delta D_s^{\pi})-v({U}^{\pi}_{s-}+\Delta X_{s})
\big]
\nonumber\\
\hspace{-0.3cm}&=&\hspace{-0.3cm}
-\int_{0-}^{t}q e^{-q s}v( {U}^{\pi}_{s-})\mathrm{d}s+\int_{0-}^{t} e^{-q s}v^{\prime}( {U}^{\pi}_{s-})
\mathrm{d}(-c s+\sigma B_{s})\nonumber\\
\hspace{-0.3cm}&&\hspace{-0.3cm}
+\int_{0-}^{t} e^{-q s}v^{\prime}( {U}^{\pi}_{s-})
\mathrm{d}\big(X_{s}+c s-\sigma B_{s}-\sum_{r\leq s}\Delta X_r\mathbf{1}_{\{\Delta X_r\geq 1\}}\big)
\nonumber\\
\hspace{-0.3cm}&&\hspace{-0.3cm}
+\int_{0-}^{t} e^{-q s}v^{\prime}( {U}^{\pi}_{s-})
\mathrm{d}R^{\pi}_{s}
+\frac{\sigma^{2}}{2} \int_{0-}^{t} e^{-q s}v^{\prime\prime}( {U}^{\pi}_{s-})\mathrm{d}s
\nonumber\\
\hspace{-0.3cm}&&\hspace{-0.3cm}
+\sum_{s\leq t}e^{-q s}\big[v( {U}^{\pi}_{s-}+\Delta X_{s})-v( U^{\pi}_{s-})
-v^{\prime}( {U}^{\pi}_{s-})\Delta X_{s}\mathbf{1}_{\{\Delta X_{s}<1\}}\big]
\nonumber\\
\hspace{-0.3cm}&&\hspace{-0.3cm}
+\int_{0-}^{t}e^{-qs}\Big[v(U_{s-}^{\pi}+\Delta X_s+\Delta D^{\pi}_{s})-v(U_{s-}^{\pi}+\Delta X_s)\Big]\mathrm{d}N_s,
\end{eqnarray}}\noindent
where $\Delta X_s=X_s-X_{s-}$, $\Delta D^{\pi}_s=\sum_{n=1}^{\infty}\Delta {D}^{\pi}_{T_n}\mathbf{1}_{\{T_n= s\}}$ and $\Delta{U}^{\pi}_s={U}^{\pi}_s-{U}^{\pi}_{s-}=\Delta X_s-\Delta D_s^{\pi}$.
Define a sequence of stopping times $(\widetilde{T}_m)_{m\geq1}$ that $$\widetilde{T}_m:=m\wedge\inf\{t\geq0;{U}^{\pi}_t\geq m\},\quad n\geq1.$$ It follows that $\widetilde{T}_m\rightarrow\infty$ almost surely as $n\rightarrow\infty$. In addition, ${U}_{t-}$ is confined in the compact set $[0,m]$ for $t\leq \widetilde{T}_m$. By (\ref{HJB})-(\ref{HJB.2}), we have
\begin{eqnarray}\label{HJB.3}
\hspace{-0.3cm}&&\hspace{-0.3cm}
e^{-q(t\wedge\widetilde{T}_m)}v({U}^{\pi}_{t\wedge\widetilde{T}_m})-v(x)
\nonumber\\
\hspace{-0.3cm}&=&\hspace{-0.3cm}
M_{t\wedge\widetilde{T}_m}+\int_{0-}^{t\wedge\widetilde{T}_m}e^{-qs}\Big((\mathcal{A}-q)v(U^{\pi}_{s-})+\lambda\omega(U^{\pi}_{s-})
\nonumber\\
\hspace{-0.3cm}&&\hspace{-0.3cm}
+\gamma\big[\Delta D^{\pi}_s+v(U_{s-}^{\pi}+\Delta X_s+\Delta D^{\pi}_{s})-v(U_{s-}^{\pi}+\Delta X_s)\big]\Big)\mathrm{d}s
\nonumber\\
\hspace{-0.3cm}&&\hspace{-0.3cm}
+\int_{0-}^{t\wedge\widetilde{T}_m}e^{-qs}v^{\prime}(U^{\pi}_{s-})\mathrm{d}R^{\pi}_s-\int_{0}^{t\wedge\widetilde{T}_m}e^{-qs}\lambda\omega(U^{\pi}_{s-})\mathrm{d}s-\int_{0}^{t\wedge\widetilde{T}_m}e^{-qs}\Delta D^{\pi}_s\mathrm{d}N_s
\nonumber\\
\hspace{-0.3cm}&\leq&\hspace{-0.3cm}
M_{t\wedge\widetilde{T}_m}+\phi\int_{0-}^{t\wedge\widetilde{T}_m}e^{-qs}\mathrm{d}R^{\pi}_s-\int_{0}^{t\wedge\widetilde{T}_m}e^{-qs}\lambda\omega(U^{\pi}_{s-})\mathrm{d}s-\int_{0}^{t\wedge\widetilde{T}_m}e^{-qs}\Delta D^{\pi}_s\mathrm{d}N_s,
\end{eqnarray}
where $M_{t\wedge\widetilde{T}_m}$ is the sum of the three zero mean martingales $M^1_{t\wedge\widetilde{T}_m}$, $M^2_{t\wedge\widetilde{T}_m}$ and $M^3_{t\wedge\widetilde{T}_m}$ given respectively by
\begin{eqnarray}
M^1_{t\wedge\widetilde{T}_m}
\hspace{-0.3cm}&=&\hspace{-0.3cm}
\int_{0-}^{t\wedge\widetilde{T}_m}e^{-qs}v^{\prime}(U^{\pi}_{s-})\mathrm{d}\bigg(X_s+c s-\sum_{r\leq s}\Delta X_r\mathbf{1}_{\{\Delta X_r\geq1\}}\bigg),\nonumber
\end{eqnarray}
and
\begin{eqnarray}
M^2_{t\wedge\widetilde{T}_m}
\hspace{-0.3cm}&=&\hspace{-0.3cm}
\int_{0-}^{t\wedge\widetilde{T}_m}\int_0^{\infty}e^{-qs}\Big(v(U^{\pi}_{s-}+y)-v(U^{\pi}_{s-})-v^{\prime}(U^{\pi}_{s-})y\mathbf{1}_{(0,1]}(y)\Big)\overline{N}(\mathrm{d}s,\mathrm{d}y),\nonumber
\end{eqnarray}
and
\begin{eqnarray}
M^3_{t\wedge\widetilde{T}_m}
\hspace{-0.3cm}&=&\hspace{-0.3cm}
\int_{0-}^{t\wedge\widetilde{T}_m}e^{-qs}\big[\Delta D^{\pi}_s+v(U_{s-}^{\pi}+\Delta X_s+\Delta D^{\pi}_{s})-v(U_{s-}^{\pi}+\Delta X_s)\big]\left(\mathrm{d}N_s-\gamma\mathrm{d}s\right),\nonumber
\end{eqnarray}
where we used the L\'evy-It\^o decomposition theorem (see, Theorem 2.1 in \cite{Kyprianou14}) and the compensation formula (see, Theorem 4.4 in \cite{Kyprianou14}), respectively. Then taking the expectation on both side of (\ref{HJB.3}), letting $t$ and $m$ go to infinity, and by bounded convergence theorem (note that $v(0)$ is bounded) yields
\begin{eqnarray}
v(x)\hspace{-0.3cm}&\geq&\hspace{-0.3cm}\lim_{t,m\rightarrow\infty}\mathrm{E}_x\Big[e^{-q(t\wedge\widetilde{T}_m)}v(U^{\pi}_{t\wedge\widetilde{T}_m})\Big]-\phi\mathrm{E}_x\Big[\int_{0-}^{t\wedge\widetilde{T}_m}e^{-qs}\mathrm{d}R^{\pi}_s\Big]
\nonumber\\
\hspace{-0.3cm}&&\hspace{-0.3cm}
+\mathrm{E}_x\Big[\int_{0}^{t\wedge\widetilde{T}_m}e^{-qs}\Delta D^{\pi}_s\mathrm{d}N_s\Big]+\mathrm{E}_x\Big[\int_{0-}^{t\wedge\widetilde{T}_m}e^{-qs}\lambda\omega(U^{\pi}_{s-})\mathrm{d}s\Big]
\nonumber\\
\hspace{-0.3cm}&\geq&\hspace{-0.3cm}
\lim_{t,m\rightarrow\infty}\mathrm{E}_x\Big[e^{-q(t\wedge\widetilde{T}_m)}v(0)\Big]-\phi\mathrm{E}_x\Big[\int_{0-}^{t\wedge\widetilde{T}_m}e^{-qs}\mathrm{d}R^{\pi}_s\Big]
\nonumber\\
\hspace{-0.3cm}&&\hspace{-0.3cm}
+\mathrm{E}_x\Big[\sum_{n=1}^{\infty}e^{-qT_n}\Delta {D}^{\pi}_{T_n}\mathbf{1}_{\{T_n\leq t\wedge\widetilde{T}_m\}}\Big]+\mathrm{E}_x\Big[\int_{0}^{t\wedge\widetilde{T}_m}e^{-qs}\lambda\omega(U^{\pi}_{s-})\mathrm{d}s\Big]
\nonumber\\
\hspace{-0.3cm}&\geq&\hspace{-0.3cm}-\phi\mathrm{E}_x\Big[\int_{0-}^{\infty}e^{-qs}\mathrm{d}R^{\pi}_s\Big]+\mathrm{E}_x\Big[\sum_{n=1}^{\infty}e^{-qT_n}\Delta D^{\pi}_{T_n}\Big]+\mathrm{E}_x\Big[\int_{0}^{\infty}e^{-qs}\lambda\omega(U^{\pi}_{s})\mathrm{d}s\Big]
\nonumber\\
\hspace{-0.3cm}&=&\hspace{-0.3cm}V_{\pi}^{\omega}(x),\quad x\in(0,\infty).\nonumber
\end{eqnarray}
The arbitrariness of $\pi$ and the continuity of $v$ imply that
$v(x)\geq V_{\pi}^{\omega}(x)$ for all $x\in[0,\infty)$ and all admissible $(D^{\pi},R^{\pi})\in\Pi$. The reverse inequality is trivial, and the proof is complete.
\end{proof}

\begin{lem}
It holds that, for $x\in(0,\infty)$,
\begin{eqnarray}\label{A.q.x}
\left\{\begin{aligned}
     & \mathcal{A}V_{0,b}^{\omega}(x)-qV_{0,b}^{\omega}(x)+\lambda \omega(x)=0,&x\in(0,b),\\
     & \mathcal{A}V_{0,b}^{\omega}(x)-qV_{0,b}^{\omega}(x)+\lambda \omega(x)+\gamma(x-b+V_{0,b}^{\omega}(b)-V_{0,b}^{\omega}(x))=0,&x\in[b,\infty).
\end{aligned}\right.
\end{eqnarray}
\end{lem}
\begin{proof}
Put $\kappa:=\kappa_0^-\wedge\kappa_{b}^+$. By the definition of $U_t^{0,b}$, one can get that the controlled process $U_t^{0,b}$ follows the same dynamics of $X$ before $\kappa$. By the strong Markov property of the process $X$, we have
\begin{eqnarray}
\hspace{-0.3cm}&&\hspace{-0.3cm}
\mathrm{E}_x\left[\sum_{n=1}^{\infty}e^{-qT_n}\Delta D^{0,b}_{T_n}-\phi\int_0^{\infty}e^{-qt}\mathrm{d}R_t^{0,b}+\lambda\int_0^{\infty}e^{-qt}\omega(U_t^{0,b})\mathrm{d}t\bigg|\mathcal{F}_{s\wedge\kappa}\right]
\nonumber\\
\hspace{-0.3cm}&=&\hspace{-0.3cm}
\mathrm{E}_x\left[\sum_{n=1}^{\infty}e^{-qT_n}\Delta D^{0,b}_{T_n}\mathbf{1}_{\{T_n\geq s\wedge\kappa\}}-\phi\int_0^{\infty}e^{-q(t+s\wedge\kappa)}\mathrm{d}R_{t+s\wedge\kappa}^{0,b}\right.
\nonumber\\
\hspace{-0.3cm}&&\hspace{-0.3cm}
\left.+\lambda\int_0^{\infty}e^{-q(t+s\wedge\kappa)}\omega(U_{t+s\wedge\kappa}^{0,b})\mathrm{d}t\bigg|\mathcal{F}_{s\wedge\kappa}\right]+\lambda\int_0^{s\wedge\kappa}e^{-qt}\omega(X_t)\mathrm{d}t
\nonumber\\
\hspace{-0.3cm}&=&\hspace{-0.3cm}
e^{-q(s\wedge\kappa)}\mathrm{E}_{X_{s\wedge\kappa}}\left[\sum_{n=1}^{\infty}e^{-qT_n}\Delta D^{0,b}_{T_n}-\phi\int_0^{\infty}e^{-qt}\mathrm{d}R_t^{0,b}+\lambda\int_0^{\infty}e^{-qt}\omega(U_t^{0,b})\mathrm{d}t\right]
\nonumber\\
\hspace{-0.3cm}&&\hspace{-0.3cm}
+\lambda\int_0^{s\wedge\kappa}e^{-qt}\omega(X_t)\mathrm{d}t
\nonumber\\
\hspace{-0.3cm}&=&\hspace{-0.3cm}
e^{-q(s\wedge\kappa)}V_{0,b}^{\omega}(X_{s\wedge\kappa})+\lambda\int_0^{s\wedge\kappa}e^{-qt}\omega(X_t)\mathrm{d}t,\quad x\in(0,b),\nonumber
\end{eqnarray}
which implies that the right-side of the above equation is a martingale. By It\^o's formula, it holds that
\begin{eqnarray}
\hspace{-0.3cm}&&\hspace{-0.3cm}
e^{-q(t\wedge\kappa)}V_{0,b}^{\omega}(X_{t\wedge\kappa})+\lambda\int_0^{t\wedge\kappa}e^{-qs}\omega(X_s)\mathrm{d}s-V_{0,b}^{\omega}(x)
\nonumber\\
\hspace{-0.3cm}&=&\hspace{-0.3cm}
\int_{0-}^{t\wedge\kappa}e^{-qs}\big((\mathcal{A}-q)V_{0,b}^{\omega}(X_{s-})+\lambda\omega(X_{s-})\big)\mathrm{d}s+\int_{0-}^{t\wedge\kappa}\sigma e^{-qs}V_{0,b}^{\omega\prime}(X_{s-})\mathrm{d}B_s
\nonumber\\
\hspace{-0.3cm}&&\hspace{-0.3cm}
+\int_{0-}^{t\wedge\kappa}e^{-qs}V_{0,b}^{\omega\prime}(X_{s-})\mathrm{d}(X_s+cs-\sigma B_s-\sum_{r\leq s}\Delta X_r\mathbf{1}_{\{\Delta X_r\geq1\}})
\nonumber\\
\hspace{-0.3cm}&&\hspace{-0.3cm}
+\int_{0-}^{t\wedge\kappa}\int_0^{\infty}e^{-qs}[V_{0,b}^{\omega}(X_{s-}-y)-V_{0,b}^{\omega}(X_{s-})+V_{0,b}^{\omega\prime}(X_{s-})y\mathbf{1}_{(0,1]}(y)]\overline{N}(\mathrm{d}s,\mathrm{d}y),\quad t\geq0.\nonumber
\end{eqnarray}
Following the same arguments in the proof of Lemma \ref{lem.V.L}, we get that all the terms (except for the first one) on the right hand side of the above equality are martingales starting from 0. Therefore, taking expectations on both sides of the above equation yields
\begin{eqnarray}
0=\mathrm{E}_x\left[\int_{0-}^{t\wedge\kappa}e^{-qs}\big((\mathcal{A}-q)V_{0,b}^{\omega}(X_{s-})+\lambda\omega(X_{s-})\big)\mathrm{d}s\right],\quad t\geq0,\,x\in(0,b^{\omega}).\nonumber
\end{eqnarray}
Dividing both sides of the above equation by $t$ and then setting $t\downarrow0$, we can obtain the equality in (\ref{A.q.x}) for $x\in(0,b)$ by the mean value theorem and the dominated convergence theorem. For $x\in[b,\infty)$, it can be verified that
\begin{eqnarray}\label{Aq.1}
\left(\mathcal{A}-q\right)(b-x+\frac{\psi^{\prime}(0+)}{q})=-q(b-x),
\end{eqnarray}
and
\begin{eqnarray}
\left(\mathcal{A}-q\right)e^{-\Phi_{q+\gamma}x}=\gamma e^{-\Phi_{q+\gamma}x}.
\end{eqnarray}
In addition, by (5.6) in \cite{Noba18}, one can get
\begin{eqnarray}\label{Aq.2}
\hspace{-0.3cm}&&\hspace{-0.3cm}
(\mathcal{A}-q)\Big(Z_q(b-x+y)+\gamma\int_0^yZ_q(b-x+y-z)W_{q+\gamma}(z)\mathrm{d}z\Big)
\nonumber\\
\hspace{-0.3cm}&=&\hspace{-0.3cm}
\gamma\Big(Z_q(b-x+y)+\gamma\int_0^yZ_q(b-x+y-z)W_{q+\gamma}(z)\mathrm{d}z\Big),\quad x\in[b,b+y).
\end{eqnarray}
By (\ref{Aq.1})-(\ref{Aq.2}), and the definition of $V_{0,b}^{\omega}(x)$, we have
\begin{eqnarray}
\hspace{-0.3cm}&&\hspace{-0.3cm}
(\mathcal{A}-q)V_{0,b}^{\omega}(x)
\nonumber\\
\hspace{-0.3cm}&=&\hspace{-0.3cm}
(\mathcal{A}-q)\bigg[\frac{-\gamma}{q+\gamma}\Big(b-x+\frac{\psi^{\prime}(0+)}{q}\Big)-\lambda\int_0^b\omega(y)W_q(y-x)\mathrm{d}y+\frac{\lambda\omega(b)}{q}Z_q(b-x)
\nonumber\\
\hspace{-0.3cm}&&\hspace{-0.3cm}
+\frac{\gamma Z_q(b)-\phi(q+\gamma)}{q(q+\gamma)\Phi_{q+\gamma}Z_q(b,\Phi_{q+\gamma})}(qe^{\Phi_{q+\gamma}(b-x)}+\gamma)
-\frac{qe^{\Phi_{q+\gamma}(b-x)}+\gamma}{q\Phi_{q+\gamma}Z_q(b,\Phi_{q+\gamma})}\times\bigg[
\lambda\int_0^b\omega^{\prime}(y)W_q(y)\mathrm{d}y
\nonumber\\
\hspace{-0.3cm}&&\hspace{-0.3cm}
+\lambda\int_0^{\infty}\omega^{\prime}(b+y)\bigg(W_q(b+y)+\gamma\int_0^yW_q(b+y-z)W_{q+\gamma}(z)\mathrm{d}z-\frac{\Phi_{q+\gamma}Z_q(b,\Phi_{q+\gamma})}{qe^{\Phi_{q+\gamma}(b-x)}+\gamma}
\nonumber\\
\hspace{-0.3cm}&&\hspace{-0.3cm}
\times\Big[\Big(Z_q(b-x+y)+\gamma\int_0^yZ_q(b-x+y-z)W_{q+\gamma}(z)\mathrm{d}z\Big)\mathbf{1}_{\{y>x-b\}}+\mathbf{1}_{\{y\leq x-b\}}\Big]\bigg)\mathrm{d}y\bigg]\bigg]
\nonumber\\
\hspace{-0.3cm}&=&\hspace{-0.3cm}
\frac{q\gamma}{q+\gamma}(x-b){\color{black}-\lambda\omega(b)}
+\frac{\gamma Z_q(b)-\phi(q+\gamma)}{(q+\gamma)\Phi_{q+\gamma}Z_q(b,\Phi_{q+\gamma})}(\gamma e^{\Phi_{q+\gamma}(b-x)}-\gamma)-\frac{\gamma e^{\Phi_{q+\gamma}(b-x)}-\gamma}{\Phi_{q+\gamma}Z_q(b,\Phi_{q+\gamma})}
\nonumber\\
\hspace{-0.3cm}&&\hspace{-0.3cm}
\times\bigg(
\lambda\int_0^b\omega^{\prime}(y)W_q(y)\mathrm{d}y+\lambda\int_0^{\infty}\omega^{\prime}(b+y)\Big[\Big(W_q(b+y)+\gamma\int_0^yW_q(b+y-z)W_{q+\gamma}(z)\mathrm{d}z\Big)
\nonumber\\
\hspace{-0.3cm}&&\hspace{-0.3cm}
-\frac{\Phi_{q+\gamma}Z_q(b,\Phi_{q+\gamma})}{ e^{\Phi_{q+\gamma}(b-x)}-1}\Big(Z_q(b-x+y)+\gamma\int_0^yZ_q(b-x+y-z)W_{q+\gamma}(z)\mathrm{d}z\Big)\mathbf{1}_{\{y>x-b\}}\Big]\mathrm{d}y\bigg)
\nonumber\\
\hspace{-0.3cm}&&\hspace{-0.3cm}
-{\lambda\omega(x)}+{\lambda\omega(b)}
\nonumber\\
\hspace{-0.3cm}&=&\hspace{-0.3cm}\gamma(x-b+V_{0,b}^{\omega}(b)-V_{0,b}^{\omega}(x))-{\color{black}\lambda\omega(x)}.\nonumber
\end{eqnarray}
The proof is complete.
\end{proof}
\begin{lem}\label{max.V}
Let $b^{\omega}>0$ be defined as in Lemma \ref{lem.b.w}. We have
\begin{eqnarray}
\max_{0\leq z\leq x}\{z+V^{\omega}_{0,b^{\omega}}(x-z)-V_{0,b^{\omega}}^{\omega}(x)\}=\left\{
\begin{aligned}
&0,&x\in(0,b^{\omega}),\\
&x-b^{\omega}+V^{\omega}_{0,b^{\omega}}(b^{\omega})-V_{0,b^{\omega}}^{\omega}(x),&x\in[b^{\omega},\infty).
\end{aligned}
\right.\nonumber
\end{eqnarray}
\end{lem}
\begin{proof}
The result is immediate consequence of Lemma \ref{V.bw.parl}.
\end{proof}

Putting together Lemma \ref{lem.par.b} and Lemmas \ref{lem.V.L}-\ref{max.V}, we can easily verify, in the following Theorem \ref{them.3.1}, the conjecture that the double barrier strategy with dividend barrier $b^{\omega}$ and capital injection barrier $0$ is the optimal strategy for the auxiliary control problem \eqref{valefunctionofpi}. The proof is omitted.

\begin{thm}\label{them.3.1}
The periodic dividend and capital injection strategy $(D_t^{0,b^{\omega}},R_t^{0,b^{\omega}})$ dominates all admissible singular periodic dividend and capital injection strategies that
$$V^{\omega}_{0,b^{\omega}}(x)=\sup_{\pi}V_{\pi}^{\omega}(x).$$
\end{thm}

\section{Optimality of Regime-modulate Double Barrier Strategy}
\label{sec:opt}

We continue to prove the main result Theorem \ref{thm2.1} using results from the previous auxiliary control
problem with a final payoff and the recursive iteration based on dynamic programming principle. As
preparations,
let us consider the following space of functions
$$\mathcal{B}:=\{f:\mathbb{R}_+\times\mathcal{E}\rightarrow\mathbb{R}|\text{ for each } i\in\mathcal{E},\, f(\cdot,i)\in C([0,\infty))\text{ and } \|f\|<\infty\},$$
endowed with the norm $$\|f\|:=\max_{i\in\mathcal{E}}\sup_{x\geq0}\frac{|f(x,i)|}{1+|x|},$$
and
the metric $\rho(\cdot,\cdot)$ induced by  $\|\cdot\|$. It is not hard to check that the metric space $(\mathcal{B},\rho)$ is complete.

For any function $f: [0,\infty)\times \mathcal{E}\rightarrow \mathbb{R}$, we define a function $\widehat{f}: [0,\infty)\times \mathcal{E}\rightarrow \mathbb{R}$ that
\begin{eqnarray}
\widehat{f}(x,i):=\sum_{j\in\mathcal{E}, j\neq  i}\frac{\lambda_{ij}}{\lambda_{i}}\int_{-\infty}^{0}\left[f(x+y,j)\mathbf{1}_{\{-y\leq x\}}+(\phi(x+y)+f(0,j))\mathbf{1}_{\{-y>x\}}\right]\mathrm{d}F_{ij}(y),\nonumber
\end{eqnarray}
where $\lambda_{i}=\sum_{j\neq i}\lambda_{ij}$, and $F_{ij}$ is the distribution function of $J_{ij}$ for $i,j\in\mathcal{E}$.
Note that
\begin{eqnarray}
\label{4.2}
\frac{\left|\widehat{f}(x,i)\right|}{1+|x|}
\hspace{-0.3cm}&=&\hspace{-0.3cm}
\left|\sum_{j\in\mathcal{E}, j\neq  i}\frac{\lambda_{ij}}{\lambda_{i}}\int_{-\infty}^{0}\bigg[\frac{f(x+y,j)}{1+|x|}\mathbf{1}_{\{-y\leq x\}}+\Big(\frac{f(0,j)}{1+|x|}+\frac{\phi\left(x+y\right)}{1+|x|}\Big)\mathbf{1}_{\{-y>x\}}\bigg]\right|
\nonumber\\
\hspace{-0.3cm}&\leq&\hspace{-0.3cm}
\sum_{j\in\mathcal{E}, j\neq  i}\frac{\lambda_{ij}}{\lambda_{i}}\Big[\|f\|\int_{-\infty}^0\frac{1+|x+y|}{1+|x|}\mathbf{1}_{\{-y\leq x\}}\mathrm{d}F_{ij}(y)+\frac{\phi\mathrm{E}|J_{ij}|}{1+|x|}+(\phi+|f(0,j)|)\Big]
\nonumber\\
\hspace{-0.3cm}&\leq&\hspace{-0.3cm}
\sum_{j\in\mathcal{E}, j\neq  i}\frac{\lambda_{ij}}{\lambda_{i}}\big[\|f\|+\phi\mathrm{E}|J_{ij}|+(\phi+|f(0,j)|)\big]
,\quad (x,i)\in [0,\infty)\times\mathcal{E},\nonumber
\end{eqnarray}
which together the fact that $\max_{i,j\in\mathcal{E}}\mathrm{E}|J_{ij}|<\infty$, one can get that $\widehat{f}\in\mathcal{B}$ when $f\in\mathcal{B}$.

For any function $\mathbf{b}=(b_{i})\in[0,\infty)^{\mathcal{E}}$, denote by $V_{0,\mathbf{b}}(x,i)$ the value function (i.e., the NPV of the accumulated differences between dividends and the costs of capital injections) of the periodic dividend and capital injection strategy with dynamic upper periodic barrier $b_{Y_t}$ and constant lower barrier $0$. In addition, let us define a mapping $\mathcal{T}_{\mathbf{b}}$ acting on $f \in \mathcal{B}$ such that
\begin{eqnarray}
\label{def.Tb}
\mathcal{T}_{\mathbf{b}}f(x,i)
\hspace{-0.3cm}&:=&\hspace{-0.3cm}
\mathrm{E}_x^i\bigg[\sum_{n=1}^{\infty}e^{-q_{i}T_n}\Delta D_{T_n}^{b_{i},i}
-\phi\int_{0}^{\infty}e^{-q_{i}t}\mathrm{d}R_{t}^{b_{i},i}+\lambda_{i}\int_{0}^{\infty}e^{-q_{i}t}\widehat{f}(U_{t}^{b_{i},i},i)\mathrm{d}t\bigg],
\end{eqnarray}
where $q_i=\delta_{i}+\lambda_i$ and $\mathrm{E}_x^i$ denotes the expectation operator with respect to the law of the process $X^{i}$ conditioned on the event $\{X_{0}^{i}=x\}$. The process $U_{t}^{b_{i},i}$ is the controlled process with upper periodic barrier $b_{i}\geq0$, lower reflecting barrier $0$, and the underlying risk process $X^{i}$; and $D_{t}^{b_{i},i}$, $R_{t}^{b_{i},i}$ are the cumulative dividends paid and capitals injected, respectively.
In what follows, the scale functions of $X^{i}$ will be denoted by $W_{q,i}$, $Z_{q,i}$ and $\overline{Z}_{q,i}$, whose definitions are given in Section 2.1 where the subscript $i$ is absent.

\begin{lem}\label{v.in.B}
For $(x,i)\in\mathbb{R}_+\times\mathcal{E}$, we have $V(x,i)\in\mathcal{B}$.
\end{lem}
\begin{proof}
Denote $\underline{\delta}:=\min_{i\in\mathcal{E}}\delta_i$, $\overline{X}_t:=\sup_{s\leq t}X_s$, and $\underline{X}_t:=\inf_{s\leq t}X_s$. We can derive an upper bound of $V(x,i)$ by considering the extreme case where the manager of the company pays every dollar accumulated by $X$ as dividends as early as possible (i.e., $D_t:=\overline{X}_t\vee0$), and cover all deficits by capital injection (i.e., $R_t:=-\inf_{s\leq t}(X_s-(\overline{X}_t\vee0))$). Note that the surplus process $U_t=X_t-(\overline{X}_t\vee0)-\inf_{s\leq t}(X_s-(\overline{X}_t\vee0))$ takes positive values. Hence, $D_t=\overline{X}_t\vee0$ amounts to the maximum reasonable amount of dividends paid until time $t\geq0$. Therefore, we have
\begin{eqnarray}
V(x,i)\leq x+\mathrm{E}_{0,i}\Big[\int_0^{\infty}e^{-\underline{\delta}t}\mathrm{d}(\overline{X}_t\vee0)\Big]=:\overline{V}(x,i).\nonumber
\end{eqnarray}
Similarly, we can also derive a lower bound by considering the extreme case where the manager of the company injects capitals to keep the surplus over $x$ before the time that the first Poisson arrival with intensity $\gamma>0$ (i.e., $R_t:=-\inf_{s\leq t}(X_s-x)\wedge0$, $t\leq e_{\gamma}$), and pays whatever he has as dividends as the first Poisson arrival time (i.e., $D_{e_{\gamma}}:=\big(X_{e_{\gamma}}-\inf_{s\leq e_{\gamma}}(X_s-x)\wedge0\big)$); and then pays no dividends afterwards and bails out all deficits by injecting capitals.
Hence, by the spatial homogeneity, we have
\begin{eqnarray}
\underline{V}(x,i)
\hspace{-0.3cm}&:=&\hspace{-0.3cm}
\mathrm{E}_{x,i}\Big[e^{-\underline{\delta}e_{\gamma}}\big(X_{e_{\gamma}}-(\underline{X}_{e_{\gamma}}-x)\wedge0\big)+\phi\int_0^{e_{\gamma}}e^{-\underline{\delta}t}\mathrm{d}\big((\underline{X}_t-x)\wedge0\big)\Big]
\nonumber\\
\hspace{-0.3cm}&&\hspace{-0.3cm}
+\mathrm{E}_{0,i}\Big[\phi\int_0^{\infty}e^{-\underline{\delta}(t+e_{\gamma})}\mathrm{d}\big(\underline{X}_{t+e_{\gamma}}\wedge0\big)\Big]
\nonumber\\
\hspace{-0.3cm}&=&\hspace{-0.3cm}
\mathrm{E}_{0,i}\Big[e^{-\underline{\delta}e_{\gamma}}\big(X_{e_{\gamma}}-\underline{X}_{e_{\gamma}}\wedge0\big)+\phi\int_0^{e_{\gamma}}e^{-\underline{\delta}t}\mathrm{d}\big(\underline{X}_t\wedge0\big)\Big]
\nonumber\\
\hspace{-0.3cm}&&\hspace{-0.3cm}
+\mathrm{E}_{0,i}\Big[\phi\int_0^{\infty}e^{-\underline{\delta}(t+e_{\gamma})}\mathrm{d}\big(\underline{X}_{t+e_{\gamma}}\wedge0\big)\Big]
+x\mathrm{E}\Big[e^{-\underline{\delta}e_{\gamma}}\Big]
\nonumber\\
\hspace{-0.3cm}&\leq&\hspace{-0.3cm}
V(x,i).\nonumber
\end{eqnarray}
It is not hard to verify that both the upper and lower bound is bounded under the norm $\|\cdot\|$, which can yield $V(x,i)\in\mathcal{B}$. The proof is complete.
\end{proof}

\begin{lem}\label{v.tv}
For $\mathbf{b}\in[0,\infty)^{\mathcal{E}}$ and $(x,i)\in\mathbb{R}\times\mathcal{E}$, we have $V_{0,\mathbf{b}}(x,i)=\mathcal{T}_{\mathbf{b}}V_{0,\mathbf{b}}(x,i)$.
\end{lem}
\begin{proof}
When $Y_0=i$, let $e_{\lambda_i}$ be the first time $Y$ switches its states. By the strong Markov property and the independence between $(X^i)_{i\in\mathcal{E}}$, $Y$ and $(J_{ij})_{i,j\in\mathcal{E}}$, we obtain
\begin{eqnarray}\label{V.0b.F}
V_{0,\mathbf{b}}(x,i)
\hspace{-0.3cm}&=&\hspace{-0.3cm}
\mathrm{E}_{x,i}\bigg[\sum_{n=1}^{\infty}e^{-\delta_{i}T_n}\Delta D^{b_{i},i}_{T_n}\mathbf{1}_{\{T_n\leq e_{\lambda_i}\}}
-\phi\int_{0}^{e_{\lambda_{i}}}e^{-\delta_{i}t}\mathrm{d}R_{t}^{b_{i},i}+e^{-\delta_{i}e_{\lambda_{i}}}V_{0,\mathbf{b}}(U_{e_{\lambda_{i}}}^{b_{i},i}+J_{iY_{e_{\lambda_{i}}}},Y_{e_{\lambda_{i}}})\bigg]
\nonumber\\
\hspace{-0.3cm}&=&\hspace{-0.3cm}
\mathrm{E}_x^i\bigg[\sum_{n=1}^{\infty}e^{-q_{i}T_n}\Delta D^{b_{i},i}_{T_n}
-\phi\int_{0}^{\infty}e^{-q_{i}t}\mathrm{d}R_{t}^{b_{i},i}\bigg]+\sum_{j\neq i}\lambda_{ij}\mathrm{E}_x^i\Big[\int_0^{\infty}e^{-q_{i}t}V_{0,\mathbf{b}}(U_{t}^{b_{i},i}+J_{ij},j)\mathrm{d}t\Big]
\nonumber\\
\hspace{-0.3cm}&=&\hspace{-0.3cm}
-\frac{\gamma}{q_i+\gamma}\left[\overline{Z}_{q_i,i}(b_i-x)+\frac{\psi_i^{\prime}(0+)}{q_i}\right]+\left[Z_{q_i,i}(b_i-x,\Phi_{q_i+\gamma})+\frac{\gamma}{q_i}Z_{q_i,i}(b_i-x)\right]
\nonumber\\
\hspace{-0.3cm}&&\hspace{-0.3cm}
\times\frac{\left(\gamma Z_{q_i,i}(b_i)-\phi(q_i+\gamma)\right)}{(q_i+\gamma)\Phi_{q_i+\gamma}Z_{q_i,i}(b_i,\Phi_{q_i+\gamma})}
+\sum_{j\neq i}\frac{\lambda_{ij}}{q_i}\int_{-\infty}^0\bigg[\int_{0+}^{\infty}V_{0,\mathbf{b}}(y+z,j)\mathrm{P}_{x}\Big(U_{e_{q_i}}^{0,b_i}\in\mathrm{d}y\Big)
\nonumber\\
\hspace{-0.3cm}&&\hspace{-0.3cm}
+V_{0,\mathbf{b}}(z,j)\mathrm{P}_{x}\Big(U_{e_{q_i}}^{0,b_i}=0\Big)\bigg]\mathrm{d}F_{ij}(z),\quad x\in[0,\infty),\,\,q_i=\delta_i+\lambda_i,
\nonumber\\
V_{0,\mathbf{b}}(x,i)
\hspace{-0.3cm}&=&\hspace{-0.3cm}
\phi x+V_{0,\mathbf{b}}(0,i),\quad x\in(-\infty,0).
\end{eqnarray}
 Using the expression of $V_{0,b_i}^0$, the boundedness of $V_{0,\mathbf{b}}$ under norm $\|\cdot\|$ in Lemma \ref{v.in.B} as well as the fact that $\max_{j\neq i}\mathrm{E}|J_{ij}|<\infty$, we can deduce that $V_{0,\mathbf{b}}\in\mathcal{B}$.
By the defination of $\mathcal{T}_{\mathbf{b}}$ in (\ref{def.Tb}), the second equality in
(\ref{V.0b.F}),
 the independence between $U_t^{b_i,i}$ and $J_{ij}$ for all $i,j\in\mathcal{E}$, and the fact that
$$V_{0,\mathbf{b}}(U_t^{b_i,i}+J_{ij},j)=V_{0,\mathbf{b}}(U_t^{b_i,i}+J_{ij},j)\mathbf{1}_{\{U_t^{b_i,i}\geq-J_{ij}\}}+\Big(V_{0,\mathbf{b}}(0,j)+\phi(U_t^{b_i,i}+J_{ij})\Big)\mathbf{1}_{\{U_t^{b_i,i}<-J_{ij}\}},$$
we can conclude that $V_{0,\mathbf{b}}(x,i)=\mathcal{T}_{\mathbf{b}}V_{0,\mathbf{b}}(x,i)$. The proof is complete.
\end{proof}
\begin{lem}
The operator $\mathcal{T}_{\mathbf{b}}$ is a contraction on $\mathcal{B}$ under the metric $\rho(\cdot,\cdot)$. In particular, for $f\in\mathcal{B}$, we have that \begin{eqnarray}\label{V.b.lim}
V_{0,\mathbf{b}}(x,i)=\lim_{n\rightarrow\infty}\mathcal{T}_{\mathbf{b}}^{n}f(x,i),\quad(x,i)\in[0,\infty)\times\mathcal{E},
\end{eqnarray}
where the convergence is under metric $\rho(\cdot,\cdot)$ and $\mathcal{T}_{\mathbf{b}}^n(f):=\mathcal{T}_{\mathbf{b}}(\mathcal{T}_{\mathbf{b}}^{n-1}(f))$ for $n>1$ with $\mathcal{T}_{\mathbf{b}}^1:=\mathcal{T}_{\mathbf{b}}$.
\end{lem}
\begin{proof}
Recall that the metric space $(\mathcal{B},\rho)$ is complete. By Lemma \ref{V.x}, for $f\in\mathcal{B}$, we have
\begin{eqnarray}\label{Tb.n}
\frac{|\mathcal{T}_{\mathbf{b}}f(x,i)|}{1+|x|}
\hspace{-0.3cm}&=&\hspace{-0.3cm}
(1+|x|)^{-1}\left|\mathrm{E}_x^i\bigg[\sum_{n=1}^{\infty}e^{-q_{i}T_n}\Delta D_{T_n}^{b_{i},i}
-\phi\int_{0}^{\infty}e^{-q_{i}t}\mathrm{d}R_{t}^{b_{i},i}+\lambda_{i}\int_{0}^{\infty}e^{-q_{i}t}\widehat{f}(U_{t}^{b_{i},i},i)\mathrm{d}t\bigg]\right|
\nonumber\\
\hspace{-0.3cm}&=&\hspace{-0.3cm}
(1+|x|)^{-1}\left|\frac{-\gamma}{q_i+\gamma}\left[\overline{Z}_{q_i,i}(b_i-x)+\frac{\psi_i^{\prime}(0+)}{q_i}\right]+\frac{\lambda_i}{q_i}\int_0^{\infty}\hat{f}(y,i)\mathrm{P}_x\big(U_{e_{q_i}}^{b_i,i}\in\mathrm{d}y\big)\right.
\nonumber\\
\hspace{-0.3cm}&&\hspace{-0.3cm}
\left.+
\frac{\left(\gamma Z_{q_i,i}(b_i)-\phi(q_i+\gamma)\right)\left[Z_{q_i,i}(b_i-x,\Phi_{q_i+\gamma})+\frac{\gamma}{q_i}Z_{q_i,i}(b_i-x)\right]}{(q_i+\gamma)\Phi_{q_i+\gamma}Z_{q_i,i}(b_i,\Phi_{q_i+\gamma})}\right|,\quad x\in(0,\infty).
\end{eqnarray}
By (\ref{Tb.n}), Lemma \ref{lem.w}, and the fact that $\widehat{f}\in\mathcal{B}$, it holds that $\mathcal{T}_{\mathbf{b}}f\in\mathcal{B}$. Furthermore, for $f,g\in\mathcal{B}$, we can get that
\begin{eqnarray}\label{rho.fg}
\rho(\mathcal{T}_{\mathbf{b}}f,\mathcal{T}_{\mathbf{b}}g)
\hspace{-0.3cm}&=&\hspace{-0.3cm}
\max_{i\in\mathcal{E}}\sup_{x\geq0}\mathrm{E}_x^i\bigg[e^{-\delta_ie_{\lambda_i}}\sum_{j\in\mathcal{E},j\neq i}\frac{\lambda_{ij}}{\lambda_i}\int_{-\infty}^{-U_{e_{\lambda_i}}^{b_i,i}}\frac{|f(0,j)-g(0,j)|}{1+|x|}\mathrm{d}F_{ij}(y)
\nonumber\\
\hspace{-0.3cm}&&\hspace{-0.3cm}
+e^{-\delta_ie_{\lambda_i}}\sum_{j\in\mathcal{E},j\neq i}\frac{\lambda_{ij}}{\lambda_i}\int_{-U_{e_{\lambda_i}}^{b_i,i}}^0\frac{|f(U_{e_{\lambda_i}}^{b_i,i}+y,j)-g(U_{e_{\lambda_i}}^{b_i,i}+y,j)|}{1+|x|}\mathrm{d}F_{ij}(y)\bigg]
\nonumber\\
\hspace{-0.3cm}&\leq&\hspace{-0.3cm}
\rho(f,g)\sup_{i\in\mathcal{E}}\mathrm{E}_0^i[e^{-\delta_ie_{\lambda_i}}]
\nonumber\\
\hspace{-0.3cm}&:=&\hspace{-0.3cm}
\beta\rho(f,g),\quad\beta\in(0,1).
\end{eqnarray}
By (\ref{rho.fg}), for $f\in\mathcal{B}$, $(\mathcal{T}^n_{\mathbf{b}}f)_{n\geq1}$ is a Cauchy sequence. Therefore, we have
$$\mathcal{T}^{\infty}_{\mathbf{b}}f:=\lim_{n\uparrow\infty}\mathcal{T}^n_{\mathbf{b}}f=\mathcal{T}_{\mathbf{b}}(\lim_{n\uparrow\infty}\mathcal{T}^n_{\mathbf{b}}f)=\mathcal{T}_{\mathbf{b}}(\mathcal{T}^{\infty}_{\mathbf{b}}f),\quad f\in\mathcal{B},$$
which implies that $\mathcal{T}^{\infty}_{\mathbf{b}}f$ is a fixed point of the mapping $\mathcal{T}_{\mathbf{b}}$. By Lemma \ref{v.tv}, we obtain the desired result. This completes the proof.
\end{proof}

Let us define another space of function that
\begin{eqnarray}
\mathcal{C}:=\{f\in\mathcal{B}|\widehat{f}(x,i)\text{ is concave and }\widehat{f^{\prime}}(0,i)\leq\phi\text{ and }\widehat{f^{\prime}}(\infty,i)\in[0,1]\text{ for }i\in\mathcal{E}\}.\nonumber
\end{eqnarray}

\vspace{-0.3cm}
\begin{lem}\label{lem4.4}
Suppose that $f\in\mathcal{B}\cap C^{1}(\mathbb{R}_{+})$ is concave, non-decreasing, and satisfies $f^{\prime}(\cdot,i)\leq \phi$ and $f^{\prime}(\infty,i)\in[0,1]$ for all $i\in\mathcal{E}$, we have that $f\in \mathcal{C}$.
\end{lem}
\begin{proof}
By definition, $\widehat{f}$ can be rewritten as
\begin{eqnarray}
\widehat{f}(x,i)
\hspace{-0.3cm}&=&\hspace{-0.3cm}
\sum_{j\in\mathcal{E}, j\neq  i}\frac{\lambda_{ij}}{\lambda_{i}}\bigg[\int_{-x}^{0}\big[f(x+y,j)-(\phi(x+y)+f(0,j))\big]\mathrm{d}F_{ij}(y)+\phi(x+\mathrm{E}[J_{ij}])+f(0,j)\bigg],
\nonumber
\end{eqnarray}
which implies that
\begin{eqnarray}
\label{f.hat.par}
\widehat{f}^{\prime}(x,i)=\sum_{j\in\mathcal{E}, j\neq  i}\frac{\lambda_{ij}}{\lambda_{i}}\bigg[\phi+\int_{-x}^{0}\big[f^{\prime}(x+y,j)-\phi\big]\mathrm{d}F_{ij}(y)\bigg].
\end{eqnarray}
Combining the concavity of $f$ and (\ref{f.hat.par}), one can get $\widehat{f}(x,i)$ is also concave. Furthermore, by the fact that $f^{\prime}(x+y,j)-\phi\leq0$, we can deduce that $$\widehat{f}^{\prime}(x,i)\leq\sum_{j\in\mathcal{E}, j\neq  i}\frac{\lambda_{ij}}{\lambda_{i}}\phi=\phi.$$
On the other hand, by the fact that $f^{\prime}(\infty,i)\in[0,1]$, we have
\begin{eqnarray}
0=\sum_{j\in\mathcal{E}, j\neq  i}\frac{\lambda_{ij}}{\lambda_{i}}\bigg[\phi-\int_{-\infty}^{0}\phi\mathrm{d}F_{ij}(y)\bigg]\leq\widehat{f}^{\prime}(\infty,i)\leq\sum_{j\in\mathcal{E}, j\neq  i}\frac{\lambda_{ij}}{\lambda_{i}}\Big[\phi+\int_{-\infty}^{0}\big(1-\phi\big)\mathrm{d}F_{ij}(y)\Big]=1.\nonumber
\end{eqnarray}
Then, $f\in\mathcal{C}$. The proof is complete.
\end{proof}

For $f\in \mathcal{C}$ and $(x,i)\in \mathbb{R}_{+} \times \mathcal{E}$, let us define another operator $\mathcal{T}_{\sup}$ that
\begin{eqnarray}\label{T.sup}
\mathcal{T}_{\sup} f(x,i)
\hspace{-0.3cm}&:=&\hspace{-0.3cm}
\sup_{D,R}\mathrm{E}_{x,i}\bigg[\sum_{n=1}^{\infty}e^{-\delta_{i}T_n}\Delta D_{T_n}\mathbf{1}_{\{T_n\leq e_{\lambda_i}\}}
-\phi\int_{0}^{e_{\lambda_{i}}}e^{-\delta_{i}t}\mathrm{d}R_t
+e^{-\delta_{i}e_{\lambda_{i}}}\widehat{f}(U_{e_{\lambda_{i}}-},i)\bigg]
\nonumber\\
\hspace{-0.3cm}&=&\hspace{-0.3cm}
\sup_{D^i,R^i}\mathrm{E}_x^i\bigg[\sum_{n=1}^{\infty}e^{-q_{i}T_n}\Delta D^{i}_{T_n}-\phi\int_{0}^{\infty}e^{-q_{i}t}\mathrm{d}R^{i}_t+\lambda_{i}\int_{0}^{\infty}e^{-q_{i}t}\widehat{f}(U^{i}_t,i)\mathrm{d}t\bigg],
\end{eqnarray}
where $U^{i}_{t}=X^{i}_{t}-D^{i}_{t}+R^{i}_{t}$ represents the controlled surplus process with control $(D^{i},R^{i})$ and driving process $X^{i}$.

Denote $\underline{V}_0:=\underline{V}$ and $\overline{V}_0:=\overline{V}$ as well as $\underline{V}_n:=\mathcal{T}_{\sup}(\underline{V}_{n-1})$ and $\overline{V}_n:=\mathcal{T}_{\sup}(\overline{V}_{n-1})$, for $n\geq1$.
\begin{lem}\label{lem.V.n}
We have $\underline{V}_n\leq V\leq \overline{V}_n$ on $\mathbb{R}_+\times\mathcal{E}$ for all $n\geq1$, and
\begin{eqnarray}\label{V.lim}
V(x,i)=\lim_{n\uparrow\infty}\underline{V}_n(x,i)=\lim_{n\uparrow\infty}\overline{V}_n(x,i),\quad(x,i)\in\mathbb{R}_+\times\mathcal{E},
\end{eqnarray}
where the convergence is under the metric $\rho(\cdot,\cdot)$. Moreover, we have $V\in\mathcal{C}$.
\end{lem}
\begin{proof}
One can verify the first claim of Lemma \ref{lem.V.n}  by the method of induction. In fact, by Lemma \ref{v.in.B} and Lemma \ref{lem4.4}, we have $\underline{V}_0\leq V\leq\overline{V}_0$ and $\underline{V}_0,\overline{V}_0\in\mathcal{C}$. Suppose that $\underline{V}_{n-1}\leq V\leq\overline{V}_{n-1}$, then$$\underline{V}_n=\mathcal{T}_{\sup}(\underline{V}_{n-1})\leq\mathcal{T}_{\sup}({V})\leq\mathcal{T}_{\sup}(\overline{V}_{n-1})=\overline{V}_n,$$which, together with the fact that $V$ is a fixed point of the mapping $\mathcal{T}_{\sup}$, implies that $\underline{V}_{n}\leq V\leq\overline{V}_{n}$ for all $n\geq1$.

To prove the second claim of Lemma \ref{lem.V.n}, for any $f\in\mathcal{C}$ and $i\in\mathcal{E}$, Theorem \ref{them.3.1} guarantees the existence of $b^{f}_i\in(0,\infty)$ such that the second equality of (\ref{T.sup}) is achieved by the expected NPV under a periodic-classical barrier strategy with upper barrier $b^f_i$ and lower barrier 0. Denote $\mathbf{b}^f=(b^f_i)_{i\in\mathcal{E}}$, it follows that $\mathcal{T}_{\sup}f=\mathcal{T}_{\mathbf{b}^f}f$ over $\mathbb{R}_+\times\mathcal{E}$. Furthermore, by Lemma \ref{lem.par.b} and Lemma \ref{lem.b.w}, one can get that $\mathcal{T}_{\sup}f(\cdot,i)\in C^1(0,\infty)$ and it is concave as well as $(\mathcal{T}_{\sup}f)^{\prime}(0,i)\leq\phi$ and $(\mathcal{T}_{\sup}f)^{\prime}(\infty,i)\in[0,1]$ for $i\in\mathcal{E}$, which together with Lemma \ref{lem4.4} yields $\mathcal{T}_{\sup}f\in\mathcal{C}$ and $$\rho(\mathcal{T}_{sup}f,\mathcal{T}_{\sup}g)=\rho(\mathcal{T}_{\mathbf{b}^f}f,\mathcal{T}_{\mathbf{b}^g}g)=\rho(\sup_{\mathbf{b}}\mathcal{T}_{\mathbf{b}}f,\sup_{\mathbf{b}}\mathcal{T}_{\mathbf{b}}g)\leq\sup_{\mathbf{b}}\rho(\mathcal{T}_{\mathbf{b}}f,\mathcal{T}_{\mathbf{b}}g)\leq\beta\rho(f,g),\quad\beta\in(0,1),$$i.e.,$\mathcal{T}_{\sup}$ is a contraction mapping from $\mathcal{C}$ to itself. Hence, the Cauchy sequences $(\underline{V}_n)_{n\geq1}$ and $(\overline{V}_n)_{n\geq1}$ converge to the unique fixed point $V$ of $\mathcal{T}_{\sup}$. In addition, by (\ref{V.lim}) and the dominated convergence theorem, we have that
\begin{eqnarray}\label{V.hat}
\widehat{V}(x,i)
\hspace{-0.3cm}&=&\hspace{-0.3cm}
\sum_{j\in\mathcal{E}, j\neq  i}\frac{\lambda_{ij}}{\lambda_{i}}\int_{-\infty}^{0}\Big[V(x+y,j)\mathbf{1}_{\{-y\leq x\}}
+(\phi(x+y)+V(0,j))\mathbf{1}_{\{-y>x\}}\Big]\mathrm{d}F_{ij}(y)
\nonumber\\
\hspace{-0.3cm}&=&\hspace{-0.3cm}
\lim\limits_{n\rightarrow\infty}\sum_{j\in\mathcal{E}, j\neq  i}\frac{\lambda_{ij}}{\lambda_{i}}\int_{-\infty}^{0}\Big[\underline{V}_n(x+y,j)\mathbf{1}_{\{-y\leq x\}}
+(\phi(x+y)+\underline{V}_n(0,j))\mathbf{1}_{\{-y>x\}}\Big]\mathrm{d}F_{ij}(y)
\nonumber\\
\hspace{-0.3cm}&=&\hspace{-0.3cm}
\lim\limits_{n\rightarrow\infty}\sum_{j\in\mathcal{E}, j\neq  i}\frac{\lambda_{ij}}{\lambda_{i}}\int_{-\infty}^{0}\Big[\overline{V}_n(x+y,j)\mathbf{1}_{\{-y\leq x\}}
+(\phi(x+y)+\overline{V}_n(0,j))\mathbf{1}_{\{-y>x\}}\Big]\mathrm{d}F_{ij}(y)
\nonumber\\
\hspace{-0.3cm}&=&\hspace{-0.3cm}
\lim\limits_{n\rightarrow\infty}\widehat{\underline{V}}_{n}(x,i)=\lim\limits_{n\rightarrow\infty}\widehat{\overline{V}}_{n}(x,i),\quad (x,i)\in \mathbb{R}_{+}\times\mathcal{E}.
\end{eqnarray}
By (\ref{V.hat}) and the fact that $(\underline{V}_n)_{n\geq1}\subseteq\mathcal{C}$ and $(\overline{V}_n)_{n\geq1}\subseteq\mathcal{C}$, we derive that $V\in\mathcal{C}$.
\end{proof}

Finally, we give the proof of Theorem \ref{thm2.1} using the previous preparations.

\begin{proof}[Proof of Theorem 2.1]
By Lemma \ref{lem.V.n}, we have $V\in\mathcal{C}$, which together with Theorem \ref{them.3.1} yields that there exists a function $\mathbf{b}^V=(b_i^V)_{i\in\mathcal{E}}\in(0,\infty)^{\mathcal{E}}$ such that $V(x,i)=\mathcal{T}_{\sup}V(x,i)=\mathcal{T}_{\mathbf{b}^V}(x,i)$ for all $(x,i)\in\mathbb{R}_+\times\mathcal{E}$. Hence, by (\ref{V.b.lim}) and $V\in\mathcal{B}$ (as $V\in\mathcal{C}$), we have that $$V(x,i)=\lim_{n\uparrow\infty}\mathcal{T}_{\mathbf{b}^V}^nV(x,i)=V_{0,\mathbf{b}^V}(x,i),$$i.e., $\mathbf{b}^*:=\mathbf{b}^V=(b^V_i)_{i\in\mathcal{E}}$ is the desired periodic barrier function such that the conclusion of Theorem \ref{thm2.1} holds. The proof is then complete.
\end{proof}

\end{document}